\documentclass[11pt,a4paper,reqno]{amsart}

\usepackage[truedimen,margin=25truemm]{geometry} 
\usepackage{amsmath,amssymb,amsthm, amscd}
\usepackage{bm, time, extarrows, mathtools}

\usepackage{color}

\allowdisplaybreaks
\numberwithin{equation}{section}

\newcommand{\Fg}{\mathfrak{g}}
\newcommand{\Fh}{\mathfrak{h}}
\newcommand{\Fsl}{\mathfrak{sl}}

\newcommand{\BZ}{\mathbb{Z}}

\newcommand{\BR}{\mathbb{R}}
\newcommand{\BC}{\mathbb{C}}
\newcommand{\BK}{\mathbb{K}}

\newcommand{\CO}{\mathcal{O}}

\newcommand{\RAP}{\mathsf{AP}_{\mathrm{red}}}
\newcommand{\RAPc}{\mathsf{AP}_{\mathrm{red}}^{\circ}}
\newcommand{\AP}{\mathsf{AP}}
\newcommand{\APc}{\mathsf{AP}^{\ast}}

\newcommand{\vpi}{\varpi}

\newcommand{\be}{\mathbf{e}}
\newcommand{\bv}{\mathbf{v}}
\newcommand{\bw}{\mathbf{w}}
\newcommand{\bb}{\mathbf{b}}
\newcommand{\bc}{\mathbf{c}}

\newcommand{\bF}{\mathbf{F}}
\newcommand{\bG}{\mathbf{G}}
\newcommand{\bI}{\mathbf{I}}

\newcommand{\minu}{\mathbf{m}}
\newcommand{\QW}{\mathbf{QW}}

\newcommand{\Hom}{\mathrm{Hom}}
\newcommand{\QBG}{\mathrm{QBG}}

\newcommand{\KNOS}{\mathrm{I}}

\DeclareMathOperator{\wt}{wt}
\DeclareMathOperator{\ed}{end}
\DeclareMathOperator{\qwt}{qwt}
\DeclareMathOperator{\gch}{gch}
\DeclareMathOperator{\sgn}{sgn}
\DeclareMathOperator{\YB}{\mathsf{YB}}
\DeclareMathOperator{\D}{\mathsf{D}}

\newcommand{\af}{\mathrm{af}}

\newcommand{\edge}[1]{ \xrightarrow{\hspace{2pt}#1\hspace{2pt}} }
\newcommand{\Qe}[1]{ \xrightarrow[\mathsf{Q}]{\hspace{2pt}#1\hspace{2pt}} }
\newcommand{\Be}[1]{ \xrightarrow[\mathsf{B}]{\hspace{2pt}#1\hspace{2pt}} }

\newcommand{\mcr}[1]{\lfloor #1 \rfloor}
\newcommand{\lng}{w_{\circ}}

\newcommand{\OQG}[1]{ [\mathcal{O}_{\mathbf{Q}_{G}(#1)}] }
\newcommand{\OQGL}[2]{ [\mathcal{O}_{\mathbf{Q}_{G}(#1)}(#2)] }
\newcommand{\QG}{\mathbf{Q}_{G}}
\newcommand{\QGr}{\mathbf{Q}_{G}^{\mathrm{rat}}}
\newcommand{\Kgrp}{K_{ \ti{\bI} }(\QGr)}
\newcommand{\KQG}{K_{H \times \BC^{\ast}}(\QG)}
\newcommand{\KQGr}{K_{H \times \BC^{\ast}}(\QGr)}

\newcommand{\pair}[2]{\langle #1,\,#2 \rangle}

\newcommand{\bra}[1]{[\![#1]\!]}
\newcommand{\pra}[1]{(\!(#1)\!)}

\newcommand{\ti}[1]{\widetilde{#1}}
\newcommand{\ha}[1]{\widehat{#1}}

\theoremstyle{plain}
\newtheorem{lem}{Lemma}[section]
\newtheorem{prop}[lem]{Proposition}
\newtheorem{thm}[lem]{Theorem}
\newtheorem{cor}[lem]{Corollary}

\theoremstyle{definition}
\newtheorem{dfn}[lem]{Definition}

\newtheorem{case}{Case}
\newtheorem{subcase}{Subcase}[case]

\theoremstyle{remark}

\newtheorem{rem}[lem]{Remark}

\newenvironment{enu}{%
 \begin{enumerate}%
}{\end{enumerate}}

\begin{document}

%
\begin{abstract}
We continue the study, begun in \cite{KNOS}, of inverse Chevalley formulas 
for the equivariant $K$-group of semi-infinite flag manifolds. 
Using the language of alcove paths, we reformulate and extend 
our combinatorial inverse Chevalley formula to arbitrary weights 
in all simply-laced types (conjecturally also for $E_8$).
\end{abstract}

\title[Inverse $K$-Chevalley formulas II]{Inverse $K$-Chevalley formulas for semi-infinite flag \\
manifolds, II: arbitrary weights in ADE type}
\author[C. Lenart]{Cristian Lenart}
\address[Cristian Lenart]{Department of Mathematics and Statistics, State University of New York at Albany, 
Albany, NY 12222, U.S.A.}
\email{clenart@albany.edu}
\author[S. Naito]{Satoshi Naito}
\address[Satoshi Naito]{Department of Mathematics, Tokyo Institute of Technology, 
2-12-1 Oh-okayama, Meguro-ku, Tokyo 152-8551, Japan.}
\email{naito@math.titech.ac.jp}
\author[D. Orr]{Daniel Orr}
\address[Daniel Orr]{Department of Mathematics (MC 0123), 460 McBryde Hall, 
Virginia Tech, 225 Stanger St., Blacksburg, VA 24061, U.S.A.}
\email{dorr@vt.edu}
\author[D. Sagaki]{Daisuke Sagaki}
\address[Daisuke Sagaki]{Department of Mathematics, 
Faculty of Pure and Applied Sciences, University of Tsukuba, 
1-1-1 Tennodai, Tsukuba, Ibaraki 305-8571, Japan.}
\email{sagaki@math.tsukuba.ac.jp}
%

\maketitle

\section{Introduction.}
\label{sec:intro}

Let $G$ be a simply-connected simple algebraic group over $\BC$ 
with Borel subgroup $B = HN$, maximal torus $H$, unipotent radical $N$, 
Weyl group $W$, weight lattice $P = \sum_{i \in I} \BZ \vpi_{i}$, 
root lattice $Q = \sum_{i \in I} \BZ \alpha_{i}$, and 
coroot lattice $Q^\vee = \sum_{i \in I} \BZ \alpha_{i}^{\vee}$. 
The work \cite{KNOS} initiated the study of inverse Chevalley formulas 
in the equivariant $K$-group $\KQGr$ of the semi-infinite flag manifold $\QGr$ associated with $G$, 
where the semi-infinite flag manifold $\QGr$ is a reduced ind-scheme whose set of $\BC$-valued points is $G(\BC\pra{z})/(H(\BC) \cdot N(\BC\pra{z}))$ (see \cite{Kat2} for details), 
with the group $\BC^*$ acting on $\QGr$ by loop rotation; 
note that our $K$-group $\KQGr$ is a variant of the Iwahori-equivariant $K$-group of $\QGr$ introduced in \cite{KNS}.
The $K$-group $\KQGr$ has a topological $K_{H \times \BC^*}(\mathrm{pt})$-basis consisting of 
Schubert classes $\{\OQG{x}\}_{x\in W_\af}$ indexed by the affine Weyl group $W_\af=W\ltimes Q^\vee$, 
where $K_{H \times \BC^*}(\mathrm{pt}) \cong \BZ[q^{\pm 1}][P]$, 
with $K_{H}(\mathrm{pt}) = R(H) \cong \BZ[P] = \BZ[\be^{\lambda} \mid \lambda \in P]$ 
the character ring of $H$ and $K_{\BC^{*}}(\mathrm{pt}) = R(\BC^{*}) \cong \BZ[q^{\pm 1}]$; 
here $q \in R(\BC^{*})$ denotes the character of loop rotation. 
By an \textit{inverse} Chevalley formula, we mean a combinatorial formula 
for the product of an equivariant scalar $\be^\lambda\in K_H(\mathrm{pt}) = R(H)$ 
with a Schubert class $\OQG{x}$, expressed as a $\BZ[q^{\pm 1}]$-linear combination of 
the twisted Schubert classes $\{\OQGL{x}{\mu}\}_{x \in W_\af,\mu \in P}$; 
here the twisted Schubert class $\OQGL{x}{\mu}$ corresponds to 
the tensor product sheaf $\OQG{x} \otimes \CO(\mu)$, 
with $\CO(\mu)$ the equivariant line bundle over $\QGr$ associated to $\mu \in P$. 
The results of \cite{KNOS} treat the case when $\lambda$ is
a (not necessarily dominant) minuscule weight and $G$ is of simply-laced type.

This work is a sequel to \cite{KNOS}. Our main result is a combinatorial formula 
which generalizes the inverse Chevalley formula of \cite{KNOS} to 
arbitrary weights $\lambda \in P$. An important feature of our formula is 
that we formulate it using \textit{alcove paths}, 
matching more closely existing results of \cite{LP1} 
on equivariant Chevalley formulas for finite-dimensional flag manifolds 
and those of \cite{LNS} regarding (non-inverse) Chevalley formulas in $\KQGr$. 
(We recall that, while the ordinary and inverse Chevalley formulas for 
equivariant $K$-theory of finite-dimensional flag manifolds are essentially identical, 
the two types of Chevalley formulas are genuinely different 
for semi-infinite flag manifolds; see \cite[Introduction]{KNOS}.)

The ingredients in our combinatorial formula are roughly as follows 
(see \S\ref{subsec:main} for details). Given any weight $\lambda\in P$, 
any alcove path $\Gamma$ from the fundamental alcove $A_\circ$ to 
$A_\lambda=A_\circ+\lambda$, and any $w\in W$, 
we define a finite set $\ti{\QW}_{\lambda,w}(\Gamma)$ of 
combinatorial objects $(\bw,\bb)$ (called decorated quantum walks) and 
various associated quantities: $(-1)^{(\bw,\bb)} \in \{\pm 1\}$, 
$\deg(\bw,\bb)\in\BZ$, $\ed(\bw)\in W$, $\wt(\bw)\in P$, 
$\qwt(\bw,\bb)\in Q$ and $\qwt^\vee(\bw,\bb)\in Q^\vee$. 
Informally, an element $(\bw,\bb)\in\ti{\QW}_{\lambda,w}(\Gamma)$ 
consists of a \textit{walk} $\bw$ in the quantum Bruhat graph $\QBG(W)$ 
(i.e., a directed path in $\QBG(W)$ with stationary steps allowed), 
which must begin at $w\in W$ and follow edges prescribed by $\Gamma$, 
together with some additional information recorded by $\bb$ at special stationary steps in $\bw$.

Our main result then reads as follows:

\begin{thm}[{= Theorem~\ref{thm:main}}] \label{ithm:main}
Assume that $G$ is simply-laced, but not of type $E_{8}$. 
Let $\lambda \in P$ be an arbitrary weight, and 
let $\Gamma$ be an arbitrary alcove path from $A_\circ$ to $A_\lambda$.
For any $w \in W$, the following equality holds in $\KQGr${\rm :}
\begin{equation} \label{eq:imain}
\begin{split}
& \be^{\lambda} \cdot \OQG{w} = \\[3mm]
& 
\sum_{(\bw,\bb) \in \ti{\QW}_{\lambda,w}(\Gamma)} (-1)^{(\bw,\bb)} q^{\deg(\bw,\bb)} 
\OQGL{ \ed(\bw) t_{ \qwt^{\vee}(\bw,\bb) } }{-\lng \wt(\bw) -\lng \qwt(\bw,\bb)},
\end{split}
\end{equation}
where $\lng$ denotes the longest element of $W$.
\end{thm}

Here we mention that the sum on the right-hand side of \eqref{eq:imain} is 
an explicit finite sum, described in terms of decorated quantum walks, 
while the finiteness itself is proved in \cite{Orr}.
Also, we note that it suffices to consider only Schubert classes indexed by $x=w\in W$, 
rather than arbitrary $x\in W_\af$, by virtue of the right action of 
translations $\{t_\xi\}_{\xi\in Q^\vee}\subset W_\af$ on $\QGr$.

We expect that \eqref{eq:imain} also holds, as stated, in type $E_8$. 
To prove this, however, further arguments will be necessary to handle 
the case of quasi-minuscule $\lambda$; we plan to provide 
the details of these arguments in a future work.

The first step in our proof of Theorem~\ref{ithm:main} is 
to reformulate the formula from \cite{KNOS}, for minuscule $\lambda$, 
in terms of a particular (reduced) alcove path $\Gamma$ (see Proposition~\ref{prop:minu}). 
Then, our main efforts are devoted to establishing that 
the right-hand side of \eqref{eq:imain}:
\begin{enumerate}
\item is invariant under Yang-Baxter transformations (Theorem~\ref{thm:YB}) and 
deletion procedures (Theorem~\ref{thm:D}) on alcove paths, and
\item respects additivity in $\lambda \in P$ (Proposition~\ref{prop:sum}).
\end{enumerate}
In order to prove part (1), we establish the existence of 
a certain ``sijection'' (bijection between signed sets in the sense of \cite{FK}) 
between decorated quantum walks associated to 
two alcove paths $\Gamma_{1}, \Gamma_{2}$ from $A_{\circ}$ to 
$A_{\lambda}$ for $\lambda \in P$ such that $\Gamma_{2}$ is 
obtained from $\Gamma_{1}$ by a Yang-Baxter transformation or a deletion procedure.
In \cite{KLN}, it is shown that similar results hold for 
a certain generating function of some statistics including weights over 
admissible subsets in the quantum alcove model, which is closely related to 
the (non-inverse) Chevalley formula in $\KQGr$.
However, since decorated quantum walks are completely different combinatorial objects 
from admissible subsets, we need to construct the desired sijection for decorated quantum walks from scratch.
In addition, our results above are much more difficult to prove than 
the corresponding results for admissible subsets in \cite{KLN} 
mainly because of the appearance of the rather delicate term $q^{\deg(\,\cdot\,)}$ 
on the right-hand side of \eqref{eq:imain}.
By a result of \cite{S} (see also \cite{LP2}), 
which asserts that an arbitrary $\lambda \in P$ can be written as a sum of 
(not necessarily dominant) minuscule weights in simply-laced types (except in type $E_{8}$), 
we are then able to deduce that \eqref{eq:imain} holds for all $\lambda \in P$ and 
all alcove paths $\Gamma$ from $A_\circ$ to $A_\lambda$. 

Here we should mention that in \cite{Kat1}, Kato established a $\BZ[P]$-module isomorphism 
from the (small) $H$-equivariant quantum $K$-theory $QK_{H}(G/B)$ 
(see \cite{ACT} for the finiteness result on $QK_{H}(G/B)$) of 
the finite-dimensional flag manifold $G/B$ onto the $\BZ[P]$-submodule 
(denoted by $K_{H}(\mathbf{Q}_{G}(e))$ in this paper) of 
the specialization of $\KQGr$ at $q = 1$ consisting of 
all finite linear combinations of 
the Schubert classes $\OQG{x}$ for $x \in W_{\af}^{\geq 0} := W \times Q^{\vee,+}$ 
with coefficients in $\BZ[P]$, where 
$Q^{\vee,+} := \sum_{i \in I} \BZ_{\geq 0} \alpha_{i}^{\vee} \subset 
Q^{\vee} = \sum_{i \in I} \BZ \alpha_{i}^{\vee}$. 
This isomorphism sends each (opposite) Schubert class in $QK_{H}(G/B)$ 
to the corresponding Schubert class in $K_{H}(\mathbf{Q}_{G}(e))$; 
also, it respects the quantum multiplication with the line bundle class 
$[\CO_{G/B}(- \vpi_{i})]$ and the tensor product with 
the line bundle class $\OQGL{e}{\lng \vpi_{i}}$ for all $i \in I$.
Since the quantum multiplication in $QK_{H}(G/B)$ is uniquely determined 
by its $\BZ[P]$-module structure and the quantum multiplication 
with the line bundle classes $[\CO_{G/B}(- \vpi_{i})]$ 
for $i \in I$ (see \cite{BCMP}), it follows that under 
Kato's $\BZ[P]$-module isomorphism above, 
the quantum multiplicative structure of $QK_{H}(G/B)$ 
can be described explicitly by means of (the specialization at $q = 1$ of) 
our inverse Chevalley formula (Theorem~\ref{ithm:main}), 
together with (the specialization at $q = 1$ of) the (non-inverse) Chevalley formula in $\KQGr$ 
for anti-dominant weights $- \varpi_{i}$, obtained in \cite{NOS}; 
recall that the (non-inverse) Chevalley formula in $\KQGr$ for anti-dominant weights 
expresses an arbitrary twisted Schubert class $\OQGL{x}{-\varpi_{i}}$ 
as an explicit finite linear combination of Schubert classes 
with coefficients in $\BZ[q^{\pm 1}][P]$ in terms of the quantum alcove model.

In Appendix~\ref{sec:example}, we give a detailed example in type $A_{2}$ of 
our inverse Chevalley formula \eqref{eq:imain}.

\section*{Acknowledgements.}
The authors would like to thank Takafumi Kouno for helpful discussions.
C.L. was partially supported by the NSF grant DMS-1855592 and the Simons Foundation grant No. 584738.
S.N. was supported in part by JSPS Grant-in-Aid for Scientific Research (B) 16H03920
D.O. was supported by a Collaboration Grant for Mathematicians from the Simons Foundation (Award ID: 638577).
D.S. was supported in part by JSPS Grant-in-Aid for Scientific Research (C) 19K03415.

\section{Basic notation.}
\label{sec:pre}

\subsection{Root systems.}
\label{sec:rs}

As above, let $G$ be a simply-connected simple algebraic group over $\BC$. 
We fix a maximal torus and Borel subgroup: $H \subset B \subset G$. 
We set $\Fg := \mathrm{Lie}(G)$ and $\Fh := \mathrm{Lie}(H)$, 
and denote by $\pair{\cdot\,}{\cdot} : \Fh^{\ast} \times \Fh \rightarrow \BC$ 
the canonical pairing, where $\Fh^{\ast} := \Hom_{\BC}(\Fh, \BC)$. 

Let $\Delta \subset \Fh^{\ast}$ be 
the root system of $\Fg$, $\Delta^{+} \subset \Delta$ the positive roots (corresponding to $B$), 
and $\{ \alpha_{i} \}_{i \in I} \subset \Delta^{+}$ the set of simple roots; 
we denote by $\alpha^{\vee} \in \Fh$ the coroot corresponding to $\alpha \in \Delta$. 
For $\alpha \in \Delta$, we set $\sgn(\alpha):=1$ (resp., $\sgn(\alpha):=-1$) 
if $\alpha$ is positive (resp., negative), and $|\alpha|:=\sgn(\alpha)\alpha \in \Delta^{+}$. 
We denote by $\theta \in \Delta^{+}$ the highest root of $\Delta$, and 
set $\rho := (1/2) \sum_{\alpha \in \Delta^{+}} \alpha$, 
$Q := \sum_{i \in I} \BZ \alpha_{i}$, and $Q^{\vee} := \sum_{i \in I} \BZ \alpha_{i}^{\vee}$; 
also, we set $Q^{\vee,+}:=\sum_{i \in I} \BZ_{\ge 0}\alpha_{i}^{\vee} \subset Q^{\vee}$. 

For $i \in I$, let $\vpi_{i} \in \Fh^{\ast}$ be the $i$-th fundamental weight determined by
$\pair{\vpi_{i}}{\alpha_{j}^{\vee}} = \delta_{i,j}$ for all $j \in I$, 
where $\delta_{i,j}$ denotes the Kronecker delta. 
The weight lattice $P$ of $\Fg$ is defined by $P := \sum_{i \in I} \BZ \vpi_{i}$; 
note that $\Fh^{\ast}_{\BR}=\BR \otimes_{\BZ} P = \BR \otimes_{\BZ} Q$ is a real form of $\Fh^{\ast}$. 
We denote by $\BZ[P]$ the group algebra of $P$, that is, 
the associative algebra generated by the formal exponentials $\be^{\lambda}$, 
$\lambda \in P$, 
where $\be^{\lambda} \be^{\mu} := \be^{\lambda + \mu}$ 
for $\lambda,\,\mu \in P$. 

For $\alpha \in \Delta$, the corresponding reflection $s_{\alpha} \in GL(\Fh^{\ast})$ is defined by $s_{\alpha}(\lambda) := \lambda - \pair{\lambda}{\alpha^{\vee}} \alpha$ 
for $\lambda \in \Fh^{\ast}$; we write $s_{i} := s_{\alpha_{i}}$ for $i \in I$. 
Then the Weyl group $W$ of $\Fg$ is the subgroup $\langle s_{i} \mid i \in I \rangle$ of $GL(\Fh^{\ast})$ generated by $\{ s_{i} \}_{i \in I}$.
We denote by $\ell(w)$ the length of $w \in W$ with respect to $\{ s_{i} \}_{i \in I}$, and by $<$ the Bruhat order on $W$. 
The following fact is well-known.
%
%
\begin{lem} \label{lem:ell}
Let $w \in W$ and $\alpha \in \Delta$. Then, 
\begin{equation*}
\begin{split}
& ws_{\alpha} > w \iff \ell(ws_{\alpha}) > \ell(w) \iff \sgn(w\alpha)=\sgn(\alpha), \\
& ws_{\alpha} < w \iff \ell(ws_{\alpha}) < \ell(w) \iff \sgn(w\alpha)=-\sgn(\alpha).
\end{split}
\end{equation*}
\end{lem}

\begin{dfn}[{cf. \cite[Definition~6.1]{BFP}}]
The quantum Bruhat graph $\QBG(W)$ 
is the $\Delta^{+}$-labeled directed graph whose vertices are the elements of $W$ 
and whose (directed) edges are of the following form: 
$x \xrightarrow{\alpha} y$, with $x, y \in W$ and $\alpha \in \Delta^+$, 
such that $y = x s_{\alpha}$ and either of the following holds: 
(B) $\ell(y) = \ell(x) + 1$, (Q) $\ell(y) = \ell(x) - 2\pair{\rho}{\alpha^\vee} + 1$. 
An edge satisfying (B) (resp. (Q)) is called a Bruhat edge (resp. a quantum edge).
\end{dfn}

Now, let $\Fg_{\af} := (\Fg \otimes \BC[z,\,z^{-1}]) \oplus \BC c \oplus \BC d$ 
be the (untwisted) affine Lie algebra over $\BC$ associated to $\Fg$, 
where $c$ is the canonical central element and $d$ is the scaling element,
with Cartan subalgebra $\Fh_{\af} := \Fh \oplus \BC c \oplus \BC d$. 
We denote by $\pair{\cdot\,}{\cdot} : \Fh_{\af}^{\ast} \times \Fh_{\af} \rightarrow \BC$ 
the canonical pairing. Regarding $\lambda \in \Fh^{\ast}$ as 
$\lambda \in \Fh_{\af}^{\ast} = \Hom_{\BC}(\Fh_{\af},\,\BC)$ 
by setting $\pair{\lambda}{c} = \pair{\lambda}{d} = 0$, 
we have $\Fh^{\ast} \subset \Fh_{\af}^{\ast}$; 
under this identification, we see that the canonical pairing $\pair{\cdot\,}{\cdot}$ 
on $\Fh_{\af}^{\ast} \times \Fh_{\af}$ extends that on $\Fh^{\ast} \times \Fh$. 

We define $\delta$ to be the unique element of $\Fh_{\af}^{\ast}$ which satisfies
$\pair{\delta}{h} = 0$ for all $h \in \Fh$, $\pair{\delta}{c} = 0$, and $\pair{\delta}{d} = 1$, 
and set $\alpha_{0} := -\theta + \delta \in \Fh_{\af}^{\ast}$. 
Then the root system $\Delta_{\af}$ of $\Fg_{\af}$ has simple roots 
$\{ \alpha_{i} \}_{i \in I_{\af}}$, where $I_{\af} := I \sqcup \{ 0 \}$.

For $\alpha \in \Delta_{\af}$, we have the corresponding reflection $s_{\alpha} \in GL(\Fh_{\af})$, 
defined as for $\Fg$. Note that for $\alpha \in \Delta \subset \Delta_{\af}$, 
the restriction of the reflection $s_{\alpha}$ defined on $\Fh_{\af}$ to $\Fh$ 
coincides with the reflection $s_{\alpha}$ defined on $\Fh$. 
We set $s_{i} := s_{\alpha_{i}}$ for $i \in I_{\af}$. 
Then, the Weyl group of $\Fg_{\af}$ (called the affine Weyl group) $W_{\af}$ is defined 
to be the subgroup of $GL(\Fh_{\af})$ generated by $\{ s_{i} \}_{i \in I_{\af}}$, 
namely, $W_{\af} = \langle s_{i} \mid i \in I_{\af} \rangle$. 
In \cite[\S 6.5]{Kac}, it is shown that $W_{\af} \simeq 
W \ltimes \{ t_{\alpha^{\vee}} \mid \alpha^{\vee} \in Q^{\vee} \} \simeq W \ltimes Q^{\vee}$, 
where $t_{\alpha^{\vee}}$ is the translation element corresponding to $\alpha^{\vee} \in Q^{\vee}$;
we set $W_{\af}^{\ge 0}:=\{ wt_{\alpha^{\vee}} \mid w \in W,\,\alpha^{\vee} \in Q^{\vee,+} \} \simeq
W \times Q^{\vee,+} \subset W_{\af}$. 

\subsection{Alcove paths.}
\label{subsec:alcove}

For $\alpha \in \Delta$ and $k \in \BZ$, we set 
\begin{equation}
H_{\alpha,k}:=\bigl\{\nu \in \Fh^{\ast}_{\BR} \mid \pair{\nu}{\alpha^{\vee}}=k\bigr\}. 
\end{equation}
Let $\ha{r}_{\alpha,k}$ denote the affine reflection 
with respect to the affine hyperplane $H_{\alpha,k}$; we have 
\begin{equation}
\ha{r}_{\alpha,k}(\nu) = \nu - (\pair{\nu}{\alpha^{\vee}}-k)\alpha = s_{\alpha}\nu + k\alpha \quad 
 \text{for $\nu \in \Fh^{\ast}_{\BR}$}.
\end{equation}
The hyperplanes $H_{\alpha,k}$, $\alpha \in \Delta,\,k \in \BZ$, 
divide the real vector space $\Fh^{\ast}_{\BR}$ into open regions, called alcoves; 
the fundamental alcove is defined as
\begin{equation*}
A_{\circ} := \bigl\{ \nu \in \Fh^{\ast}_{\BR} \mid 
 0 < \pair{\nu}{\alpha^{\vee}} < 1 \ \text{for all $\alpha \in \Delta^{+}$} \bigr\}. 
\end{equation*}
We say that two alcoves are adjacent if they are distinct and have a common wall. Given a pair
of adjacent alcoves $A$ and $B$, we write $A \xrightarrow{\alpha} B$ for $\alpha \in \Delta$ 
if the common wall is orthogonal to $\alpha$, and $\alpha$ points in the direction from $A$ to $B$.

\begin{dfn}[\cite{LP1}]
	An alcove path is a sequence of alcoves $(A_{0},\,A_{1},\,\ldots,\,A_{m})$ such that
	$A_{j-1}$ and $A_j$ are adjacent for $j=1,\,2,\ldots,\,m$. 
	We say that an alcove path $(A_{0},\,A_{1},\,\ldots,\,A_{m})$ is reduced 
	if it has minimal length among all alcove paths from $A_{0}$ to $A_{m}$.
\end{dfn}

Let $\lambda \in P$, and let $\Gamma=(A_{0},\,A_{1},\,\ldots,\,A_{m})$ 
be an alcove path from the fundamental alcove $A_{\circ}$ to $A_{\lambda}:=A_{\circ}+\lambda$. 
If $\gamma_{1},\gamma_{2},\dots,\gamma_{m} \in \Delta$ are such that
%
%
\begin{equation} \label{eq:Gamma}
  A_{\circ}=A_{0} 
  \xrightarrow{\gamma_{1}} A_{1} 
  \xrightarrow{\gamma_{2}} \cdots 
  \xrightarrow{\gamma_{m}} A_{m}=A_{\lambda} \ (=A_{\circ}+\lambda), 
\end{equation}
then we write $\Gamma=(\gamma_{1},\dots,\gamma_{m})$. 
For each $1 \le t \le m$, let $H_{\gamma_{t},l_{t}}$ be the affine hyperplane 
between $A_{t-1}$ and $A_{t}$, and set
\begin{equation} \label{eq:l'}
l_{t}':=\pair{\lambda}{\gamma_{t}^{\vee}} - l_{t}.
\end{equation}
Let $\AP(\lambda)$ (resp., $\RAP(\lambda)$) denote 
the set of all alcove paths (resp., all reduced alcove paths) 
from $A_{\circ}$ to $A_{\lambda}$. 

Let $\lambda, \mu \in P$, and 
let $\Gamma=(\gamma_{1},\dots,\gamma_{m}) \in \AP(\lambda)$, 
$\Xi=(\xi_{1},\dots,\xi_{p}) \in \AP(\mu)$. It is obvious that 
the concatenation 
\begin{equation} \label{eq:cat}
\Gamma \ast \Xi:=(\gamma_{1},\dots,\gamma_{m},\xi_{1},\dots,\xi_{p})
\end{equation}
of $\Gamma$ with $\Xi$ is an alcove path 
from $A_{\circ}$ to $A_{\lambda+\mu}$; namely, $\Gamma \ast \Xi \in \AP(\lambda+\mu)$. 

\begin{rem} \label{rem:cat}
Keep the notation and setting above. 
For each $1 \le q \le p$, let $H_{\xi_{q},k_{q}}$ be the affine hyperplane between 
the $(q-1)$-th alcove and the $q$-th alcove in $\Xi$. Then the affine hyperplane between 
the $(t-1)$-th alcove and the $t$-th alcove in $\Gamma \ast \Xi$ is 
$H_{\gamma_{t},l_{t}}$ (resp., $H_{\xi_{t-m},\pair{\lambda}{\xi_{t-m}^{\vee}} + k_{t-m}}$) 
for $1 \le t \le m$ (resp., $m+1 \le t \le m+p$). 
\end{rem}

\subsection{Simply-laced assumption.}
\label{subsec:simplylaced}

In this paper, we assume throughout that $G$ is {\em simply-laced}. 
As a result, by means of the non-degenerate $W$-invariant symmetric bilinear form 
$(\cdot\,,\cdot) : \Fh^{\ast} \times \Fh^{\ast} \rightarrow \BC$, 
normalized so that $(\alpha,\alpha)=2$ for all $\alpha \in \Delta$, 
we can identify roots with coroots; note that $\pair{\nu}{\alpha^{\vee}}=(\nu,\alpha)$ 
for $\nu \in \Fh^{\ast}$ and $\alpha \in \Delta$. 
Under this identification, 
if $\alpha \in \Delta$ is of the form $\alpha=\sum_{i \in I} c_{i}\alpha_{i}$, 
with $c_{i} \in \BZ$ for $i \in I$, then we can write 
$\alpha^{\vee}=\sum_{i \in I} c_{i}\alpha_{i}^{\vee}$. 


\section{Main results.}
\label{sec:main}

\subsection{$K$-groups.}
\label{subsec:Kgrp}

Let $\Kgrp$ denote the equivariant $K$-group of the semi-infinite flag manifold $\QGr$ introduced in \cite[Section~6]{KNS}, 
where $\ti{\bI} = \bI \rtimes \BC^{\ast}$ is the semi-direct product of 
the Iwahori subgroup $\bI \subset G(\BC\bra{z})$ and loop rotation $\BC^{\ast}$. 
Correspondingly, $\Kgrp$ is a module over $\BZ[P]\pra{q^{-1}}$, 
which acts by equivariant scalar multiplication. 

One has the following classes in $\Kgrp$, 
for each $x\in W_{\af}$ and $\lambda \in P$:
\begin{itemize}
\item Schubert classes $\OQG{x}$, 
\item equivariant line bundle classes $[\CO(\lambda)]$,
\item classes $\OQGL{x}{\lambda}$ corresponding to 
the tensor product sheaves $\CO_{\mathbf{Q}_{G}(x)} \otimes \CO(\lambda)$.
\end{itemize}
We follow the conventions of \cite{KNS} for indexing equivariant line bundles and 
Schubert varieties in $\QGr$. 

\begin{dfn}[$(H \times \BC^{\ast})$-equivariant $K$-groups of $\QGr$ and $\QG$]
Let $\KQGr$ denote the $\BZ[q^{\pm 1}][P]$-submodule of $\Kgrp$ 
consisting of all convergent infinite $\BZ[q^{\pm 1}][P]$-linear combinations of 
Schubert classes $\OQG{x}$ for $x \in W_{\af}$, 
where convergence holds in the sense of \cite[Proposition 5.8]{KNS}. 

Similarly, we define $\KQG$ to be the $\BZ[q^{\pm 1}][P]$-submodule of 
$\Kgrp$ consisting of all convergent infinite $\BZ[q^{\pm 1}][P]$-linear combinations of 
Schubert classes $\OQG{x}$ for $x \in W_{\af}^{\ge 0}$. 
\end{dfn}

The classes $\{ \OQG{x} \}_{x \in W_{\af}}$ satisfy a notion of 
topological linear independence in $\Kgrp$ 
given by \cite[Proposition~5.8]{KNS}; thus they form a topological $\BZ[q^{\pm 1}][P]$-basis of 
$\KQGr$. Also, one has $\OQGL{x}{\lambda} \in \KQGr$ 
for any $x \in W_{\af}$ and $\lambda \in P$, thanks to the Chevalley formulas for dominant weights \cite{KNS} and anti-dominant weights \cite{NOS}. 
Similar (in fact, equivalent) assertions hold for $\KQG$.

\begin{dfn}
Define $\BK \subset \KQGr$ to be the $\BZ[q^{\pm 1}]$-submodule consisting of 
all finite $\BZ[q^{\pm 1}]$-linear combinations of 
the classes $\{ \OQGL{x}{\lambda} \}_{x \in W_{\af},\,\lambda\in P}$.
\end{dfn}

By definition, $\BK$ is only a $\BZ[q^{\pm 1}]$-submodule of $\KQGr$. 
But, as shown in \cite[Theorem~5.1]{Orr}, $\BK$ is indeed a $\BZ[q^{\pm 1}][P]$-submodule of $\KQGr$; 
here we recall the identification $K_{H \times \BC^{*}}(\mathrm{pt}) \simeq \BZ[q^{\pm 1}][P]$.
We note that the classes $\{ \OQGL{x}{\lambda} \}_{x \in W_{\af},\,\lambda \in P}$ are 
linearly independent over $\BZ[q^{\pm 1}]$, again by the Chevalley formula of \cite{KNS}.
 
To summarize, we have the following chain of $\BZ[q^{\pm 1}][P]$-modules 
(and $\Kgrp$ is in fact a $\BZ[P]\pra{q^{-1}}$-module):
\begin{equation*}
\BK \subset \KQGr \subset \Kgrp.
\end{equation*}

\subsection{Inverse $K$-Chevalley formula for minuscule weights.}
\label{subsec:KNOS}

In this subsection, we review the inverse Chevalley formula for minuscule weights, 
obtained in \cite[Theorem~3.14]{KNOS}. Assume that $\Fg$ is simply-laced, but not of type $E_{8}$; 
in this subsection, we identify $\sum_{i \in I}c_{i}\alpha_{i} \in Q$ 
with $\sum_{i \in I}c_{i}\alpha_{i}^{\vee} \in Q^{\vee}$, 
as mentioned in Section~\ref{subsec:simplylaced}. 
Let $\lambda \in P$ be a (not necessarily dominant) minuscule weight, with $\vpi_{k}$ 
the unique minuscule fundamental weight contained in $W\lambda$. 
Let $x \in W$ be the unique minimal-length element of $W$ such that 
$\lambda=x\vpi_{k}$, and let $y \in W$ be the (unique) element such that 
$yx$ is the unique minimal-length element in 
$\bigl\{w \in W \mid w\vpi_{k}=\lng\vpi_{k}\bigr\}$; 
it is easy to verify that $\ell(yx) = \ell(y) + \ell(x)$.
We fix reduced expressions $x=s_{j_a}\dotsm s_{j_1}$ and 
$y = s_{i_{1}} \cdots s_{i_{b}}$, and define
\begin{align*}
\beta_{r} & := 
 s_{j_{a}}s_{j_{a-1}} \cdots s_{j_{r+1}} \alpha_{j_{r}} \in \Delta^{+} \qquad 
 \text{for $1 \le r \le a$}, \\
\gamma_s & := 
 s_{i_{b}}s_{i_{b-1}} \cdots s_{i_{s+1}} \alpha_{i_{s}} \in \Delta^{+} \qquad 
 \text{for $1 \le s \le b$}.
\end{align*}
We set 
\begin{equation} \label{eq:veceta}
\vec{\eta}:=(\eta_1,\dotsc,\eta_m)=(\beta_a,\,\ldots,\,\beta_1,\,\gamma_1,\,\ldots,\,\gamma_b),
\end{equation}
where $m = a + b$. For $w \in W$, let $\QW_{\lambda,w}^{\KNOS}$ denote 
the set of sequences $(w_{0},w_{1},\dots,w_{m})$ such that 
\begin{itemize}
\item $w_{0}=w$; 
\item $w_{t} \in \bigl\{ w_{t-1},\,s_{\eta_{t}}w_{t-1} \bigr\}$ for all $1 \le t \le m$; 
\item for $1 \le t \le m$ such that $w_{t} = s_{\gamma_{t}}w_{t-1}$, 
$w_{t-1} \to w_{t} = s_{\gamma_{t}}w_{t-1}$ is an edge (labeled by $|w_{t-1}^{-1} \gamma_{t}| = |w_{t}^{-1} \gamma_{t}|$) in the quantum Bruhat graph $\QBG(W)$. 
\end{itemize}
Given $\bw=(w_{0},w_1,\dots,w_m) \in \QW_{\lambda,w}^{\KNOS}$, let 
$S^{-}(\bw)^{\KNOS}$ denote the set of steps $t$, for $1\le t\le a$, such that $w_{t}=w_{t-1}$ and $(\rho,w_{t-1}^{-1}\eta_t)=1$ (or equivalently, $w_{t-1}^{-1}\eta_t$ is a simple root). 
Similarly, let $S^{+}(\bw)^{\KNOS}$ denote the set of steps $t$, 
for $a < t \le m$, such that $w_{t}=w_{t-1}$ and $(\rho,w_{t-1}^{-1}\eta_t)=-1$ 
(or equivalently, $-w_{t-1}\eta_t$ is a simple root). 
Let $S(\bw)^{\KNOS}=S^{-}(\bw)^{\KNOS} \cup S^{+}(\bw)^{\KNOS}$, and 
define $\ti{\QW}^{\KNOS}_{\lambda,w}$ to consist of all pairs $(\bw,\bb)$ 
where $\bw \in \QW^{\KNOS}_{\lambda,w}$ and $\bb$ is a $\{0,1\}$-valued function on $S(\bw)^{\KNOS}$.
For $(\bw,\bb)\in \ti{\QW}^{\KNOS}_{\lambda,w}$, we define
\begin{align*}
(-1)^{(\bw,\bb)} & := 
 \prod_{ \substack{ 1 \le t\le a \\ w_t < w_{t-1}} } (-1) 
 \prod_{ \substack{ a < t \le m \\ w_t>w_{t-1}} }(-1)
 \prod_{ t \in S(\bw)^{\KNOS} }(-1)^{\bb(t)}, 
\end{align*}
and 
\begin{align*}
\wt_0(\bw,\bb)^{\KNOS} & := 0, \\[2mm]
\wt_t(\bw,\bb)^{\KNOS} & := \wt_{t-1}(\bw,\bb)^{\KNOS} +
\begin{cases}
-\bb(t)\lng w_{t-1}^{-1}\eta_t & \text{if $t\in S^{-}(\bw)^{\KNOS}$}, \\
\lng w_{t-1}^{-1}\eta_t & \text{if $w_t < w_{t-1}$}, \\
0 & \text{otherwise}, 
\end{cases}
\qquad \text{for $1\le t\le a$} \\[2mm]
\wt_t(\bw,\bb)^{\KNOS} & := \wt_{t-1}(\bw,\bb)^{\KNOS} +
\begin{cases}
\bb(t)\lng w_{t-1}^{-1}\eta_t & \text{if $t\in S^{+}(\bw)^{\KNOS}$}, \\
\lng w_{t-1}^{-1}\eta_t & \text{if $w_t < w_{t-1}$}, \\
0 & \text{otherwise}, 
\end{cases}
\qquad \text{for $a<t\le m$}; 
\end{align*}
define $\wt(\bw,\bb)^{\KNOS}:=\wt_{n}(\bw,\bb)^{\KNOS}$, and 
set $d_t(\bw,\bb):=\wt_{t}(\bw,\bb)^{\KNOS}-\wt_{t-1}(\bw,\bb)^{\KNOS}$ for $1\le t \le m$.
Then we define
\begin{align*}
\deg_{0}^{-}(\bw,\bb)^{\KNOS} & = 0, \\
\deg_{t}^{-}(\bw,\bb)^{\KNOS} & = 
\deg_{t-1}^{-}(\bw,\bb)^{\KNOS} + \frac{(d_t(\bw,\bb),d_t(\bw,\bb))}{2}+
(d_t(\bw,\bb),\wt_{t-1}(\bw,\bb)^{\KNOS}) \quad \text{for $1 \le t \le a$} \\
\deg_{a}^{+}(\bw,\bb)^{\KNOS} & = \deg_{a}^{-}(\bw,\bb)^{\KNOS}+
(-\lng w_{a}^{-1}\lambda,\,\wt_{a}(\bw,\bb)^{\KNOS}) \\
\deg_{t}^{+}(\bw,\bb)^{\KNOS} &= 
\deg_{t-1}^{+}(\bw,\bb)^{\KNOS}+\frac{(d_t(\bw,\bb),d_t(\bw,\bb))}{2}+
(d_t(\bw,\bb),\wt_{t-1}(\bw,\bb)^{\KNOS}) \qquad \text{for $a < t \le m$};
\end{align*}
define $\deg(\bw,\bb)^{\KNOS}=\deg_{m}^{+}(\bw,\bb)^{\KNOS} \in \BZ$.
%
%
\begin{thm}[{\cite[Theorem~3.14]{KNOS}}] \label{thm:KNOS}
For any minuscule $\lambda\in P$ and $w\in W$, we have in $\BK \subset \KQGr$, 
\begin{equation} \label{eq:KNOS}
\begin{split}
& \be^{\lambda} \cdot \OQG{w} = \\
& \quad \sum_{(\bw,\bb) \in \ti{\QW}^{\KNOS}_{\lambda,w}}
  (-1)^{(\bw,\bb)} q^{\deg(\bw,\bb)^{\KNOS}} \cdot 
  \OQGL{w_{m} t_{-\lng \wt(\bw,\bb)^{\KNOS}} }{ -\lng w_{a}^{-1}\lambda+\wt(\bw,\bb)^{\KNOS} }. 
\end{split}
\end{equation}
\end{thm}

\subsection{Inverse $K$-Chevalley formula for arbitrary weights.}
\label{subsec:main}

Let $\lambda \in P$, and 
\begin{equation*}
\Gamma: 
  A_{\circ}=A_{0} 
  \xrightarrow{\gamma_{1}} A_{1} 
  \xrightarrow{\gamma_{2}} \cdots 
  \xrightarrow{\gamma_{m}} A_{m}=A_{\lambda}
\end{equation*}
be an alcove path from the fundamental alcove $A_{\circ}$ to $A_{\lambda}$; 
for $1 \le t \le m$, let $H_{\gamma_{t},l_{t}}$ be the affine hyperplane 
between $A_{t-1}$ and $A_{t}$. 
For $w \in W$, let $\QW_{\lambda,w} = \QW_{\lambda,w}(\Gamma)$ denote 
the set of sequences $(w_{0},w_{1},\dots,w_{m})$ such that 
\begin{itemize}
\item $w_{0}=w$; 
\item $w_{t} \in \bigl\{ w_{t-1},\,s_{\gamma_{t}}w_{t-1} \bigr\}$ for all $1 \le t \le m$; 
\item for $1 \le t \le m$ such that $w_{t} = s_{\gamma_{t}}w_{t-1}$, 
$w_{t-1} \to w_{t} = s_{\gamma_{t}} w_{t-1}$ is an edge in the quantum Bruhat graph $\QBG(W)$; 
note that the label of this edge is $|w_{t-1}^{-1}\gamma_{t}| = |w_{t}^{-1}\gamma_{t}|$. 
\end{itemize}
For $\bw = (w_{0},w_{1},\dots,w_{m}) \in \QW_{\lambda,w}$, we define
\begin{equation} \label{eq:end}
\ed(\bw):=w_{m}, 
\end{equation}
\begin{equation} \label{eq:Tw}
T(\bw):=\bigl\{ 1 \le t \le m \mid w_{t} = s_{\gamma_{t}} w_{t-1} \bigr\}, 
\end{equation}
\begin{equation} \label{eq:Tw-}
T^{-}(\bw):=\bigl\{ 1 \le t \le m \mid w_{t-1} > w_{t} \bigr\} \subset T(\bw), 
\end{equation}
\begin{equation} \label{eq:Sw}
S(\bw):=\bigl\{ 1 \le t \le m \mid 
 \text{$w_{t}=w_{t-1}$ and $-w_{t}^{-1}\gamma_{t}$ is a simple root} \bigr\}. 
\end{equation}
We set
\begin{equation} \label{eq:tiQW}
\ti{\QW}_{\lambda,w}=\ti{\QW}_{\lambda,w}(\Gamma):=
\bigl\{ (\bw,\bb) \mid 
\text{$\bw \in \QW_{\lambda,w}(\Gamma)$ and $\bb:S(\bw) \rightarrow \{0,1\}$} \bigr\}. 
\end{equation}
For $(\bw,\bb) \in \ti{\QW}_{\lambda,w}$, we define
\begin{equation}
\ed(\bw,\bb):=\ed(\bw),
\end{equation}
\begin{align}
(-1)^{(\bw,\bb)} & :=
  \prod_{ \begin{subarray}{c} 1 \le t \le m \\[1mm] \gamma_{t} \in \Delta^{-} \\[1mm] w_{t-1} > w_{t} \end{subarray}} (-1) 
  \prod_{ \begin{subarray}{c} 1 \le t \le m \\[1mm] \gamma_{t} \in \Delta^{+} \\[1mm] w_{t-1} < w_{t} \end{subarray}} (-1) 
  \prod_{t \in S(\bw)} (-1)^{\bb(t)} \nonumber \\[3mm]
& = 
  \prod_{ \begin{subarray}{c} t \in T(\bw) \\ w_{t-1}^{-1}\gamma_{t} \in \Delta^{+} \end{subarray}} (-1) 
  \prod_{t \in S(\bw)} (-1)^{\bb(t)}
  = 
  \prod_{ \begin{subarray}{c} t \in T(\bw) \\ w_{t}^{-1}\gamma_{t} \in \Delta^{-} \end{subarray}} (-1) 
  \prod_{t \in S(\bw)} (-1)^{\bb(t)}, 
\end{align}
\begin{align}
&
\begin{cases}
\displaystyle{%
\qwt_{t}(\bw,\bb):=\sum_{
 \begin{subarray}{c}
 1 \le u \le t \\[1mm]
 u \in T^{-}(\bw)
 \end{subarray}} |w_{u}^{-1}\gamma_{u}| + 
\sum_{
 \begin{subarray}{c}
 1 \le u \le t \\[1mm]
 u \in S(\bw)
 \end{subarray}} (- \bb(u) w_{u}^{-1}\gamma_{u})} & \text{for $0 \le t \le m$}, \\[10mm]
\qwt(\bw,\bb):=\qwt_{m}(\bw,\bb), 
\end{cases} \\[5mm]
& 
\begin{cases}
\displaystyle{%
\qwt_{t}^{\vee}(\bw,\bb):=\sum_{
 \begin{subarray}{c}
 1 \le u \le t \\[1mm]
 u \in T^{-}(\bw) 
 \end{subarray}} |w_{u}^{-1}\gamma_{u}|^{\vee} + 
\sum_{
 \begin{subarray}{c}
 1 \le u \le t \\[1mm]
 u \in S(\bw)
 \end{subarray}} (- \bb(u) w_{u}^{-1}\gamma_{u}^{\vee})} & \text{for $0 \le t \le m$}, \\[10mm]
\qwt^{\vee}(\bw,\bb):=\qwt_{m}^{\vee}(\bw,\bb), 
\end{cases}
\end{align}
\begin{equation}
\begin{cases}
\ha{r}_{0}(\bw):=w^{-1}, \\[3mm]
\ha{r}_{t}(\bw):=
\begin{cases}
\ha{r}_{t-1}(\bw)
\ha{r}_{\gamma_{t},l_{t}} & \text{\rm if $t \in T(\bw)$}, \\[1mm]
\ha{r}_{t-1}(\bw) & \text{\rm otherwise}, 
\end{cases}
\quad \text{for $1 \le t \le m$},
\end{cases}
\end{equation}
\begin{equation}
\begin{cases}
\wt_{0}(\bw) :=\ha{r}_{0}(\bw)\lambda = w^{-1}\lambda, \\[2mm]
\wt_{t}(\bw) :=\ha{r}_{t}(\bw)\lambda-\ha{r}_{t-1}(\bw)\lambda \quad \text{for $1 \le t \le m$}, \\[2mm]
\wt(\bw) := \ha{r}_{m}(\bw)\lambda = {\displaystyle\sum_{t=0}^{m} \wt_{t}(\bw)}, \qquad \wt(\bw,\bb):=\wt(\bw), 
\end{cases}
\end{equation}
\begin{equation} \label{eq:deg}
\deg(\bw,\bb) := \frac{1}{2} \pair{ \qwt(\bw,\bb) }{ \qwt^{\vee}(\bw,\bb) } + \deg'(\bw,\bb), 
\end{equation}
where
\begin{equation} \label{eq:deg'}
\deg'(\bw,\bb) := \sum_{t=1}^{m} \pair{ \wt_{t}(\bw) }{ \qwt_{t}^{\vee}(\bw,\bb) } -
 \sum_{t \in T^{-}(\bw)} \sgn(\gamma_{t})l_{t}' - 
 \sum_{t \in S(\bw)}\bb(t)l_{t}'. 
\end{equation}
The main result of this paper is the following inverse $K$-Chevalley formula in $\BK \subset \KQGr$, 
which generalizes Theorem~\ref{thm:KNOS} to the case that 
$\lambda \in P$ is an arbitrary weight (see also Section~\ref{sec:minu} below). 
%
%
\begin{thm} \label{thm:main}
Assume that $\Fg$ is simply-laced, but not of type $E_{8}$. 
Let $\lambda \in P$ be an arbitrary weight, and $\Gamma \in \AP(\lambda)$. 
For $w \in W$, the following equality holds in $\BK \subset \KQGr${\rm :}
\begin{equation} \label{eq:main}
\begin{split}
& \be^{\lambda} \cdot \OQG{w} = \\[3mm]
& 
\underbrace{
\sum_{(\bw,\bb) \in \ti{\QW}_{\lambda,w}(\Gamma)} \overbrace{(-1)^{(\bw,\bb)} q^{\deg(\bw,\bb)} 
\OQGL{ \ed(\bw) t_{ \qwt^{\vee}(\bw,\bb) } }{-\lng \wt(\bw) -\lng \qwt(\bw,\bb)}}^{=:\bG(\bw,\bb)}%
 }_{ =:\bF_{\lambda,w}(\Gamma)}. 
\end{split}
\end{equation}
\end{thm}

An example in type $A_{2}$ is given in 
Appendix~\ref{sec:example}.

\begin{rem}
The degree function $\deg$ defined in \eqref{eq:deg} and \eqref{eq:deg'} may seem ad hoc to the reader. In fact, it arises naturally from commutation relations in the $q$-Heisenberg algebra used in \cite{KNOS}.
\end{rem}

\subsection{Outline of the proof of Theorem~\ref{thm:main}.}
\label{subsec:outline}

Keep the setting of Theorem~\ref{thm:main}. 
Using Theorem~\ref{thm:KNOS}, 
we first prove the following. 
%
%
\begin{prop}[to be proved in Section~\ref{sec:minu}] \label{prop:minu}
Assume that $\lambda \in P$ is a minuscule weight. 
Then, there exists $\Gamma \in \RAP(\lambda)$ for which Theorem~\ref{thm:main} holds. 
\end{prop}

For a minuscule weight $\lambda \in P$, 
let $\RAPc(\lambda)$ denote the subset of $\RAP(\lambda)$ 
consisting of those elements for which Theorem~\ref{thm:main} holds. 

Next we prove the following. 
%
%
\begin{prop}[to be proved in Section~\ref{sec:sum}] \label{prop:sum}
Let $\lambda,\,\mu \in P$, and $\Gamma \in \AP(\lambda)$, $\Xi \in \AP(\mu)$. 
Assume that both of the equalities $\be^{\lambda} \cdot \OQG{w} = \bF_{\lambda,w}(\Gamma)$ and 
$\be^{\mu} \cdot \OQG{w} = \bF_{\mu,w}(\Xi)$ hold in $\BK \subset \KQGr$ for all $w \in W$. 
Then we have $\be^{\lambda+\mu} \cdot \OQG{w} = \bF_{\lambda+\mu,w}(\Gamma \ast \Xi)$ in $\BK \subset \KQGr$
for all $w \in W$. 
\end{prop}

Here we know the following fact from \cite{S} (see also \cite[Theorem~2.1]{L}); 
recall that $\Fg$ is simply-laced, but not of type $E_{8}$. 
%
%
\begin{prop} \label{prop:stem}
For each $\lambda \in P$, there exist minuscule weights 
$\nu_{1},\nu_{2},\dots,\nu_{s} \in P$ such that $\lambda=\nu_{1}+\nu_{2}+\cdots+\nu_{s}$. 
\end{prop}

For $\lambda \in P$, we define $\minu(\lambda)$ to be the set of all finite 
sequences $(\nu_{1},\nu_{2},\dots,\nu_{s})$, $s \ge 0$, of 
minuscule weights in $P$ such that $\lambda=\nu_{1}+\nu_{2}+\cdots+\nu_{s}$, and set
\begin{equation}
\APc(\lambda):=\bigcup_{
(\nu_{1},\nu_{2},\dots,\nu_{s}) \in \minu(\lambda)}
\bigl\{ \Gamma_{1} \ast \cdots \ast \Gamma_{s} \mid 
\Gamma_{u} \in \RAPc(\nu_{u}),\,1 \le u \le s \bigr\};
\end{equation}
notice that $\APc(\lambda) \ne \emptyset$ by Propositions~\ref{prop:minu} and 
\ref{prop:stem}. Combining Propositions~\ref{prop:minu} and \ref{prop:sum}, 
we obtain the following. 
%
%
\begin{cor} \label{cor:AP0}
Theorem~\ref{thm:main} holds for arbitrary $\lambda \in P$ and 
$\Gamma \in \AP_{0}(\lambda)$. 
\end{cor}

Also, we know the following fact from the proof of \cite[Lemma~9.3]{LP2}. 
%
%
\begin{prop} \label{prop:YBD}
Let $\lambda \in P$. Let $\Gamma \in \RAP(\lambda)$ and $\Gamma' \in \AP(\lambda)$. 
Then, $\Gamma$ can be obtained from $\Gamma'$ by repeated application of 
the following procedures {\rm (YB)} and {\rm (D):}
\begin{enumerate}
\item[\rm (YB)] for $\alpha,\beta \in \Delta$ 
such that $\pair{\alpha}{\beta^{\vee}} \le 0$, or equivalently, 
$\pair{\beta}{\alpha^{\vee}} \le 0$, one replaces a segment 
$\alpha,\,s_{\alpha} \beta,\,
s_{\alpha}s_{\beta} \alpha,\,\dots,\,
s_{\beta} \alpha,\,\beta$ by
$\beta,\,s_{\beta} \alpha,\,\dots,\,
s_{\alpha}s_{\beta} \alpha,\,
s_{\alpha} \beta,\,\alpha$\,{\rm;}

\item[\rm (D)] one deletes a segment of the form 
$\alpha$, $-\alpha$ for $\alpha \in \Delta$. 
\end{enumerate}
\end{prop}

\begin{rem} \label{rem:YB}
Since $\Fg$ is simply-laced, 
if $\alpha,\beta \in \Delta$ satisfy 
$\pair{\alpha}{\beta^{\vee}} \le 0$, or equivalently 
$\pair{\beta}{\alpha^{\vee}} \le 0$, then either of the following holds: 
\begin{equation*}
\text{(a) } \pair{\alpha}{\beta^{\vee}}= \pair{\beta}{\alpha^{\vee}}=0, \qquad 
\text{(b) } \pair{\alpha}{\beta^{\vee}}= \pair{\beta}{\alpha^{\vee}}=-1.
\end{equation*}
If (a) (resp., (b)) holds, then (YB) is just the following replacement: 
\begin{equation*}
\alpha,\,\beta \mapsto \beta,\,\alpha \qquad 
\text{(resp., $\alpha,\,\alpha+\beta,\,\beta \mapsto \beta,\,\alpha+\beta,\,\alpha$)}. 
\end{equation*}
\end{rem}
%
%
\begin{thm}[to be proved in Section~\ref{sec:YB}] \label{thm:YB}
Let $\lambda \in P$ and $w \in W$. 
Let $\Gamma, \Xi \in \AP(\lambda)$, and 
assume that $\Xi$ is obtained from $\Gamma$ by {\rm (YB)} in Proposition~\ref{prop:YBD}. 
Then, there exist a subset $\ti{\QW}^{(0)}_{\lambda,w}(\Gamma)$ of 
$\ti{\QW}_{\lambda,w}(\Gamma)$ and a subset $\ti{\QW}^{(0)}_{\lambda,w}(\Xi)$ of 
$\ti{\QW}_{\lambda,w}(\Xi)$ such that the following hold {\rm :}
\begin{enu}
\item There exists a bijection 
$\YB:\ti{\QW}^{(0)}_{\lambda,w}(\Gamma) \rightarrow \ti{\QW}^{(0)}_{\lambda,w}(\Xi)$
satisfying the conditions that for $(\bw,\bb) \in \ti{\QW}^{(0)}_{\lambda,w}(\Gamma)$, 
\begin{equation} \label{eq:YB1}
\begin{cases}
(-1)^{\YB(\bw,\bb)}=(-1)^{(\bw,\bb)}, \qquad \ed(\YB(\bw,\bb))=\ed(\bw,\bb), \\[1mm]
  \qwt(\YB(\bw,\bb)) = \qwt(\bw,\bb), \qquad 
  \qwt^{\vee}(\YB(\bw,\bb)) = \qwt^{\vee}(\bw,\bb), \\[1mm]
  \wt(\YB(\bw,\bb)) = \wt(\bw,\bb), \qquad \deg(\YB(\bw,\bb)) = \deg(\bw,\bb). 
\end{cases}
\end{equation}
\item There exists an involution $\YB$ on 
$\ti{\QW}_{\lambda,w}(\Gamma) \setminus \ti{\QW}^{(0)}_{\lambda,w}(\Gamma)$ 
satisfying the conditions that for $(\bw,\bb) \in 
\ti{\QW}_{\lambda,w}(\Gamma) \setminus \ti{\QW}^{(0)}_{\lambda,w}(\Gamma)$, 
\begin{equation} \label{eq:YB2}
\begin{cases}
(-1)^{\YB(\bw,\bb)}=-(-1)^{(\bw,\bb)}, \qquad \ed(\YB(\bw,\bb))=\ed(\bw,\bb), \\[1mm]
  \qwt(\YB(\bw,\bb)) = \qwt(\bw,\bb), \qquad 
  \qwt^{\vee}(\YB(\bw,\bb)) = \qwt^{\vee}(\bw,\bb), \\[1mm]
  \wt(\YB(\bw,\bb)) = \wt(\bw,\bb), \qquad \deg(\YB(\bw,\bb)) = \deg(\bw,\bb). 
\end{cases}
\end{equation}

\item There exists an involution $\YB$ on 
$\ti{\QW}_{\lambda,w}(\Xi) \setminus \ti{\QW}^{(0)}_{\lambda,w}(\Xi)$ 
satisfying the same conditions as in \eqref{eq:YB2}. 
\end{enu}
\end{thm}
%
%
\begin{cor} \label{cor:YB}
Let $\lambda \in P$ and $w \in W$. 
Let $\Gamma, \Xi \in \AP(\lambda)$. 
If $\Xi$ is obtained from $\Gamma$ by {\rm (YB)}, then 
$\bF_{\lambda,w}(\Gamma) = \bF_{\lambda,w}(\Xi)$ holds in $\BK \subset \KQGr$. 
\end{cor}
%
%
\begin{thm}[to be proved in Section~\ref{sec:D}] \label{thm:D}
Let $\lambda \in P$ and $w \in W$. 
Let $\Gamma, \Xi \in \AP(\lambda)$, and 
assume that $\Xi$ is obtained from $\Gamma$ by {\rm (D)} in Proposition~\ref{prop:YBD}. 
Then, there exists a subset $\ti{\QW}^{(1)}_{\lambda,w}(\Gamma)$ of 
$\ti{\QW}_{\lambda,w}(\Gamma)$ such that the following hold {\rm :}
\begin{enu}
\item There exists a bijection 
$\D:\ti{\QW}^{(1)}_{\lambda,w}(\Gamma) \rightarrow \ti{\QW}_{\lambda,w}(\Xi)$
satisfying the conditions that for $(\bw,\bb) \in \ti{\QW}^{(1)}_{\lambda,w}(\Gamma)$, 
\begin{equation} \label{eq:D1}
\begin{cases}
(-1)^{\D(\bw,\bb)}=(-1)^{(\bw,\bb)}, \qquad \ed(\D(\bw,\bb))=\ed(\bw,\bb), \\[1mm]
  \qwt(\D(\bw,\bb)) = \qwt(\bw,\bb), \qquad 
  \qwt^{\vee}(\D(\bw,\bb)) = \qwt^{\vee}(\bw,\bb), \\[1mm]
  \wt(\D(\bw,\bb)) = \wt(\bw,\bb), \qquad \deg(\D(\bw,\bb)) = \deg(\bw,\bb). 
\end{cases}
\end{equation}
\item There exists an involution $\D$ on 
$\ti{\QW}_{\lambda,w}(\Gamma) \setminus \ti{\QW}^{(1)}_{\lambda,w}(\Gamma)$ 
satisfying the conditions that for $(\bw,\bb) \in 
\ti{\QW}_{\lambda,w}(\Gamma) \setminus \ti{\QW}^{(1)}_{\lambda,w}(\Gamma)$, 
\begin{equation} \label{eq:D2}
\begin{cases}
(-1)^{\D(\bw,\bb)}=-(-1)^{(\bw,\bb)}, \qquad \ed(\D(\bw,\bb))=\ed(\bw,\bb), \\[1mm]
  \qwt(\D(\bw,\bb)) = \qwt(\bw,\bb), \qquad 
  \qwt^{\vee}(\D(\bw,\bb)) = \qwt^{\vee}(\bw,\bb), \\[1mm]
  \wt(\D(\bw,\bb)) = \wt(\bw,\bb), \qquad \deg(\D(\bw,\bb)) = \deg(\bw,\bb). 
\end{cases}
\end{equation}
\end{enu}
\end{thm}
%
%
\begin{cor} \label{cor:D}
Let $\lambda \in P$ and $w \in W$. 
Let $\Gamma, \Xi \in \AP(\lambda)$. 
If $\Xi$ is obtained from $\Gamma$ by {\rm (D)}, then 
$\bF_{\lambda,w}(\Gamma) = \bF_{\lambda,w}(\Xi)$ holds in $\BK \subset \KQGr$. 
\end{cor}

Let $\Gamma \in \APc(\lambda)$ and $\Xi \in \RAP(\lambda)$. 
By Proposition~\ref{prop:YBD}, there exists a sequence 
$\Gamma=\Gamma_{0},\Gamma_{1},\dots,\Gamma_{p}=\Xi$ of elements in $\AP(\lambda)$ 
such that $\Gamma_{q}$ is obtained from $\Gamma_{q-1}$ by (YB) or (D) for each $1 \le q \le p$. 
By making use of Corollaries~\ref{cor:YB} and \ref{cor:D} repeatedly, we deduce that 
$\bF_{\lambda,w}(\Gamma) = \bF_{\lambda,w}(\Xi)$ holds in $\BK$. 
Since $\be^{\lambda} \cdot \OQG{w} = \bF_{\lambda,w}(\Gamma)$ in $\BK$ by Corollary~\ref{cor:AP0}, 
we obtain the following. 
%
%
\begin{cor} \label{cor:RAP}
Theorem~\ref{thm:main} holds for arbitrary $\lambda \in P$ and 
$\Gamma \in \RAP(\lambda)$. 
\end{cor}

Finally, let $\Gamma \in \AP(\lambda)$ and $\Xi \in \RAP(\lambda)$. 
By the same argument as above, we deduce that $\bF_{\lambda,w}(\Gamma) = \bF_{\lambda,w}(\Xi)$ holds in $\BK$. 
Since $\be^{\lambda} \cdot \OQG{w} = \bF_{\lambda,w}(\Xi)$ in $\BK$ by Corollary~\ref{cor:RAP}, 
we conclude that Theorem~\ref{thm:main} holds for arbitrary $\lambda \in P$ and 
$\Gamma \in \AP(\lambda)$. 

\subsection{A few technical remarks about the proof.}
\label{subsec:har}

Keep the notation and setting of Section~\ref{subsec:main}. 
Let $\bw \in \QW_{\lambda,w}(\Gamma)$. 
Let $0 \le t \le m$, and let $t_{1} < t_{2} < \cdots < t_{c}$ be the elements of 
$T(\bw)$ (see \eqref{eq:Tw}) less than or equal to $t$. 
Then we have 
\begin{align}
\ha{r}_{t}(\bw)\lambda & = 
w^{-1}\ha{r}_{\gamma_{t_1},l_{t_1}} \cdots \ha{r}_{\gamma_{t_c},l_{t_c}}(\lambda) \nonumber \\
& = w^{-1}\ha{r}_{\gamma_{t_1},l_{t_1}} \cdots \ha{r}_{\gamma_{t_{c-1}},l_{t_{c-1}}}
(s_{\gamma_{t_c}} \lambda +l_{t_c}\gamma_{t_c}) \nonumber \\
& = w^{-1}\ha{r}_{\gamma_{t_1},l_{t_1}} \cdots \ha{r}_{\gamma_{t_{c-2}},l_{t_{c-2}}}
(s_{\gamma_{t_{c-1}}}s_{\gamma_{t_c}} \lambda +
   l_{t_c}s_{\gamma_{t_{c-1}}} \gamma_{t_c} +l_{t_{c-1}}\gamma_{t_{c-1}}) \nonumber \\
& = \cdots \cdots \cdots \cdots \nonumber \\
& = \underbrace{w^{-1}s_{\gamma_{t_1}} \cdots s_{\gamma_{t_{c-1}}}s_{\gamma_{t_c}}}_{%
    =w_{t}^{-1}} \lambda + 
    \sum_{a=1}^{c} l_{t_a} w^{-1}s_{\gamma_{t_1}} \cdots s_{\gamma_{t_{a-1}}} \gamma_{t_a}. \label{eq:har}
\end{align}
Using this formula, we can show that for $1 \le t \le m$, 
\begin{equation} \label{eq:wttw}
\wt_{t}(\bw) = \ha{r}_{t}(\bw)\lambda-\ha{r}_{t-1}(\bw)\lambda = 
\begin{cases}
- l_{t}' w_{t-1}^{-1}\gamma_{t} & \text{\rm if $t \in T(\bw)$}, \\[1mm]
0 & \text{\rm otherwise}. 
\end{cases}
\end{equation}
Indeed, if $t \notin T(\bw)$, then it is obvious that 
$\wt_{t}(\bw) = 0$ since $\ha{r}_{t}(\bw) = \ha{r}_{t-1}(\bw)$. 
Assume that $t \in T(\bw)$; note that in this case, $t_{c}=t$, and 
$w^{-1}s_{\gamma_{t_1}} \cdots s_{\gamma_{t_{c-1}}} \gamma_{t_c} = 
w_{t-1}^{-1}\gamma_{t}$. We see that 
\begin{align*}
\wt_{t}(\bw) & = \ha{r}_{t}(\bw)\lambda-\ha{r}_{t-1}(\bw)\lambda = 
w_{t}^{-1}\lambda - w_{t-1}^{-1}\lambda + l_{t}w_{t-1}^{-1}\gamma_{t} \\
& = w_{t-1}^{-1}s_{\gamma_{t}}\lambda - w_{t-1}^{-1}\lambda + l_{t}w_{t-1}^{-1}\gamma_{t}
 = -\pair{\lambda}{\gamma_{t}^{\vee}}w_{t-1}^{-1}\gamma_{t} + l_{t}w_{t-1}^{-1}\gamma_{t} \\
& = - l_{t}' w_{t-1}^{-1}\gamma_{t}.
\end{align*}
This proves \eqref{eq:wttw}, as desired. 

\section{Proof of Proposition~\ref{prop:minu}.}
\label{sec:minu}

In this section, we assume that $\lambda \in P$ is a (not necessarily dominant) minuscule weight; 
as in Section~\ref{subsec:KNOS}, 
let $\vpi_{k}$ be the unique minuscule fundamental weight contained in $W\lambda$, and 
let $x \in W$ be the unique minimal-length element of $W$ 
such that $\lambda=x\vpi_{k}$. Also, let $y \in W$ be the (unique) element such that 
$yx$ is the unique minimal-length element in 
$\bigl\{w \in W \mid w\vpi_{k}=\lng\vpi_{k}\bigr\}$; 
recall that $\ell(yx) = \ell(y) + \ell(x)$.
Let $x=s_{j_{a}} \cdots s_{j_{1}}$ and $y = s_{i_{1}} \cdots s_{i_{b}}$ 
be reduced expressions for $x$ and $y$, respectively, and set
\begin{align*}
& \beta_{c}:=s_{j_{a}} \cdots s_{j_{c+1}} \alpha_{j_{c}} \in \Delta^{+} \qquad \text{for $1 \le c \le a$}, \\
& \zeta_{d}:=s_{i_{b}} \cdots s_{i_{d+1}} \alpha_{i_{d}} \in \Delta^{+} \qquad \text{for $1 \le d \le b$}. 
\end{align*}
%
%
\begin{lem}[{cf. \eqref{eq:veceta}}] \label{lem:min}
In the notation and setting above, 
%
%
\begin{equation} \label{eq:mGam}
\Gamma=(-\beta_{a},\dots,-\beta_{1},\zeta_{1},\dots,\zeta_{b}) \in \RAP(\lambda). 
\end{equation}
Moreover, for $1 \le c \le a$, the affine hyperplane between the $(c-1)$-th alcove and 
the $c$-th alcove in $\Gamma$ is $H_{-\beta_{a-c+1},0}$, and 
for $1 \le d \le b$, the affine hyperplane between the $(a+d-1)$-th alcove and 
the $(a+d)$-th alcove in $\Gamma$ is $H_{\zeta_{d},1}$. 
\end{lem}

\begin{proof} 
Consider a dominant weight $\mu$, and an element $w$ in the set $W^{\mu}$ of 
minimal(-length) coset representatives modulo the stabilizer of $\mu$, 
denoted $W_{\mu}$ (as a parabolic subgroup). In the extended affine Weyl group, 
we have the following length-additive factorization of the translation element $t_{w \mu}$:
\begin{equation} \label{dec-trans}
t_{w \mu} = w (t_{\mu} w^{-1}).
\end{equation}
Indeed, this follows from the well-known equality $\ell(t_{w \mu}) = \ell(t_{\mu})$ 
(see, for example, \cite[(2.4.1)]{Mac}) and the fact that
\begin{equation*}
\ell(t_{\mu} w^{-1}) = \ell(wt_{-\mu})=\ell(t_{-\mu})-\ell(w)=\ell(t_{\mu}) - \ell(w).
\end{equation*}
The first and last equalities above are obvious, 
while the second one is the straightforward extension to 
the extended affine Weyl group of the corresponding result in \cite[Lemma~3.3]{LS}; 
indeed, the hypothesis of the mentioned lemma is satisfied, 
as $-\mu$ is anti-dominant and $w \in W^{\mu}$.

Now consider the following reduced alcove paths: 
$\Delta$ from $A_{\circ}$ to $w A_{\circ}$, 
and $\Xi$ from $A_{\circ}$ to $w^{-1} A_{\circ} + \mu$. 
The reduced alcove paths $\Delta$ and $w\Xi$ can be 
concatenated (as sequences of alcoves), and 
we obtain in this way the alcove path $\Delta*w\Xi$ in $\AP(w\mu)$. 
In fact, $\Delta*w\Xi$ is in $\RAP(w\mu)$, 
due to the length-additive factorization \eqref{dec-trans}. 

We now specialize $\mu=\varpi_k$ and $w=x$, 
so we obtain $\Delta*x\Xi$ in $\RAP(\lambda)$. 
The reduced alcove path $\Delta$, written as a sequence of roots, 
can be obtained from the reduced decomposition $x=s_{j_a}\ldots s_{j_1}$ as needed; 
see~\cite[Theorem~4.5]{Hum}. 

It remains to analyze the reduced alcove path $x\Xi$. 
Upon translation by $-\lambda=-x\varpi_k$ and reversal, 
we obtain a reduced alcove path from $A_\circ$ to $x(A_\circ-\varpi_k)$. 
We claim that $A_\circ-\varpi_k=\mcr{w_\circ}^{-1} A_\circ$, 
where $\mcr{u} \in W^{\vpi_{k}}$ denotes the minimal(-length) coset representative of 
$uW_{\varpi_k}$. Therefore, we have $x(A_\circ-\vpi_k)=x \mcr{w_\circ}^{-1} A_\circ=y^{-1}A_\circ$. 
We conclude that, as a sequence of roots, the alcove path $x\Xi$ is 
obtained from the reduced decomposition $y^{-1}=s_{i_b}\ldots s_{i_1}$ 
by reversing the sequence of roots given by \cite[Theorem~4.5]{Hum}. 

Finally, we address the claim that $A_\circ-\vpi_k=\mcr{\lng}^{-1} A_\circ$. 
Using the assumption that $\vpi_{k}$ is a minuscule fundamental weight, 
we have $A_{\circ} - \vpi_{k}=v A_{\circ}$ for some $v\in W$. 
Indeed, for any positive root $\beta$, we have precisely $\pair{\vpi_k}{\beta^\vee}$ hyperplanes 
separating $A_\circ$ and $A_{\circ} - \vpi_{k}$, that is, 
either $0$ or $1$; in the second case, the respective hyperplane is $H_{\beta,0}$. 
On another hand, we have
\begin{equation*}
\pair{\vpi_k}{\beta^\vee} \ne 0 \iff 
\beta \in \Delta^{+} \setminus \Delta_{\vpi_k}^{+} \iff 
\beta \in \mathrm{Inv}(\lng w_\circ^{\varpi_k});
\end{equation*}
here, $\Delta_{\vpi_k}^{+}$ denotes the positive roots corresponding to 
the parabolic subgroup $W_{\vpi_k}$, $w_\circ^{\vpi_k}$ denotes 
the longest element of $W_{\vpi_k}$, and $\mathrm{Inv}(u)$ denotes 
the inversion set of $u$, namely $\Delta^{+} \cap u^{-1}(-\Delta^{+})$. 
We conclude that the hyperplanes separating $A_\circ$ from $vA_\circ$ are precisely those 
$H_{\beta,0}$ for which $\beta \in \mathrm{Inv}(\mcr{\lng})$. 
But it is well-known that these hyperplanes correspond to $\mathrm{Inv}(v^{-1})$. 
As Weyl group elements are uniquely determined by their inversion sets 
(see, for example, \cite[Proposition~2.1]{HL} or \cite[(2.2.6)]{Mac}), 
we deduce $v=\mcr{\lng}^{-1}$. This concludes the proof. 
\end{proof}

We write the $\Gamma$ in Lemma~\ref{lem:min} as $\Gamma=(\gamma_{1},\dots,\gamma_{a+b})$, i.e., 
\begin{equation*}
\Gamma = (\gamma_{1},\dots,\gamma_{a+b}) = 
(-\beta_{a},\dots,-\beta_{1},\zeta_{1},\dots,\zeta_{b}). 
\end{equation*}
For $1 \le t \le a+b$, let $H_{\gamma_{t},l_{t}}$ be the affine hyperplane 
between the $(t-1)$-th alcove and the $t$-th alcove in $\Gamma$; 
recall that $l_{t}' = \pair{\lambda}{\gamma_{t}^{\vee}}-l_{t}$ 
for $1 \le t \le a+b$. Then we see that 
\begin{align*}
(l_{1},\dots,l_{a},l_{a+1},\dots,l_{a+b}) & = (0,\dots,0,1,\dots,1), \\
(l_{1}',\dots,l_{a}',l_{a+1}',\dots,l_{a+b}') & = (1,\dotsc,1,0,\dotsc,0). 
\end{align*}

Let $w \in W$. We see that the sets $\QW_{\lambda,w}=\QW_{\lambda,w}(\Gamma)$ and 
$\ti{\QW}_{\lambda,w}=\ti{\QW}_{\lambda,w}(\Gamma)$ agree with
$\QW_{\lambda,w}^{\KNOS}$ and $\ti{\QW}_{\lambda,w}^{\KNOS}$ 
(defined by using $\vec{\eta}=
(\beta_{a},\dots,\beta_{1},\zeta_{1},\dots,\zeta_{b})$), respectively. 
Let $\bw=(w_{0},w_{1},\dots,w_{a},w_{a+1},\dots,w_{a+b}) \in \QW_{\lambda,w}$. 
Then we deduce from Lemma~\ref{lem:min} and \eqref{eq:har} that 
$\wt(\bw)=w_{a}^{-1}\lambda$. Also, we see that 
the set $S(\bw)$ agrees with $S(\bw)^{\KNOS}$, and that
$\qwt(\bw,\bb)$ (and $\qwt^{\vee}(\bw,\bb)$) is identical to 
$-\lng \wt(\bw,\bb)^{\KNOS}$ for $(\bw,\bb) \in \ti{\QW}_{\lambda,w}$; 
remark that $\qwt(\bw,\bb)=\qwt^{\vee}(\bw,\bb)$ under the identification 
of roots and coroots, mentioned in Section~\ref{subsec:simplylaced}. 

We will show that the degree function $\deg(\bw,\bb)$ defined by \eqref{eq:deg} 
agrees with $\deg(\bw,\bb)^{\KNOS}$. 
Fix $(\bw,\bb) \in \ti{\QW}_{\lambda,w}$. We set
\begin{equation*}
A := \sum_{u \in T^{-}(\bw)} A_u, \qquad
B := \sum_{u \in S(\bw)} B_u,
\end{equation*}
where
\begin{align*}
A_{u} & := \sum_{t=u}^{a+b} \pair{\wt_t(\bw)}{|w_u^{-1}\gamma_u|^{\vee}} 
\qquad \text{for $u \in T^{-}(\bw)$}, \\[2mm]
B_{u} & := -\bb(u) \sum_{t=u}^{a+b} \pair{\wt_t(\bw)}{w_u^{-1}\gamma_u^{\vee}} 
\qquad \text{for $u \in S(\bw)$}, 
\end{align*}
so that
\begin{align*}
& \sum_{t=1}^{a+b} \pair{\wt_t(\bw)}{\qwt_t^{\vee}(\bw,\bb)} \\
& \qquad = 
  \sum_{t=1}^{a+b}
 \sum_{
 \begin{subarray}{c}
 1 \le u \le t \\[1mm]
 u \in T^{-}(\bw)
 \end{subarray}} \pair{\wt_t(\bw)}{|w_{u}^{-1}\gamma_{u}|^{\vee}} + 
\sum_{t=1}^{a+b}
\sum_{
 \begin{subarray}{c}
 1 \le u \le t \\[1mm]
 u \in S(\bw)
 \end{subarray}} (- \bb(t)) \pair{\wt_t(\bw)}{ w_{t}^{-1}\gamma_{t}^{\vee} } \\[3mm]
& \qquad = 
  \underbrace{ \sum_{u \in T^{-}(\bw)}
  \overbrace{ \sum_{t=u}^{a+b} \pair{\wt_t(\bw)}{|w_{u}^{-1}\gamma_{u}|^{\vee}} }^{=A_{u}} }_{=A} + 
  \underbrace{ \sum_{u \in S(\bw)} 
  \overbrace{ (- \bb(u)) \sum_{t=u}^{a+b} \pair{\wt_t(\bw)}{ w_{t}^{-1}\gamma_{t}^{\vee} } }^{=B_{u}} }_{=B}. 
\end{align*}
By \eqref{eq:wttw}, we observe that $\wt_t(\bw) = 0$ for all $a < t \le a+b$ 
since $l_{t}'=0$ for such a $t$. 
Hence we have
\begin{align*}
& A_u = 0 \quad \text{for $u \in T^{-}(\bw)$ such that $a < u \le a+b$}, \\
& B_u = 0 \quad \text{for $u \in S(\bw)$ such that $a < u \le a+b$}. 
\end{align*}
Since $l_{t}=0$ for all $1 \le t \le a$, we see by \eqref{eq:har} that 
$\ha{r}_{t}(\bw)\lambda = w_{t}^{-1}\lambda$ for $0 \le t \le a$. Hence, 
for $u \in T^{-}(\bw)$ with $1\le u \le a$, we have
\begin{align*}
A_u & = \sum_{t=u}^{a+b} \pair{\wt_t(\bw)}{ |w_u^{-1}\gamma_u|^{\vee} } 
      = \sum_{t=u}^{a} \pair{\wt_t(\bw)}{ |w_u^{-1}\gamma_u|^{\vee} } \\[2mm]
& = \pair{\ha{r}_{a}(\bw)\lambda - \ha{r}_{u-1}(\bw)\lambda}{ |w_u^{-1}\gamma_u|^{\vee} } \\
& = \pair{w_{a}^{-1}\lambda}{ |w_u^{-1}\gamma_u|^{\vee} } - 
    \pair{w_{u-1}^{-1}\lambda}{ |w_u^{-1}\gamma_u|^{\vee} }. 
\end{align*}
Here we note that $\gamma_{u} \in \Delta^{-}$ for all $1 \le u \le a$, 
which implies that 
$w_{u}^{-1}\gamma_{u} \in \Delta^{-}$ for all $u \in T^{-}(\bw)$ with $1\le u\le a$. 
Therefore, 
\begin{align*}
A_{u} & = 
    \pair{w_{a}^{-1}\lambda}{ |w_u^{-1}\gamma_u|^{\vee} } 
    + \pair{w_{u-1}^{-1}\lambda}{ w_u^{-1}\gamma_u^{\vee} } \\
& = \pair{w_{a}^{-1}\lambda}{ |w_{u}^{-1}\gamma_u|^{\vee} } 
    - \pair{\lambda}{ \gamma_u^{\vee} } 
  = \pair{w_{a}^{-1}\lambda}{ |w_{u}^{-1}\gamma_u|^{\vee} } - 1 \\
& = \pair{w_{a}^{-1}\lambda}{ |w_{u}^{-1}\gamma_u|^{\vee} } + \sgn(\gamma_u), 
\end{align*}
and hence
\begin{align*}
A & = \sum_{u \in T^{-}(\bw)} A_{u}
    = \sum_{\substack{u \in T^{-}(\bw) \\[1mm] 1 \le u \le a}} A_{u}
    = \sum_{\substack{u \in T^{-}(\bw) \\[1mm] 1 \le u \le a}}
      \pair{w_{a}^{-1}\lambda}{ |w_{u}^{-1}\gamma_u|^{\vee} } +
      \sum_{\substack{u \in T^{-}(\bw) \\[1mm] 1 \le u \le a}} \sgn(\gamma_u) \\[3mm]
& = \sum_{\substack{u \in T^{-}(\bw) \\[1mm] 1 \le u \le a}}
      \pair{w_{a}^{-1}\lambda}{ |w_{u}^{-1}\gamma_u|^{\vee} } +
      \sum_{u \in T^{-}(\bw)} \sgn(\gamma_u)l_{u}'.
\end{align*}
Similarly, for $u \in S(\bw)$ with $1\le u \le a$, we have
\begin{align*}
B_u & = - \bb(u) \sum_{t=u}^{a+b} \pair{\wt_t(\bw)}{ w_u^{-1}\gamma_u^{\vee} } 
      = - \bb(u) \sum_{t=u}^{a} \pair{\wt_t(\bw)}{ w_u^{-1}\gamma_u^{\vee} } \\[2mm]
& = - \bb(u) \pair{\ha{r}_{a}(\bw)\lambda - \ha{r}_{u-1}(\bw)\lambda}{ w_u^{-1}\gamma_u^{\vee} }
  = - \bb(u)( 
      \pair{w_{a}^{-1}\lambda}{ w_{u}^{-1}\gamma_u^{\vee} } -
      \pair{w_{u-1}^{-1}\lambda}{ w_u^{-1}\gamma_u^{\vee} }) \\
& = - \bb(u)( 
      \pair{w_{a}^{-1}\lambda}{ w_{u}^{-1}\gamma_u^{\vee} } - 
      \pair{\lambda}{ \gamma_u^{\vee} }) \quad \text{(note that $w_{u-1}=w_{u}$ since $u \in S(\bw)$)} \\
& = - \bb(u)( 
      \pair{w_{a}^{-1}\lambda}{ w_{u}^{-1}\gamma_u^{\vee} } - 1)
      \quad \text{(since $\pair{\lambda}{ \gamma_u^{\vee} } = \pair{\lambda}{-\beta_{a-u+1}^{\vee}} = 1$)}, 
\end{align*}
and hence
\begin{align*}
B & = \sum_{u \in S(\bw)} B_{u}
    = \sum_{\substack{u \in S(\bw) \\[1mm] 1 \le u \le a}} B_{u}
    = \sum_{\substack{u \in S(\bw) \\[1mm] 1 \le u \le a}}
      \pair{w_{a}^{-1}\lambda}{- \bb(u) w_{u}^{-1}\gamma_u^{\vee} } +
      \sum_{\substack{u \in S(\bw) \\[1mm] 1 \le u \le a}} \bb(u) \\[3mm]
&   = \sum_{\substack{u \in S(\bw) \\[1mm] 1 \le u \le a}}
      \pair{w_{a}^{-1}\lambda}{- \bb(u) w_{u}^{-1}\gamma_u^{\vee} } +
      \sum_{\substack{u \in S(\bw) \\[1mm] 1 \le u \le a}} \bb(u)l_{u}'. 
\end{align*}
Putting all this together, we deduce that
$\deg'(\bw,\bb) = \pair{w_{a}^{-1}\lambda}{\qwt_{a}^{\vee}(\bw,\bb)}$, 
and hence that
\begin{align*}
\deg(\bw,\bb) & = 
\frac{1}{2} \pair{\qwt(\bw,\bb)}{\qwt^{\vee}(\bw,\bb)} + 
\pair{w_{a}^{-1}\lambda}{\qwt_{a}^{\vee}(\bw,\bb)} \\
& = 
\frac{1}{2} ( \qwt(\bw,\bb), \qwt(\bw,\bb) ) + 
( w_{a}^{-1}\lambda, \qwt_{a}(\bw,\bb) ), 
\end{align*}
which agrees with $\deg(\bw,\bb)^{\KNOS}$, as desired. 
Therefore, by Theorem~\ref{thm:KNOS}, 
we conclude that Theorem~\ref{thm:main} holds for $\Gamma \in \RAP(\lambda)$ 
of the form \eqref{eq:mGam}. 
This completes the proof of Proposition~\ref{prop:minu}.

\section{Proof of Proposition~\ref{prop:sum}.}
\label{sec:sum}

Recall the notation and setting of Proposition~\ref{prop:sum}, 
equation \eqref{eq:cat}, and Remark~\ref{rem:cat}. 
We write $\Gamma \in \AP(\lambda)$ and $\Xi \in \AP(\mu)$ as: 
$\Gamma=(\gamma_{1},\dots,\gamma_{m})$ and 
$\Xi=(\xi_{1},\dots,\xi_{p})$, respectively. 
If we set $\gamma_{t}:=\xi_{t-m}$ for $m+1 \le t \le m+p$, then
\begin{equation*}
\Gamma \ast \Xi = (\gamma_{1},\dots,\gamma_{m}, \xi_{1},\dots,\xi_{p}) = 
(\gamma_{1},\dots,\gamma_{m+p}). 
\end{equation*}
For $1 \le t \le m$, 
let $H_{\gamma_{t},l_{t}}$ be the affine hyperplane 
between the $(t-1)$-th alcove and the $t$-th alcove in $\Gamma$. 
Similarly, for $1 \le q \le p$, 
let $H_{\xi_{q},k_{q}}$ be the affine hyperplane 
between the $(q-1)$-th alcove and the $q$-th alcove in $\Xi$. 
Then the affine hyperplane between 
the $(t-1)$-th alcove and the $t$-th alcove in $\Gamma \ast \Xi$ is 
$H_{\gamma_{t},l_{t}}$ (resp., $H_{\xi_{t-m},\pair{\lambda}{\xi_{t-m}^{\vee}} + k_{t-m}}$) 
for $1 \le t \le m$ (resp., $m+1 \le t \le m+p$); we set 
$l_{t}:=\pair{\lambda}{\xi_{t-m}^{\vee}} + k_{t-m}$ for $m+1 \le t \le m+p$. 
Also, we set
\begin{equation} \label{eq:lta}
\begin{cases}
l_{t}'':= \pair{\lambda+\mu}{\gamma_{t}^{\vee}} - l_{t} 
 & \text{for $1 \le t \le m+p$}, \\[1mm]
l_{t}':= \pair{\lambda}{\gamma_{t}^{\vee}} - l_{t} 
 & \text{for $1 \le t \le m$}, \\[1mm]
k_{q}':= \pair{\mu}{\xi_{q}^{\vee}} - k_{q} = 
\pair{\mu}{\gamma_{m+q}^{\vee}} - l_{m+q} 
 & \text{for $1 \le q \le p$}.
\end{cases}
\end{equation}

For $\bw = (w_{0},\dots,w_{m}) \in \QW_{\lambda,w}(\Gamma)$ and 
$\bv = (v_{0},v_{1},\dots,v_{p}) \in \QW_{\mu,\ed(\bw)}(\Xi)$ 
(note that $\ed(\bw)=w_{m}$), we set 
\begin{equation}
\bw \ast \bv:=(w_{0},\dots,w_{m},v_{1},\dots,v_{p}) \in 
\QW_{\lambda+\mu,w}(\Gamma \ast \Xi); 
\end{equation}
if we set $w_{q}:=v_{q-m}$ for $m+1 \le q \le m+p$, then 
we can write $\bw \ast \bv$ as: 
\begin{equation*}
\bw \ast \bv 
 = (w_{0},\dots,w_{m},v_{1},\dots,v_{p}) 
 = (w_{0},\dots,w_{m},w_{m+1},\dots,w_{m+p});
\end{equation*}
note that $S(\bw \ast \bv) \subset [1,m+p]$. 
We see that $S(\bw \ast \bv) = S(\bw) \sqcup \bigl\{ m+q \mid q \in S(\bv) \bigr\}$. 
For $\bb:S(\bw) \rightarrow \{0,1\}$ and $\bc:S(\bv) \rightarrow \{0,1\}$, 
we define $\bb \ast \bc: S(\bw \ast \bv) \rightarrow \{0,1\}$ by
\begin{equation}
(\bb \ast \bc)(t) := 
\begin{cases}
b(t) & \text{if $t \in S(\bw)$}, \\[1mm]
c(t-m) & \text{if $t \in \bigl\{ m+q \mid q \in S(\bv) \bigr\}$}. 
\end{cases}
\end{equation}
Then it follows that 
\begin{equation} \label{eq:lm}
\ti{\QW}_{\lambda+\mu,w}(\Gamma \ast \Xi)= 
\left\{ (\bw \ast \bv,\bb \ast \bc) \ \Biggm| \ 
\begin{array}{l}
(\bw,\bb) \in \ti{\QW}_{\lambda,w}(\Gamma), \\[3mm]
(\bv,\bc) \in \ti{\QW}_{\mu,\ed(\bw)}(\Xi)
\end{array} \right\}. 
\end{equation}

By the assumption, we have in $\BK \subset \KQGr$,
\begin{align*}
& \be^{\lambda+\mu} \cdot \OQG{w} \\[5mm]
& =
 \sum_{(\bw,\bb) \in \ti{\QW}_{\lambda,w}(\Gamma)} (-1)^{(\bw,\bb)} q^{\deg(\bw,\bb)} 
 \be^{\mu} \cdot \OQGL{ \ed(\bw) t_{ \qwt^{\vee}(\bw,\bb) } }{ -\lng \wt(\bw) - \lng \qwt(\bw,\bb) } \\[5mm]
%
%
& = \sum_{(\bw,\bb) \in \ti{\QW}_{\lambda,w}(\Gamma)} \ \sum_{(\bv,\bc) \in \ti{\QW}_{\mu,\ed(\bw)}(\Xi)}
(-1)^{(\bw,\bb)}(-1)^{(\bv,\bc)} \times \\[2mm]
& \hspace*{5mm} q^{\deg(\bw,\bb) + \deg(\bv,\bc) + \pair{ \wt(\bv)+\qwt(\bv,\bc) }{ \qwt^{\vee}(\bw,\bb) } } \times \\
& \hspace*{5mm} \OQGL{ \ed(\bv) t_{ \qwt^{\vee}(\bw,\bb) + \qwt^{\vee}(\bv,\bc) } }%
  { -\lng \wt(\bw) -\lng \qwt(\bw,\bb) -\lng \wt(\bv)-\lng \qwt(\bv,\bc) } \\
& = : (\spadesuit)
\end{align*}
We can easily verify that 
\begin{equation} \label{eq:misc*}
\begin{cases}
(-1)^{(\bw,\bb)}(-1)^{(\bv,\bc)} = (-1)^{(\bw \ast \bc,\bb \ast \bc)}, \\[1mm]
\qwt(\bw,\bb) + \qwt(\bv,\bc) =  \qwt(\bw \ast \bc,\bb \ast \bc), \\[1mm]
\qwt^{\vee}(\bw,\bb) + \qwt^{\vee}(\bv,\bc) =  \qwt^{\vee}(\bw \ast \bc,\bb \ast \bc), \\[1mm]
\ed(\bv) = \ed(\bw \ast \bv). 
\end{cases}
\end{equation}
%
First we claim that
\begin{equation} \label{eq:wt*}
\wt(\bw)+\wt(\bv) = \wt(\bw \ast \bv). 
\end{equation}
Let $\bw = (w_{0},w_{1},\dots,w_{m}) \in \QW_{\lambda,w}$ and
$\bv = (v_{0},v_{1},\dots,v_{p}) \in \QW_{\mu,\ed(\bw)}$. 
If we write 
\begin{align*}
& 
T(\bw) = 
\bigl\{ 1 \le t \le m \mid w_{t-1} \ne w_{t} \bigr\} = 
\bigl\{ t_{1} < t_{2} < \cdots < t_{c} \bigr\}, \\
& 
T(\bv) = 
\bigl\{ 1 \le q \le p \mid v_{q-1} \ne v_{q} \bigr\} = 
\bigl\{ q_{1} < q_{2} < \cdots < q_{d} \bigr\},
\end{align*}
then we see by \eqref{eq:har} that 
\begin{align*}
\wt(\bw) & = \underbrace{w^{-1}s_{\gamma_{t_1}} \cdots s_{\gamma_{t_{c-1}}}s_{\gamma_{t_c}}}_{%
    =w_{m}^{-1}=\ed(\bw)^{-1}} \lambda + 
    \sum_{a=1}^{c} l_{t_a} w^{-1}s_{\gamma_{t_1}} \cdots s_{\gamma_{t_{a-1}}} \gamma_{t_a}, 
\end{align*}
and that 
\begin{align*}
& \wt(\bv) = 
  w_{m}^{-1}s_{\xi_{q_1}} \cdots s_{\xi_{q_{d-1}}}s_{\xi_{q_d}} \mu + 
  \sum_{b=1}^{d} k_{q_b} w_{m}^{-1}s_{\xi_{q_1}} \cdots s_{\xi_{q_{b-1}}} \xi_{q_b} \\
& = 
w^{-1}s_{\gamma_{t_1}} \cdots s_{\gamma_{t_{c}}}s_{\xi_{q_1}} \cdots s_{\xi_{q_d}} \mu + 
  \sum_{b=1}^{d} k_{q_b} w^{-1}s_{\gamma_{t_1}} \cdots s_{\gamma_{t_{c}}}
  s_{\xi_{q_1}} \cdots s_{\xi_{q_{b-1}}} \xi_{q_b}. 
\end{align*}
Also, we deduce from Remark~\ref{rem:cat} and \eqref{eq:har} that 
\begin{align*}
& \wt(\bw \ast \bv) =
  w^{-1}s_{\gamma_{t_1}} \cdots s_{\gamma_{t_{c}}}s_{\xi_{q_1}} \cdots s_{\xi_{q_d}}(\lambda+\mu) \\
& \quad + 
  \sum_{a=1}^{c} l_{t_a} w^{-1}s_{\gamma_{t_1}} \cdots s_{\gamma_{t_{a-1}}} \gamma_{t_a} +
  \sum_{b=1}^{d} (\pair{\lambda}{\xi_{q_b}^{\vee}}+k_{q_b})w^{-1}s_{\gamma_{t_1}} \cdots s_{\gamma_{t_{c}}}
  s_{\xi_{q_1}} \cdots s_{\xi_{q_{b-1}}} \xi_{q_b}. 
\end{align*}
Since 
\begin{equation*}
s_{\xi_{q_1}} \cdots s_{\xi_{q_d}} \lambda =  
\lambda - \sum_{b=1}^{d} \pair{\lambda}{\xi_{q_b}^{\vee}}
  s_{\xi_{q_1}} \cdots s_{\xi_{q_{b-1}}} \xi_{q_b}, 
\end{equation*}
we obtain \eqref{eq:wt*}, as desired. 

Next we claim that
\begin{equation} \label{eq:deg*}
\deg(\bw,\bb) + \deg(\bv,\bc) + \pair{ \wt(\bv)+\qwt(\bv,\bc) }{ \qwt^{\vee}(\bw,\bb) } = 
\deg(\bw \ast \bv,\bb \ast \bc).
\end{equation}
We use $\qwt(\bw \ast \bv,\bb \ast \bc)=\qwt(\bw,\bb)+\qwt(\bv,\bc)$ and 
$\qwt^{\vee}(\bw \ast \bv,\bb \ast \bc)=\qwt^{\vee}(\bw,\bb)+\qwt^{\vee}(\bv,\bc)$ to obtain
\begin{align*}
& \frac{1}{2}\pair{\qwt(\bw \ast \bv,\bb \ast \bc)}{\qwt^{\vee}(\bw \ast \bv,\bb \ast \bc)} \\
& = 
\frac{1}{2}\pair{\qwt(\bw,\bb)}{\qwt^{\vee}(\bw,\bb)}+
\frac{1}{2}\pair{\qwt(\bv,\bc)}{\qwt^{\vee}(\bv,\bc)}, \\
& \qquad 
+\underbrace{\frac{1}{2}\pair{\qwt(\bv,\bc)}{\qwt^{\vee}(\bw,\bb)}
+\frac{1}{2}\pair{\qwt(\bw,\bb)}{\qwt^{\vee}(\bv,\bc)}}_{=\pair{\qwt(\bv,\bc)}{\qwt^{\vee}(\bw,\bb)}
 \text{ since $\Fg$ is simply-laced} },
\end{align*}
which enables us to simplify \eqref{eq:deg*} as:
\begin{equation}\label{eq:deg*2}
\deg'(\bw \ast \bv,\bb \ast \bc) = 
\deg'(\bw,\bb)+\deg'(\bv,\bc)+\pair{\wt(\bv)}{\qwt^{\vee}(\bw,\bb)}. 
\end{equation}
Thus, what we need to show is:
\begin{equation} \label{eq:deg*3}
\begin{split}
& \sum_{t=1}^{m+p} \pair{\wt_t(\bw \ast \bv)}{\qwt_{t}^{\vee}(\bw \ast \bv,\bb \ast \bc)} 
  - \sum_{ t \in T^{-}(\bw \ast \bv) } \sgn(\gamma_{t})l_{t}'' 
  - \sum_{ t \in S(\bw \ast \bv) } (\bb \ast \bc)(t) l_{t}'' \\[3mm]
& =
\sum_{t=1}^{m} \pair{\wt_t(\bw)}{\qwt_{t}^{\vee}(\bw,\bb)} 
  - \sum_{ t \in T^{-}(\bw) } \sgn(\gamma_{t})l_{t}' 
  - \sum_{ t \in S(\bw) } \bb(t) l_{t}' \\[3mm]
& + \sum_{q=1}^{p} \pair{\wt_q(\bv)}{\qwt_{q}^{\vee}(\bv,\bc)} 
  - \sum_{ q \in T^{-}(\bv) } \sgn(\xi_{q})k_{q}' 
  - \sum_{ q \in S(\bv) } \bc(q) k_{q}' \\[3mm]
& + \pair{\wt(\bv)}{\qwt^{\vee}(\bw,\bb)}. 
\end{split}
\end{equation}
Let us write a part (i.e., the first sum) of $\deg'(\bw \ast \bv,\bb \ast \bc)$ as follows: 
\begin{equation*}
\sum_{t=1}^{m+p} \pair{\wt_t(\bw \ast \bv)}{\qwt_{t}^{\vee}(\bw \ast \bv,\bb \ast \bc)}=
\sum_{u \in T^{-}(\bw \ast \bv)}A_u + \sum_{u \in S(\bw \ast \bv)} B_u, 
\end{equation*}
where
\begin{align*}
A_u & := \sum_{t=u}^{m+p} \pair{\wt_t(\bw \ast \bv)}{ |w_u^{-1}\gamma_u|^{\vee} } \qquad 
      \text{for $u \in T^{-}(\bw \ast \bv)$}, \\
B_u & := - (\bb \ast \bc)(u)
      \sum_{t=u}^{m+p} \pair{ \wt_t(\bw \ast \bv) }{ w_u^{-1}\gamma_u^{\vee} } \qquad 
      \text{for $u \in S(\bw \ast \bv)$}. 
\end{align*}
For $m < t \le m+p$, we have
\begin{equation} \label{eq:lt''}
l_{t}'' = \pair{\lambda+\mu}{\xi_{t-m}^{\vee}}-(k_{t-m}+\pair{\lambda}{\xi_{t-m}^{\vee}})
        = \pair{\mu}{\xi_{t-m}^{\vee}}-k_{t-m} = k_{t-m}', 
\end{equation}
and hence
\begin{align*}
\wt_t(\bw \ast \bv) & =
\begin{cases}
- l_{t}'' w_{t-1}^{-1}\gamma_t & \text{if $t \in T(\bw \ast \bv)$}, \\[2mm]
0 & \text{otherwise}
\end{cases} \qquad \text{by \eqref{eq:wttw}} \\[3mm]
&=
\begin{cases}
- k_{t-m}' v_{t-m-1}^{-1}\xi_{t-m} & \text{if $t-m \in T(\bv)$}, \\[2mm]
0 & \text{otherwise}
\end{cases} \\[3mm]
&= \wt_{t-m}(\bv).
\end{align*}
Therefore, we deduce that
\begin{align*}
A_{u} 
& = \sum_{t=u}^{m+p} \pair{\wt_t(\bw \ast \bv)}{|w_u^{-1}\gamma_u|^{\vee}}
  = \sum_{t=u-m}^{p}\pair{\wt_{t}(\bv)}{|v_{u-m}^{-1}\xi_{u-m}|^{\vee}}
\end{align*}
for $u \in T^{-}(\bw \ast \bv)$ with $m < u\le m+p$, and that
\begin{align*}
B_{u} 
& = - (\bb \ast \bc)(u) \sum_{t=u}^{m+p} \pair{ \wt_t(\bw \ast \bv) }{ w_u^{-1}\gamma_u^{\vee} } 
  = - \bc(u-m) \sum_{t=u-m}^{p}\pair{ \wt_{t}(\bv)}{ v_{u-m}^{-1}\xi_{u-m}^{\vee} }
\end{align*}
for $u \in S(\bw \ast \bv)$ with $m < u\le m+p$. Here we remark that 
\begin{equation*}
\sum_{q=1}^{p} \pair{\wt_q(\bv)}{\qwt_{q}^{\vee}(\bv,\bc)} = 
 \sum_{u \in T^{-}(\bv)}A_u' + \sum_{u \in S(\bv)} B_u',
\end{equation*}
where
\begin{align*}
A_u' & := \sum_{t=u}^{p} \pair{\wt_t(\bv)}{ |v_u^{-1}\xi_u|^{\vee} } \qquad 
      \text{for $u \in T^{-}(\bv)$}, \\
B_u' & := - \bc(u)
      \sum_{t=u}^{p} \pair{ \wt_t(\bv) }{ v_u^{-1}\xi_u^{\vee} } \qquad 
      \text{for $u \in S(\bv)$}. 
\end{align*}
Since $A_{u}=A_{u-m}'$ for 
$u \in T^{-}(\bw \ast \bv) \cap \{m+1,m+2,\dots,m+p\} = \{ m+u' \mid u' \in T^{-}(\bv) \}$, 
and since $B_{u}=B_{u-m}'$ for 
$u \in S(\bw \ast \bv) \cap \{m+1,m+2,\dots,m+p\} = \{ m+u' \mid u' \in S(\bv) \}$, 
it follows that 
\begin{equation*}
\sum_{ \substack{ u \in T^{-}(\bw \ast \bv) \\[1mm] m < u \le m+p} }A_u + 
\sum_{ \substack{ u \in S(\bw \ast \bv) \\[1mm] m < u \le m+p} } B_u 
= \sum_{q=1}^{p} \pair{\wt_q(\bv)}{\qwt_{q}^{\vee}(\bv,\bc)}. 
\end{equation*}
Thus, equation \eqref{eq:deg*3} (which we need to show) is equivalent to: 
\begin{equation} \label{eq:deg*4}
\begin{split}
& \sum_{ \substack{ u \in T^{-}(\bw \ast \bv) \\[1mm] 1 \le u \le m} }A_u + 
  \sum_{ \substack{ u \in S(\bw \ast \bv) \\[1mm] 1 \le u \le m} } B_u
  - \sum_{ t \in T^{-}(\bw \ast \bv) } \sgn(\gamma_{t})l_{t}'' 
  - \sum_{ t \in S(\bw \ast \bv) } (\bb \ast \bc)(t) l_{t}'' \\[3mm]
& =
\sum_{t=1}^{m} \pair{\wt_t(\bw)}{\qwt_{t}^{\vee}(\bw,\bb)} 
  - \sum_{ t \in T^{-}(\bw) } \sgn(\gamma_{t})l_{t}' 
  - \sum_{ t \in S(\bw) } \bb(t) l_{t}' \\[3mm]
& - \sum_{ q \in T^{-}(\bv) } \sgn(\xi_{q})k_{q}' 
  - \sum_{ q \in S(\bv) } \bc(q) k_{q}' + \pair{\wt(\bv)}{\qwt^{\vee}(\bw,\bb)}.
\end{split}
\end{equation}

If $1 \le t \le m$, then we have
\begin{equation} \label{eq:lt2}
l_{t}'' = \pair{\lambda+\mu}{\gamma_t^{\vee}}-l_t 
= \pair{\lambda}{\gamma_{t}^{\vee}} -l_{t} + \pair{\mu}{\gamma_{t}^{\vee}} = 
l_{t}' + \pair{\mu}{\gamma_{t}^{\vee}}, 
\end{equation}
and hence
\begin{align*}
\wt_{t}(\bw \ast \bv) &=
\begin{cases}
- l_{t}''w_{t-1}^{-1}\gamma_{t} & \text{if $t \in T(\bw \ast \bv)$}, \\
0 & \text{otherwise}
\end{cases} \quad \text{by \eqref{eq:wttw}} \\[3mm]
&=
\begin{cases}
-(l_{t}'+\pair{\mu}{\gamma_t^{\vee}})w_{t-1}^{-1}\gamma_t
  & \text{if $t \in T(\bw)$}, \\
0 & \text{otherwise}
\end{cases} \\[3mm]
& = \wt_{t}(\bw)+w_{t}^{-1}\mu-w_{t-1}^{-1}\mu. 
\end{align*}
Therefore, we have
\begin{align*}
\sum_{t=u}^{m+p}\wt_t(\bw \ast \bv)
 & = \sum_{t=u}^{m}\wt_t(\bw \ast \bv) + 
     \sum_{t=m+1}^{m+p} \underbrace{\wt_{t}(\bw \ast \bv)}_{=\wt_{t-m}(\bv)} \\
 & = w_{m}^{-1}\mu-w_{u-1}^{-1}\mu+\sum_{t=u}^{m} \wt_t(\bw) +
     \wt(\bv) - w_{m}^{-1} \mu. 
\end{align*}
Hence we deduce that
\begin{align*}
A_u 
  & = \pair{- w_{u-1}^{-1}\mu}{ |w_u^{-1}\gamma_u|^{\vee} }
     + \sum_{t=u}^{m} \pair{ \wt_t(\bw) }{ |w_u^{-1}\gamma_u|^{\vee} }
     + \pair{\wt(\bv)}{ |w_u^{-1}\gamma_u|^{\vee} } \\
  & = \sgn(\gamma_u) \pair{\mu}{\gamma_u^{\vee}} 
     +\sum_{t=u}^{m}\pair{\wt_t(\bw)}{ |w_u^{-1}\gamma_u|^{\vee} }
     +\pair{\wt(\bv)}{ |w_u^{-1}\gamma_u|^{\vee} }
\end{align*}
for $u \in T^{-}(\bw \ast \bv)$ with $1\le u\le m$, and that
\begin{align*}
B_u 
  & = \bb(u)\pair{\mu}{\gamma_u^{\vee}}
      -\bb(u)\sum_{t=u}^{m}\pair{\wt_t(\bw)}{w_u^{-1}\gamma_u^{\vee}}  
      -\bb(u)\pair{\wt(\bv)}{w_u^{-1}\gamma_u^{\vee}}
\end{align*}
for $u \in S(\bw \ast \bv)$ with $1 \le u \le m$. 
Here we remark that 
\begin{equation*}
\sum_{t=1}^{m} \pair{\wt_t(\bw)}{\qwt_{t}^{\vee}(\bw,\bb)} = 
 \sum_{u \in T^{-}(\bw)}A_u'' + \sum_{u \in S(\bw)} B_u'',
\end{equation*}
where
\begin{align*}
A_u'' & := \sum_{t=u}^{m} \pair{\wt_t(\bw)}{ |w_u^{-1}\gamma_u|^{\vee} } \qquad 
      \text{for $u \in T^{-}(\bw)$}, \\
B_u'' & := - \bb(u)
      \sum_{t=u}^{m} \pair{ \wt_t(\bw) }{ w_u^{-1}\gamma_u^{\vee} } \qquad 
      \text{for $u \in S(\bw)$}. 
\end{align*}
Since $A_{u}=A_{u}'' + \sgn(\gamma_u) \pair{\mu}{\gamma_u^{\vee}} + 
\pair{\wt(\bv)}{ |w_u^{-1}\gamma_u|^{\vee} }$ 
for $u \in T^{-}(\bw \ast \bv) \cap \{1,2,\dots,m\} = T^{-}(\bw)$, 
and since $B_{u}=B_u''+\bb(u)\pair{\mu}{\gamma_u^{\vee}}-
\bb(u)\pair{\wt(\bv)}{w_u^{-1}\gamma_u^{\vee}}$ for 
$u \in S(\bw \ast \bv) \cap \{1,2,\dots,m\} = S(\bw)$, 
it follows that 
\begin{equation*}
\begin{split}
\sum_{ \substack{ u \in T^{-}(\bw \ast \bv) \\[1mm] 1 \le u \le m} }A_u + 
\sum_{ \substack{ u \in S(\bw \ast \bv) \\[1mm] 1 \le u \le m} } B_u 
& = \sum_{t=1}^{m} \pair{\wt_t(\bw)}{\qwt_{t}^{\vee}(\bw,\bb)} \\
& + \sum_{ \substack{ u \in T^{-}(\bw \ast \bv) \\[1mm] 1 \le u \le m} }
\left( \sgn(\gamma_u) \pair{\mu}{\gamma_u^{\vee}}+\pair{\wt(\bv)}{ |w_u^{-1}\gamma_u|^{\vee}} \right) \\[3mm]
& + \sum_{ \substack{ u \in S(\bw \ast \bv) \\[1mm] 1 \le u \le m} } 
\left( \bb(u)\pair{\mu}{\gamma_u^{\vee}} - \bb(u)\pair{\wt(\bv)}{w_u^{-1}\gamma_u^{\vee}} \right) \\[3mm]
& = \sum_{t=1}^{m} \pair{\wt_t(\bw)}{\qwt_{t}^{\vee}(\bw,\bb)} + \pair{\wt(\bv)}{\qwt^{\vee}(\bw,\bb)} \\[3mm]
& + \sum_{ \substack{ u \in T^{-}(\bw \ast \bv) \\[1mm] 1 \le u \le m} }
    \sgn(\gamma_u) \pair{\mu}{\gamma_u^{\vee}}
  + \sum_{ \substack{ u \in S(\bw \ast \bv) \\[1mm] 1 \le u \le m} } 
    \bb(u) \pair{\mu}{\gamma_u^{\vee}}. 
\end{split}
\end{equation*}
Thus, equation \eqref{eq:deg*4} (which we need to show) is equivalent to: 
\begin{equation} \label{eq:deg*5}
\begin{split}
& \sum_{ \substack{ u \in T^{-}(\bw \ast \bv) \\[1mm] 1 \le u \le m} }
    \sgn(\gamma_u) \pair{\mu}{\gamma_u^{\vee}}
  + \sum_{ \substack{ u \in S(\bw \ast \bv) \\[1mm] 1 \le u \le m} } 
    \bb(u) \pair{\mu}{\gamma_u^{\vee}} \\
&  - \sum_{ t \in T^{-}(\bw \ast \bv) } \sgn(\gamma_{t})l_{t}'' 
  - \sum_{ t \in S(\bw \ast \bv) } (\bb \ast \bc)(t) l_{t}'' \\[3mm]
& =
  - \sum_{ t \in T^{-}(\bw) } \sgn(\gamma_{t})l_{t}' 
  - \sum_{ t \in S(\bw) } \bb(t) l_{t}'
  - \sum_{ q \in T^{-}(\bv) } \sgn(\xi_{q})k_{q}' 
  - \sum_{ q \in S(\bv) } \bc(q) k_{q}'. 
\end{split}
\end{equation}
We see by \eqref{eq:lta} and \eqref{eq:lt''} that 
\begin{equation*}
\sum_{ \substack{ u \in T^{-}(\bw \ast \bv) \\[1mm] 1 \le u \le m} }
    \sgn(\gamma_u) \pair{\mu}{\gamma_u^{\vee}}
- \sum_{ t \in T^{-}(\bw \ast \bv) } \sgn(\gamma_{t})l_{t}'' = 
  - \sum_{ t \in T^{-}(\bw) } \sgn(\gamma_{t})l_{t}' 
  - \sum_{ q \in T^{-}(\bv) } \sgn(\xi_{q})k_{q}', 
\end{equation*}
\begin{equation*}
  \sum_{ \substack{ u \in S(\bw \ast \bv) \\[1mm] 1 \le u \le m} } 
    \bb(u) \pair{\mu}{\gamma_u^{\vee}}
  - \sum_{ t \in S(\bw \ast \bv) } (\bb \ast \bc)(t) l_{t}'' =
  - \sum_{ t \in S(\bw) } \bb(t) l_{t}'  - \sum_{ q \in S(\bv) } \bc(q) k_{q}'. 
\end{equation*}
This proves \eqref{eq:deg*5}, and hence \eqref{eq:deg*}, as desired. 

Substituting \eqref{eq:misc*}, \eqref{eq:wt*}, and 
\eqref{eq:deg*} into ($\spadesuit$), we conclude that 
\begin{align*}
(\spadesuit) & = 
\sum_{(\bw \ast \bv,\bb \ast \bc) \in \ti{\QW}_{\lambda+\mu,w}(\Gamma \ast \Xi)}
(-1)^{(\bw \ast \bv,\bb \ast \bc)} q^{\deg(\bw \ast \bv,\bb \ast \bc)} \times \\
& \hspace*{20mm} \OQGL{ \ed(\bw \ast \bv) t_{ \qwt^{\vee}(\bw \ast \bv,\bb \ast \bc) } }%
{- \lng \wt(\bw \ast \bv) - \lng \qwt(\bw \ast \bv,\bb \ast \bc) } \\
& = \bF_{\lambda+\mu,w}(\Gamma \ast \Xi). 
\end{align*}
This completes the proof of Proposition~\ref{prop:sum}. 

\section{Proof of Theorem~\ref{thm:D}.} \label{sec:D}

Let $\lambda \in P$ be an arbitrary weight, and let 
\begin{equation} \label{eq:Gamma3}
\Gamma: 
  A_{\circ}=A_{0} 
  \xrightarrow{\gamma_{1}} A_{1} 
  \xrightarrow{\gamma_{2}} \cdots 
  \xrightarrow{\gamma_{m}} A_{m}=A_{\lambda}
\end{equation}
be an alcove path from the fundamental alcove $A_{\circ}$ 
to $A_{\lambda}=A_{\circ}+\lambda$. 
Assume that $\gamma_{s}=\alpha$, $\gamma_{s+1}=-\alpha$ 
for some $\alpha \in \Delta$ and $1 \le s \le m-1$, 
i.e., 
\begin{equation*}
\cdots \edge{\gamma_{s-1}} A_{s-1} \edge{\alpha} A_{s} \edge{-\alpha} A_{s+1} \edge{\gamma_{s+2}} A_{s+2} \cdots; 
\end{equation*}
note that $A_{s-1}=A_{s+1}$. 
We define $\Xi=(\beta_{1},\dots,\beta_{s-1},\beta_{s+2},\dots,\beta_{m})$, 
where $\beta_{k}:=\gamma_{k}$ for $1 \le k \le m$ with $k \ne s, s+1$. 
Then, $\Xi$ is an alcove path from $A_{\circ}$ to $A_{\lambda}$ of the form:
\begin{equation}
\Xi: 
  A_{\circ}=A_{0} 
  \xrightarrow{\gamma_{1}} \cdots 
  \edge{\gamma_{s-1}} A_{s-1} = A_{s+1} \edge{\gamma_{s+2}} A_{s+2} \edge{\gamma_{s+3}} \cdots 
\xrightarrow{\gamma_{m}} A_{m}=A_{\lambda}. 
\end{equation}

Now, let $(\bw,\bb) \in \ti{\QW}_{\lambda,w}(\Gamma)$, 
with $\bw=(w_{0},w_{1},\dots,w_{m}) \in \QW_{\lambda,w}(\Gamma)$. 
In the following, we will define $\D(\bw,\bb)$. 

\begin{case} \label{case1D}
Assume that $(\bw,\bb)$ satisfies $w_{s-1}=w_{s}=w_{s+1}$; 
observe that $S(\bw) \cap \{s,s+1\} = \emptyset$, $\{s\}$, or $\{s+1\}$. 

\begin{subcase} \label{subcase11D}
Assume that $S(\bw) \cap \{s,s+1\} = \emptyset$, or 
$S(\bw) \cap \{s,s+1\} \ne \emptyset$ and 
$\bb(t) = 0$ for $t \in S(\bw) \cap \{s,s+1\}$. 
In this case, we set 
\begin{equation*}
v_{k}:=w_{k} \qquad \text{for $0 \le k \le m$ with $k \ne s,s+1$}.
\end{equation*}
We see that $\bv=(v_{0},v_{1},\dots,v_{s-1},v_{s+2},\dots,v_{m}) \in 
\QW_{\lambda,w}(\Xi)$; notice that $S(\bv) = S(\bw) \setminus \{s,s+1\}$. 
We set $\bc(t) := \bb(t)$ for $t \in S(\bv)$. 
Then, $\D(\bw,\bb):=(\bv,\bc) \in \ti{\QW}_{\lambda,w}(\Xi)$ satisfies \eqref{eq:D1}.
Indeed, we can easily show the equalities in \eqref{eq:D1}, 
except for $\deg(\bv,\bc) = \deg(\bw,\bb)$. In order to show $\deg(\bv,\bc) = \deg(\bw,\bb)$, 
it suffices to show that $\deg'(\bv,\bc) = \deg'(\bw,\bb)$. For this equality, we claim that 
\begin{equation} \label{eq:X}
X := \sum_{t=s,s+1} \pair{ \wt_{t}(\bw) }{ \qwt_{t}^{\vee}(\bw,\bb) } -
 \sum_{ t \in T^{-}(\bw) \cap \{s,s+1\} } \sgn(\gamma_{t})l_{t}' -
 \sum_{ t \in S(\bw) \cap \{s,s+1\} }\bb(t)l_{t}'
\end{equation}
is equal to $0$. Indeed, since $\wt_{s}(\bw)=\wt_{s+1}(\bw)=0$, 
$T^{-}(\bw) \cap \{s,s+1\} = \emptyset$, and 
$\bb(t) = 0$ for $t \in S(\bw) \cap \{s,s+1\}$, 
we obtain $X=0$, as desired. 
\end{subcase}

\begin{subcase}[to be paired with Case~\ref{case4D} below] \label{subcase12D}
Assume that $S(\bw) \cap \{s,s+1\} \ne \emptyset$, and 
$\bb(t) = 1$ for $t \in S(\bw) \cap \{s,s+1\}$. 
We deduce that $|w_{s-1}^{-1}\alpha|$ is a simple root. 
Hence it follows that 
\begin{equation*}
w_{s-1}
\edge{|w_{s-1}^{-1}\alpha|} 
s_{\alpha}w_{s-1}
\edge{|w_{s-1}^{-1}\alpha|}
w_{s-1}; 
\end{equation*}
notice that one of these edges is a Bruhat edge, and the other is a quantum edge. 
We set 
\begin{equation*}
\begin{cases}
v_{k}:=w_{k} \qquad \text{for $0 \le k \le m$ with $k \ne s,s+1$}, \\
v_{s}:=s_{\alpha}w_{s-1}, \qquad v_{s+1}:=w_{s-1}.
\end{cases}
\end{equation*}
We see that $\bv=(v_{0},v_{1},\dots,v_{s-1},v_{s},v_{s+1},v_{s+2},\dots,v_{m}) \in 
\QW_{\lambda,w}(\Gamma)$; notice that $S(\bv) \cap \{s,s+1\} = \emptyset$, 
and $S(\bv) = S(\bw) \setminus \{s,s+1\}$. We set $\bc(t) := \bb(t)$ for $t \in S(\bv)$. 
Then, $\D(\bw,\bb)=(\bv,\bc) \in \ti{\QW}_{\lambda,w}(\Gamma)$ satisfies \eqref{eq:D2}.
Indeed, we can easily show the equalities in \eqref{eq:D2}, 
except for $\deg(\bv,\bc) = \deg(\bw,\bb)$. In order to show that $\deg(\bv,\bc) = \deg(\bw,\bb)$, 
it suffices to show that $\deg'(\bv,\bc) = \deg'(\bw,\bb)$. 
For this equality, we claim that $X$ in \eqref{eq:X} is equal to
\begin{equation} \label{eq:Y}
Y := \sum_{t=s,s+1} \pair{ \wt_{t}(\bv) }{ \qwt_{t}^{\vee}(\bv,\bc) } -
 \sum_{ t \in T^{-}(\bv) \cap \{s,s+1\} } \sgn(\gamma_{t})l_{t}' -
 \sum_{ t \in S(\bv) \cap \{s,s+1\} }\bc(t)l_{t}'. 
\end{equation}
Since $l_{s+1}'= - l_{s}'$, and $\wt_{s}(\bv) = -l_{s}'w_{s-1}^{-1}\alpha$, 
$\wt_{s+1}(\bv) = l_{s}'w_{s-1}^{-1}\alpha = -\wt_{s}(\bv)$, we deduce that
\begin{equation*}
\sum_{t=s,s+1} \pair{ \wt_{t}(\bv) }{ \qwt_{t}^{\vee}(\bv,\bc) } 
= \pair{ l_{s}'w_{s-1}^{-1}\alpha }{ \delta_{s+1 \in T^{-}(\bv)}|w_{s-1}^{-1}\alpha|^{\vee} },
\end{equation*}
where for a statement $\mathsf{P}$, we define
$\delta_{\mathsf{P}} := 1$ (resp., $:=0$) if $\mathsf{P}$ is true (resp., false). 
We see that $\delta_{s+1 \in T^{-}(\bv)}|w_{s-1}^{-1}\alpha|^{\vee} = 
\delta_{s+1 \in T^{-}(\bv)} \sgn(\alpha)w_{s-1}^{-1}\alpha^{\vee}$. 
Hence it follows that
\begin{equation*}
\sum_{t=s,s+1} \pair{ \wt_{t}(\bv) }{ \qwt_{t}^{\vee}(\bv,\bc) } 
= 2 \delta_{s+1 \in T^{-}(\bv)} l_{s}' \sgn(\alpha).
\end{equation*}
Also, we see that 
\begin{equation*}
\sum_{ t \in T^{-}(\bv) \cap \{s,s+1\} } \sgn(\gamma_{t})l_{t}' = \sgn(\alpha)l_{s}'
\quad \text{and} \quad
\sum_{ t \in S(\bv) \cap \{s,s+1\} }\bc(t)l_{t}' = 0.
\end{equation*}
Hence it follows that 
$Y = (2 \delta_{s+1 \in T^{-}(\bv)}-1) l_{s}' \sgn(\alpha)$.
As for $X$, we have 
\begin{equation*}
\sum_{t=s,s+1} \pair{ \wt_{t}(\bw) }{ \qwt_{t}^{\vee}(\bw,\bb) } = 
\sum_{ t \in T^{-}(\bw) \cap \{s,s+1\} } \sgn(\gamma_{t})l_{t}'= 0.
\end{equation*}
Also, we deduce that
\begin{align*}
& s \in S(\bw) \iff 
s+1 \in T^{-}(\bv) \text{ and } \alpha \in \Delta^{-}, \text{ or } 
s+1 \not\in T^{-}(\bv) \text{ and } \alpha \in \Delta^{+}, \\
& s+1 \in S(\bw) \iff 
s+1 \in T^{-}(\bv) \text{ and } \alpha \in \Delta^{+}, \text{ or }
s+1 \not\in T^{-}(\bv) \text{ and } \alpha \in \Delta^{-},
\end{align*}
which implies that 
\begin{equation*}
X= - \sum_{ t \in S(\bw) \cap \{s,s+1\} }\bb(t)l_{t}' = 
(2 \delta_{s+1 \in T^{-}(\bv)}-1) l_{s}' \sgn(\alpha). 
\end{equation*}
This proves $X=Y$, as desired. 
\end{subcase}

\end{case}

\begin{case}[to be paired with Case~\ref{case3D} below] \label{case2D}
Assume that $(\bw,\bb)$ satisfies 
$w_{s-1} \edge{|w_{s-1}^{-1}\alpha|} w_{s}=w_{s+1}$. We set 
\begin{equation*}
\begin{cases}
v_{k}:=w_{k} \qquad \text{for $0 \le k \le m$ with $k \ne s,s+1$}, & \\
\underbrace{w_{s-1}}_{=v_{s-1}=:v_{s}} \edge{|w_{s-1}^{-1}\alpha|} 
\underbrace{s_{\alpha}w_{s-1}=w_{s}=w_{s+1}}_{=:v_{s+1}}.
\end{cases}
\end{equation*}
We see that $\bv=(v_{0},v_{1},\dots,v_{s-1},v_{s},v_{s+1},v_{s+2},\dots,v_{m}) \in 
\QW_{\lambda,w}(\Gamma)$; notice that 
\begin{center}
$S(\bw) \cap \{s,s+1\} = \emptyset$ (resp., $=\{s+1\}$) 
if and only if 
$S(\bv) \cap \{s,s+1\} = \emptyset$ (resp., $=\{s\}$).
\end{center}
We set $\bc(t):=\bb(t)$ for $t \in S(\bv) \setminus \{s,s+1\} = S(\bw) \setminus \{s,s+1\}$, 
and $\bc(s):=\bb(s+1)$ if $S(\bv) \cap \{s,s+1\} = \{s\}$. 
Then, $\D(\bw,\bb)=(\bv,\bc) \in \ti{\QW}_{\lambda,w}(\Gamma)$ 
satisfies \eqref{eq:D2}. 
For the equality $\deg(\bv,\bc) = \deg(\bw,\bb)$, let us show that 
$X$ in \eqref{eq:X} is equal to $Y$ in \eqref{eq:Y}. We compute 
\begin{align*}
X & = \sum_{t=s,s+1} \pair{ \wt_{t}(\bw) }{ \qwt_{t}^{\vee}(\bw,\bb) } -
 \sum_{ t \in T^{-}(\bw) \cap \{s,s+1\} } \sgn(\gamma_{t})l_{t}' -
 \sum_{ t \in S(\bw) \cap \{s,s+1\} }\bb(t)l_{t}' \\[3mm]
& = \pair{ \wt_{s}(\bw) }{ \qwt_{s}^{\vee}(\bw,\bb) } - 
    \delta_{s \in T^{-}(\bw)} \sgn(\gamma_{s})l_{s}' - 
    \delta_{s+1 \in S(\bw) }\bb(s+1)l_{s+1}' \\
& = \pair{ \wt_{s}(\bw) }{ \qwt_{s}^{\vee}(\bw,\bb) } - 
    \delta_{s+1 \in T^{-}(\bv)} \sgn(\gamma_{s+1})l_{s+1}' + 
    \delta_{s \in S(\bv) }\bc(s)l_{s}', \\[3mm]
Y & = 
  \pair{ \wt_{s+1}(\bv) }{ \qwt_{s+1}^{\vee}(\bv,\bc) } - 
  \delta_{s+1 \in T^{-}(\bv)} \sgn(\gamma_{s+1})l_{s+1}' -
  \delta_{s \in S(\bv) }\bc(s)l_{s}'.
\end{align*}
Here, 
\begin{align*}
& \pair{ \wt_{s}(\bw) }{ \qwt_{s}^{\vee}(\bw,\bb) } = 
  \pair{ - l_{s}'w_{s-1}^{-1} \gamma_{s} }{ \qwt_{s-1}^{\vee}(\bw,\bb) + 
  \delta_{s \in T^{-}(\bw)}|w_{s}^{-1}\gamma_{s}|^{\vee} },  \\[3mm]
& \pair{ \wt_{s+1}(\bv) }{ \qwt_{s+1}^{\vee}(\bv,\bc) } \\
& = 
  \pair{ - \underbrace{ l_{s+1}'v_{s}^{-1} \gamma_{s+1} }_{= l_{s}'w_{s-1}^{-1}\gamma_{s} } }
  { \underbrace{ \qwt_{s-1}^{\vee}(\bv,\bc) }_{= \qwt_{s-1}^{\vee}(\bw,\bb)} + 
    \underbrace{\delta_{s+1 \in T^{-}(\bv)}|v_{s+1}^{-1}\gamma_{s+1}|^{\vee}}_{=
    \delta_{s \in T^{-}(\bw)}|w_{s}^{-1}\gamma_{s}|^{\vee} } - 
  \delta_{s \in S(\bv)} \bc(s) \underbrace{v_{s}^{-1}\gamma_{s}^{\vee}}_{=w_{s-1}^{-1}\gamma_{s}^{\vee}}} \\
& = \pair{ \wt_{s}(\bw) }{ \qwt_{s}^{\vee}(\bw,\bb) } + 
    2 \delta_{s \in S(\bv)} \bc(s)l_{s}'. 
\end{align*}
Combining these equalities, we obtain $X=Y$, as desired. 
\end{case}

\begin{case}[to be paired with Case~\ref{case2D}] \label{case3D}
Assume that $(\bw,\bb)$ satisfies 
$w_{s-1}=w_{s} \edge{|w_{s-1}^{-1}\alpha|} s_{\alpha}w_{s}=w_{s+1}$. We set 
\begin{equation*}
\begin{cases}
v_{k}:=w_{k} \qquad \text{for $0 \le k \le m$ with $k \ne s,s+1$}, & \\
\underbrace{w_{s-1}=w_{s}}_{=v_{s-1}} \edge{|w_{s-1}^{-1}\alpha|} 
\underbrace{s_{\alpha}w_{s}=w_{s+1}}_{=:v_{s}=v_{s+1}}.
\end{cases}
\end{equation*}
We see that $\bv=(v_{0},v_{1},\dots,v_{s-1},v_{s},v_{s+1},v_{s+2},\dots,v_{m}) \in 
\QW_{\lambda,w}(\Gamma)$; notice that 
\begin{center}
$S(\bw) \cap \{s,s+1\} = \emptyset$ (resp., $=\{s\}$) 
if and only if 
$S(\bv) \cap \{s,s+1\} = \emptyset$ (resp., $=\{s+1\}$).
\end{center}
We set $\bc(t):=\bb(t)$ for $t \in S(\bv) \setminus \{s,s+1\} = S(\bw) \setminus \{s,s+1\}$, 
and $\bc(s+1):=\bw(s)$ if $S(\bv) \cap \{s,s+1\} = \{s+1\}$. 
Then, $\D(\bw,\bb)=(\bv,\bc) \in \ti{\QW}_{\lambda,w}(\Gamma)$ 
satisfies \eqref{eq:D2}; we can show the equalities in \eqref{eq:D2} 
in exactly the same way as in Case~\ref{case2D}.
\end{case}

\begin{case}[to be paired with Subcase~\ref{subcase12D}] \label{case4D}
Assume that $(\bw,\bb)$ satisfies 
$w_{s-1} \edge{|w_{s-1}^{-1}\alpha|} w_{s} 
\edge{|w_{s-1}^{-1}\alpha|} w_{s+1}$; note that $w_{s+1}=w_{s-1}$. 
Also, notice that 
$S(\bw) \cap \{s,s+1\} = \emptyset$, and 
$|w_{s-1}^{-1}\alpha|$ is a simple root. We set 
\begin{equation*}
\begin{cases}
v_{k}:=w_{k} \qquad \text{for $0 \le k \le m$ with $k \ne s,s+1$}, & \\
w_{s-1} = v_{s-1} = v_{s} = v_{s+1}. 
\end{cases}
\end{equation*}
We see that $\bv=(v_{0},v_{1},\dots,v_{s-1},v_{s},v_{s+1},v_{s+2},\dots,v_{m}) \in 
\QW_{\lambda,w}(\Gamma)$; notice that 
$S(\bv) \setminus \{s,s+1\} = S(\bw) \setminus \{s,s+1\}$, and that 
$S(\bv) \cap \{s,s+1\}$ is either $\{s\}$ or $\{s+1\}$ 
since $|w_{s-1}^{-1}\alpha|$ is a simple root. We set 
\begin{equation*}
\bc(t):=
\begin{cases}
\bb(t) & \text{for $t \in S(\bv) \setminus \{s,s+1\}$}, \\[1mm]
1 & \text{for $t \in S(\bv) \cap \{s,s+1\}$}.
\end{cases}
\end{equation*}
Then, $\D(\bw,\bb)=(\bv,\bc) \in \ti{\QW}_{\lambda,w}(\Gamma)$ 
satisfies \eqref{eq:D2}; we can show the equalities in \eqref{eq:D2} 
in exactly the same way as in Subcase~\ref{subcase12D}. 
\end{case}

Now, we set 
\begin{equation*}
\ti{\QW}^{(1)}_{\lambda,w}(\Gamma):=
\bigl\{ (\bw,\bb) \in \ti{\QW}_{\lambda,w}(\Gamma) \mid 
\D(\bw,\bb) \in \ti{\QW}_{\lambda,w}(\Xi) \bigr\}. 
\end{equation*}
We can easily verify that the map $(\bw,\bb) \mapsto \D(\bw,\bb)$ 
gives a bijection from $\ti{\QW}_{\lambda,w}(\Gamma) \setminus 
\ti{\QW}^{(1)}_{\lambda,w}(\Gamma)$ onto 
$\ti{\QW}_{\lambda,w}(\Xi)$ satisfying \eqref{eq:D1}, and also 
an involution on $\ti{\QW}^{(1)}_{\lambda,w}(\Gamma)$ satisfying \eqref{eq:D2}. 
This completes the proof of Theorem~\ref{thm:D}. 

\section{Proof of Theorem~\ref{thm:YB}.} \label{sec:YB}

In what follows, we indicate that an edge 
$w \edge{\alpha} ws_{\alpha}$ in $\QBG(W)$ 
is a quantum edge (resp., Bruhat edge) 
by writing $w \Qe{\alpha} ws_{\alpha}$ 
(resp., $w \Be{\alpha} ws_{\alpha}$). 

\subsection{In type $A_{2}$.}

We give a proof of Theorem~\ref{thm:YB} 
in the case that $\pair{\alpha}{\beta^{\vee}}=
\pair{\beta}{\alpha^{\vee}}=-1$ (see Remark~\ref{rem:YB}).  
Assume that $\Gamma \in \AP(\lambda)$ is of the form
\begin{equation} \label{eq:Gamma2}
\Gamma: 
  A_{\circ}=A_{0} 
  \xrightarrow{\gamma_{1}} A_{1} 
  \xrightarrow{\gamma_{2}} \cdots 
  \xrightarrow{\gamma_{m}} A_{m}=A_{\lambda}, 
\end{equation}
with $\gamma_{s+1}=\alpha$, $\gamma_{s+2}=\alpha+\beta$, $\gamma_{s+3}=\beta$, 
i.e., 
\begin{equation*}
\cdots \edge{\gamma_{s}} A_{s} \edge{\alpha} A_{s+1} \edge{\alpha+\beta} A_{s+2} 
  \edge{\beta} A_{s+3} \edge{\gamma_{s+4}} \cdots  \qquad \text{(in $\Gamma$)}. 
\end{equation*}
Then, $\Xi=(\beta_{1},\dots,\beta_{m})$, where 
$\beta_{k}:=\gamma_{k}$ for $1 \le k \le m$ with $k \ne s+1,s+2,s+3$, and 
$\beta_{s+1}=\beta$, $\beta_{s+2}=\alpha+\beta$, $\beta_{s+3}=\alpha$; note that 
$\Xi$ is an alcove path from $A_{\circ}$ to $A_{\lambda}$ of the form:
\begin{equation}
\Xi: 
  A_{\circ}=A_{0} 
  \xrightarrow{\gamma_{1}} \cdots 
  \edge{\gamma_{s}} A_{s} \edge{\beta} B_{s+1} \edge{\alpha+\beta} B_{s+2} 
  \edge{\alpha} A_{s+3} \edge{\gamma_{s+4}} \cdots 
\xrightarrow{\gamma_{m}} A_{m}=A_{\lambda}
\end{equation}
for some alcoves $B_{s+1}$ and $B_{s+2}$; observe that (the closure of) 
$A_{s} \sqcup A_{s+1} \sqcup A_{s+2} \sqcup A_{s+3} \sqcup B_{s+1} \sqcup B_{s+2}$ 
forms a ``hexagon'' lying in $\BR \alpha \oplus \BR \beta \subset \Fh_{\BR}^{\ast}$. 
For $1 \le t \le m$, 
let $H_{\gamma_{t},l_{t}}$ (resp., $H_{\beta_{t},k_{t}}$) be 
the affine hyperplane between the $(t-1)$-th alcove and 
the $t$-th alcove in $\Gamma$ (resp., $\Xi$). Then we see that 
$l_{t}=k_{t}$ for all $1 \le t \le m$ with $t \ne s+1, s+2, s+3$, and that 
\begin{equation} \label{eq:lk1}
k_{s+3}=l_{s+1}, \qquad 
k_{s+1}=l_{s+3}, \qquad 
l_{s+2}=l_{s+1} + l_{s+3} = k_{s+1} + k_{s+3} = k_{s+2}.
\end{equation}
Recall that $l_{t}'=\pair{\lambda}{\gamma_{t}^{\vee}}-l_{t}$ for $1 \le t \le m$. 
Set $k_{t}':=\pair{\lambda}{\beta_{t}^{\vee}}-k_{t}$ for $1 \le t \le m$. 
We see that $l_{t}'=k_{t}'$ for all $1 \le t \le m$ with $t \ne s+1, s+2, s+3$, and that 
\begin{equation} \label{eq:lk2}
k_{s+3}'=l_{s+1}', \qquad 
k_{s+1}'=l_{s+3}', \qquad 
l_{s+2}'=l_{s+1}' + l_{s+3}' = k_{s+1}' + k_{s+3}' = k_{s+2}'.
\end{equation}

\setcounter{case}{0}

Now, for $\bw=(w_{0},w_{1},\dots,w_{m}) \in \QW_{\lambda,w}(\Gamma)$, 
we set 
\begin{equation} \label{eq:Theta1a}
v_{k}:=w_{k} \qquad \text{for $0 \le k \le m$ with $k \ne s+1,s+2$}.
\end{equation}
In the following, we will define $v_{s+1}$ and $v_{s+2}$ in such a way that 
$\bv:=(v_{0},v_{1},\dots,v_{m}) \in \QW_{\lambda,w}(\Gamma) \sqcup \QW_{\lambda,w}(\Xi)$. 
Note that $S(\bw) \setminus \{s+1,s+2,s+3\} = S(\bv) \setminus \{s+1,s+2,s+3\}$. 
Also, for $\bb:S(\bw) \rightarrow \{0,1\}$, we will define $\bc:S(\bv) \rightarrow \{0,1\}$ 
in such a way that
\begin{equation} \label{eq:Theta1b}
\bc|_{S(\bv) \setminus \{s+1,s+2,s+3\}} = 
\bb|_{S(\bw) \setminus \{s+1,s+2,s+3\}}, 
\end{equation}
and $\YB(\bw,\bb):=(\bv,\bc) \in 
\ti{\QW}_{\lambda,w}(\Gamma) \sqcup \ti{\QW}_{\lambda,w}(\Xi)$. 
Then we set 
\begin{equation} \label{eq:QW0}
\ti{\QW}^{(0)}_{\lambda,w}(\Gamma):=
\bigl\{ (\bw,\bb) \in \ti{\QW}_{\lambda,w}(\Gamma) \mid 
\YB(\bw,\bb) \in \ti{\QW}_{\lambda,w}(\Xi) \bigr\}.
\end{equation}
%

\begin{case} \label{case1}
Assume that $w_{s}=w_{s+1}=w_{s+2}=w_{s+3}$. 
Then we set $w_{s}=v_{s}=v_{s+1}=v_{s+2}=v_{s+3}=w_{s+3}$. 
It is obvious that $\bv=(v_{0},v_{1},\dots,v_{m}) \in \QW_{\lambda,w}(\Xi)$. 

We see that $S(\bw) \cap \{s+1,s+2,s+3\}$ is one of 
$\emptyset$, $\{s+1\}$, $\{s+2\}$, $\{s+3\}$, and $\{s+1, s+3\}$, and that
\begin{align*}
& \text{$S(\bw) \cap \{s+1,s+2,s+3\} = \emptyset$ 
  (resp., $\{s+1\}$, $\{s+2\}$, $\{s+3\}$, $\{s+1, s+3\}$)} \\
& \iff \text{$S(\bv) \cap \{s+1,s+2,s+3\} = \emptyset$ 
  (resp., $\{s+3\}$, $\{s+2\}$, $\{s+1\}$, $\{s+1, s+3\}$)}.
\end{align*}
Hence we can define $\bc(t):=\bb(2s+4-t)$ for 
$t \in S(\bv) \cap \{s+1,s+2,s+3\}$. 
In this case, $\YB(\bw,\bb)=(\bv,\bc) \in \ti{\QW}_{\lambda,w}(\Xi)$ 
satisfies \eqref{eq:YB1}. 
Indeed, we can easily show the equalities in \eqref{eq:YB1}, 
except for $\deg(\bv,\bc) = \deg(\bw,\bb)$. In order to show that $\deg(\bv,\bc) = \deg(\bw,\bb)$, 
it suffices to show that $\deg'(\bv,\bc) = \deg'(\bw,\bb)$. 
For this equality, we claim that 
\begin{equation} \label{eq:X1}
\begin{split}
X := & \overbrace{ \sum_{t=s+1,s+2,s+3} \pair{ \wt_{t}(\bw) }{ \qwt_{t}^{\vee}(\bw,\bb) } }^{=:X_{1}} \\[3mm]
& 
 - \underbrace{ \sum_{ t \in T^{-}(\bw) \cap \{s+1,s+2,s+3\} } \sgn(\gamma_{t})l_{t}' }_{=:X_{2}}
 - \underbrace{ \sum_{ t \in S(\bw) \cap \{s+1,s+2,s+3\} }\bb(t)l_{t}' }_{=:X_{3}}
\end{split}
\end{equation}
is equal to
\begin{equation} \label{eq:Y1}
\begin{split}
Y := & \overbrace{ \sum_{t=s+1,s+2,s+3} \pair{ \wt_{t}(\bv) }{ \qwt_{t}^{\vee}(\bv,\bc) } }^{=:Y_{1}} \\[3mm]
&
 - \underbrace{ \sum_{ t \in T^{-}(\bv) \cap \{s+1,s+2,s+3\} } \sgn(\beta_{t})k_{t}' }_{=:Y_{2}}
 - \underbrace{ \sum_{ t \in S(\bv) \cap \{s+1,s+2,s+3\} }\bc(t)k_{t}' }_{=:Y_{3}}.
\end{split}
\end{equation}
We see that $X_{1}=Y_{1}=0$. Also, it is easily verified by \eqref{eq:lt2} that 
$X_{2}=Y_{2}$ and $X_{3}=Y_{3}$. This proves $X=Y$, as desired. 
\end{case}

\begin{case} \label{case2}
Assume that $w_{s} \edge{|w_{s}^{-1}\alpha|} s_{\alpha}w_{s}=w_{s+1}=w_{s+2}=w_{s+3}$. 
Then, we define $v_{s+1}$ and $v_{s+2}$ by the following directed path in $\QBG(W)$: 
\begin{equation*}
(w_{s}=) \quad v_{s}=v_{s+1}=v_{s+2} \edge{|v_{s+2}^{-1}\alpha|} s_{\alpha}v_{s+2}=v_{s+3} \quad (=w_{s+3}).
\end{equation*}
It is easily checked that $\bv=(v_{0},v_{1},\dots,v_{m}) \in \QW_{\lambda,w}(\Xi)$. 

We see that $S(\bw) \cap \{s+1,s+2,s+3\}$ is one of 
$\emptyset$, $\{s+2\}$, and $\{s+3\}$, and that
\begin{align*}
& \text{$S(\bw) \cap \{s+1,s+2,s+3\} = \emptyset$ 
  (resp., $\{s+2\}$, $\{s+3\}$)} \\
& \iff \text{$S(\bv) \cap \{s+1,s+2,s+3\} = \emptyset$ 
  (resp., $\{s+1\}$, $\{s+2\}$)}.
\end{align*}
Hence we can define $\bc(t):=\bb(t+1)$ for $t \in S(\bv) \cap \{s+1,s+2,s+3\}$. 
In this case, $\YB(\bw,\bb)=(\bv,\bc) \in \ti{\QW}_{\lambda,w}(\Xi)$ satisfies \eqref{eq:YB1}.
Indeed, we can easily show the equalities in \eqref{eq:YB1}, 
except for $\deg(\bv,\bc) = \deg(\bw,\bb)$. In order to show that $\deg(\bv,\bc) = \deg(\bw,\bb)$, 
it suffices to show that $\deg'(\bv,\bc) = \deg'(\bw,\bb)$. For this equality, we claim that 
$X$ in \eqref{eq:X1} is equal to $Y$ in \eqref{eq:Y1}. We have
\begin{align*}
X_{1} & = \pair{ \wt_{s+1}(\bw) }{ \qwt_{s+1}^{\vee}(\bw,\bb) }
 = \pair{-l_{s+1}'w_{s}^{-1}\gamma_{s+1}}{%
 \qwt_{s}^{\vee}(\bw,\bb) + \delta_{s+1 \in T^{-}(\bw)}|w_{s+1}^{-1}\gamma_{s+1}|^{\vee} } \\
& = \pair{ -l_{s+1}'w_{s}^{-1}\alpha }{ \qwt_{s}^{\vee}(\bw,\bb) } + 
\pair{ -l_{s+1}'w_{s}^{-1}\alpha }{ \delta_{s+1 \in T^{-}(\bw)}|w_{s}^{-1}\alpha|^{\vee} }, \\[3mm]
Y_{1} & = \pair{ \wt_{s+3}(\bv) }{ \qwt_{s+3}^{\vee}(\bv,\bc) } \\
& = \pair{-k_{s+3}'v_{s+2}^{-1}\beta_{s+3}}{ \qwt_{s}^{\vee}(\bv,\bc) } + 
    \pair{-k_{s+3}'v_{s+2}^{-1}\beta_{s+3}}{ \delta_{s+3 \in T^{-}(\bv)}|v_{s+3}^{-1}\beta_{s+3}|^{\vee} } \\
& \qquad -
 \pair{-k_{s+3}'v_{s+2}^{-1}\beta_{s+3}}{
 \delta_{s+1 \in S(\bv)} \bc(s+1) v_{s+1}^{-1}\beta_{s+1}^{\vee} + 
 \delta_{s+2 \in S(\bv)} \bc(s+2) v_{s+2}^{-1}\beta_{s+2}^{\vee} } \\
& = \pair{ -l_{s+1}'w_{s}^{-1}\alpha }{ \qwt_{s}^{\vee}(\bw,\bb) } + 
    \pair{ -l_{s+1}'w_{s}^{-1}\alpha }{ \delta_{s+1 \in T^{-}(\bw)}|w_{s}^{-1}\alpha|^{\vee} } \\
& \qquad - 
 \pair{ -l_{s+1}'w_{s}^{-1}\alpha }{
 \delta_{s+1 \in S(\bv)} \bc(s+1) w_{s}^{-1}\beta^{\vee} + 
 \delta_{s+2 \in S(\bv)} \bc(s+2) w_{s}^{-1}(\alpha+\beta)^{\vee} } \\
& = \pair{ -l_{s+1}'w_{s}^{-1}\alpha }{ \qwt_{s}^{\vee}(\bw,\bb) } + 
    \pair{ -l_{s+1}'w_{s}^{-1}\alpha }{ \delta_{s+1 \in T^{-}(\bw)}|w_{s}^{-1}\alpha|^{\vee} } \\
& \qquad - \delta_{s+1 \in S(\bv)} \bc(s+1)l_{s+1}' + \delta_{s+2 \in S(\bv)} \bc(s+2)l_{s+1}',
\end{align*}
and hence 
\begin{equation*}
Y_{1} = X_{1} - \delta_{s+1 \in S(\bv)} \bc(s+1)l_{s+1}' + \delta_{s+2 \in S(\bv)} \bc(s+2)l_{s+1}'. 
\end{equation*}
Also, we have
\begin{align*}
X_{2} = \delta_{s+1 \in T^{-}(\bw)} \sgn(\gamma_{s+1}) l_{s+1}' 
 = \delta_{s+3 \in T^{-}(\bv)} \sgn(\beta_{s+3}) k_{s+3}' = Y_{2}, 
\end{align*}
and
\begin{align*}
X_{3} & = \delta_{s+2 \in S(\bw)} \bb(s+2) l_{s+2}' + \delta_{s+3 \in S(\bw)} \bb(s+3) l_{s+3}', \\
Y_{3} & = \delta_{s+1 \in S(\bv)} \bc(s+1) k_{s+1}' + \delta_{s+2 \in S(\bv)} \bc(s+2) k_{s+2}'. 
\end{align*}
Therefore, we deduce that 
\begin{align*}
Y & = Y_{1} - Y_{2} - Y_{3} \\
& = X_{1} - \delta_{s+1 \in S(\bv)} \bc(s+1)l_{s+1}' + \delta_{s+2 \in S(\bv)} \bc(s+2)l_{s+1}' - X_{2} \\
& \qquad - \delta_{s+1 \in S(\bv)} \bc(s+1) k_{s+1}' - \delta_{s+2 \in S(\bv)} \bc(s+2) k_{s+2}' \\
& = X_{1} - \delta_{s+2 \in S(\bw)} \bb(s+2)l_{s+1}' + \delta_{s+3 \in S(\bw)} \bb(s+3)l_{s+1}' - X_{2} \\
& \qquad - \delta_{s+2 \in S(\bw)} \bb(s+2) l_{s+3}' - \delta_{s+3 \in S(\bw)} \bb(s+3) l_{s+2}' \\
& = X_{1} - X_{2} - X_{3} = X, 
\end{align*}
as desired. 
\end{case}

\begin{case} \label{case3}
Assume that $w_{s} = w_{s+1} = w_{s+2} \edge{|w_{s}^{-1}\beta|} s_{\beta}w_{s+2}=w_{s+3}$. 
Then, we define $v_{s+1}$ and $v_{s+2}$ by the following directed path in $\QBG(W)$: 
\begin{equation*}
(w_{s}=) \quad 
v_{s} \edge{|v_{s}^{-1}\beta|} s_{\alpha}v_{s}=v_{s+1}=v_{s+2}=v_{s+3} \quad (=w_{s+3}).
\end{equation*}
It is easily checked that $\bv=(v_{0},v_{1},\dots,v_{m})=\bv \in \QW_{\lambda,w}(\Xi)$. 

We see that $S(\bw) \cap \{s+1,s+2,s+3\}$ is one of 
$\emptyset$, $\{s+1\}$, and $\{s+2\}$, and that
\begin{align*}
& \text{$S(\bw) \cap \{s+1,s+2,s+3\} = \emptyset$ 
  (resp., $\{s+1\}$, $\{s+2\}$)} \\
& \iff \text{$S(\bv) \cap \{s+1,s+2,s+3\} = \emptyset$ 
  (resp., $\{s+2\}$, $\{s+3\}$)}.
\end{align*}
Hence we can define $\bc(t):=\bb(t-1)$ for 
$t \in S(\bv) \cap \{s+1,s+2,s+3\}$. 
As in Case~\ref{case2}, we can show that 
$\YB(\bw,\bb)=(\bv,\bc) \in \ti{\QW}_{\lambda,w}(\Xi)$ satisfies \eqref{eq:YB1}. 
\end{case}

\begin{case} \label{case4}
Assume that 
$w_{s} = w_{s+1} \edge{|w_{s+1}^{-1}(\alpha+\beta)|} s_{\alpha+\beta}w_{s+1}=w_{s+2}=w_{s+3}$; 
note that $w_{s+1}^{-1}(\alpha+\beta) = w_{s}^{-1}\alpha+w_{s}^{-1}\beta$, and 
$\pair{w_{s}^{-1}\alpha}{ w_{s}^{-1}\beta^{\vee}}=\pair{w_{s}^{-1}\beta}{ w_{s}^{-1}\alpha^{\vee}}=-1$. 
Also, since $w_{s+1}^{-1} \gamma_{s+1} = w_{s}^{-1}\alpha$ and 
$w_{s+3}^{-1} \gamma_{s+3} = w_{s}^{-1}s_{\alpha+\beta} \beta = -w_{s}^{-1}\alpha$, 
we see that $S(\bw) \cap \{s+1,s+2,s+3\}$ is one of 
$\emptyset$, $\{s+1\}$, and $\{s+3\}$.  

\begin{subcase}[to be paired with Subcase~\ref{subcase61} below] \label{subcase41}
Assume that $(\bw,\bb)$ satisfies either
\begin{equation} \label{eq:c41a}
\begin{cases}
\text{$S(\bw) \cap \{s+1,s+2,s+3\} = \{s+1\}$ (i.e., $-w_{s}^{-1}\alpha$ is a simple root)}, \\[3mm]
\bb(s+1)=1, \quad w_{s}^{-1}\beta \in \Delta^{-}, \quad \text{and} \quad 
w_{s+1} \Qe{|w_{s+1}^{-1}(\alpha+\beta)|} s_{\alpha+\beta}w_{s+1}=w_{s+2},
\end{cases}
\end{equation}
or
\begin{equation} \label{eq:c41b}
\begin{cases}
\text{$S(\bw) \cap \{s+1,s+2,s+3\} = \{s+3\}$ (i.e., $w_{s}^{-1}\alpha$ is a simple root)}, \\[3mm]
\bb(s+3)=1, \quad w_{s}^{-1}\beta \in \Delta^{+}, \quad \text{and} \quad 
w_{s+1} \Qe{|w_{s+1}^{-1}(\alpha+\beta)|} s_{\alpha+\beta}w_{s+1}=w_{s+2}. 
\end{cases}
\end{equation}
It follows from Lemma~\ref{lem:qqq} (applied to $w_{s}$, $w_{s}^{-1}\alpha$, and $w_{s}^{-1}\beta$) that 
\begin{equation*}
\underbrace{w_{s}}_{=v_{s}} \Qe{|w_{s}^{-1}\alpha|} \underbrace{s_{\alpha}w_{s}}_{=:v_{s+1}} 
\Qe{|w_{s}^{-1}\beta|} \underbrace{s_{\alpha+\beta}s_{\alpha}w_{s}}_{=:v_{s+2}}
\Qe{|w_{s}^{-1}\alpha|} \underbrace{s_{\beta}s_{\alpha+\beta}s_{\alpha}w_{s}}_{=w_{s+3}=v_{s+3}} \quad 
\text{in $\QBG(W)$}, 
\end{equation*}
and that
\begin{equation*}
\sgn(w_{s}^{-1}\alpha)=\sgn(w_{s}^{-1}\beta)=\sgn(w_{s}^{-1}(\alpha+\beta)) 
= - \sgn(\alpha) = - \sgn(\beta) = - \sgn(\alpha+\beta).
\end{equation*}
We see that $\bv=(v_{0},v_{1},\dots,v_{m}) \in \QW_{\lambda,w}(\Gamma)$; notice that 
$S(\bv) \cap \{s+1,s+2,s+3\} = \emptyset$, and hence $\bc$ is defined only by \eqref{eq:Theta1b}.
Since $\sgn(w_{s}^{-1}\alpha)=\sgn(w_{s}^{-1}\beta)=\sgn(w_{s}^{-1}(\alpha+\beta))$, 
we have $(-1)^{(\bw,\bb)}=-(-1)^{(\bv,\bc)}$. In this case, 
$\YB(\bw,\bb):=(\bv,\bc) \in \ti{\QW}_{\lambda,w}(\Gamma)$ satisfies \eqref{eq:YB2}. 
Indeed, we can easily show the equalities in \eqref{eq:YB2}, 
except for $\deg(\bv,\bc) = \deg(\bw,\bb)$. In order to show that $\deg(\bv,\bc) = \deg(\bw,\bb)$, 
it suffices to show that $\deg'(\bv,\bc) = \deg'(\bw,\bb)$. For this equality, we claim that 
$X$ of \eqref{eq:X1} is equal to 
\begin{equation} \label{eq:Z1}
\begin{split}
Z := & \overbrace{ \sum_{t=s+1,s+2,s+3} \pair{ \wt_{t}(\bv) }{ \qwt_{t}^{\vee}(\bv,\bc) } }^{=:Z_{1}} \\[3mm]
&
 - \underbrace{ \sum_{ t \in T^{-}(\bv) \cap \{s+1,s+2,s+3\} } \sgn(\gamma_{t})l_{t}' }_{=:Z_{2}}
 - \underbrace{ \sum_{ t \in S(\bv) \cap \{s+1,s+2,s+3\} }\bc(t)l_{t}' }_{=:Z_{3}}.
\end{split}
\end{equation}
We give a proof only in the case that \eqref{eq:c41a} holds;
the proof in the case that \eqref{eq:c41b} holds is similar. We have 
\begin{align*}
X_{1} & = \pair{ \wt_{s+2}(\bw) }{ \qwt_{s+2}^{\vee}(\bw,\bb) } \\
& = \pair{-l_{s+2}'w_{s+1}^{-1}\gamma_{s+2}}{%
 \qwt_{s}^{\vee}(\bw,\bb) + |w_{s+2}^{-1}\gamma_{s+2}|^{\vee} - w_{s+1}^{-1}\gamma_{s+1}^{\vee} } \\
& = \pair{-l_{s+2}'w_{s}^{-1}(\alpha+\beta)}{%
 \qwt_{s}^{\vee}(\bw,\bb) + |w_{s}^{-1}(\alpha+\beta)|^{\vee} - w_{s}^{-1}\alpha^{\vee} } \\
& = \pair{-l_{s+2}'w_{s}^{-1}(\alpha+\beta)}{\qwt_{s}^{\vee}(\bw,\bb)} + 2l_{s+2}' + l_{s+2}',
\end{align*}
\begin{equation*}
X_{2} = l_{s+2}', \qquad X_{3} = l_{s+1}', 
\end{equation*}
and hence 
\begin{equation*}
X=X_{1}-X_{2}-X_{3} = 
\pair{-l_{s+2}'w_{s}^{-1}(\alpha+\beta)}{\qwt_{s}^{\vee}(\bw,\bb)} 
+ 2l_{s+2}' - l_{s+1}'.
\end{equation*}
Also, we have 
\begin{align*}
Z_{1} & = 
  \pair{ \wt_{s+1}(\bv) }{ \qwt_{s+1}^{\vee}(\bv,\bc) } + 
  \pair{ \wt_{s+2}(\bv) }{ \qwt_{s+2}^{\vee}(\bv,\bc) } + 
  \pair{ \wt_{s+3}(\bv) }{ \qwt_{s+3}^{\vee}(\bv,\bc) } \\
& = \pair{-l_{s+1}'v_{s}^{-1}\gamma_{s+1}}{%
 \qwt_{s}^{\vee}(\bv,\bc) + |v_{s+1}^{-1}\gamma_{s+1}|^{\vee} } \\
& \qquad + \pair{-l_{s+2}'v_{s+1}^{-1}\gamma_{s+2}}{%
 \qwt_{s}^{\vee}(\bv,\bc) + |v_{s+1}^{-1}\gamma_{s+1}|^{\vee} + |v_{s+2}^{-1}\gamma_{s+2}|^{\vee} } \\
& \qquad + \pair{-l_{s+3}'v_{s+2}^{-1}\gamma_{s+3}}{%
 \qwt_{s}^{\vee}(\bv,\bc) + |v_{s+1}^{-1}\gamma_{s+1}|^{\vee} + 
 |v_{s+2}^{-1}\gamma_{s+2}|^{\vee} + |v_{s+3}^{-1}\gamma_{s+3}|^{\vee} } \\
& = \pair{-l_{s+1}'w_{s}^{-1}\alpha}{%
 \qwt_{s}^{\vee}(\bw,\bb) - w_{s}^{-1}\alpha^{\vee} } \\
& \qquad + \pair{-l_{s+2}'w_{s}^{-1}\beta}{%
 \qwt_{s}^{\vee}(\bw,\bb) - w_{s}^{-1}\alpha^{\vee} - w_{s}^{-1}\beta^{\vee} } \\
& \qquad + \pair{-l_{s+3}'w_{s}^{-1}\alpha}{%
 \qwt_{s}^{\vee}(\bw,\bb) - w_{s}^{-1}\alpha^{\vee} - w_{s}^{-1}\beta^{\vee} 
 - w_{s}^{-1}\alpha^{\vee} } \\
& = \pair{-l_{s+2}'w_{s}^{-1}(\alpha+\beta)}{\qwt_{s}^{\vee}(\bw,\bb)} + 2l_{s+1}' + l_{s+2}' + 3l_{s+3}', 
\end{align*}
\begin{equation*}
Z_{2} = l_{s+1}'+l_{s+2}'+l_{s+3}', \qquad Z_{3}=0, 
\end{equation*}
and hence 
\begin{equation*}
Z=Z_{1}-Z_{2}-Z_{3} = 
\pair{-l_{s+2}'w_{s}^{-1}(\alpha+\beta)}{\qwt_{s}^{\vee}(\bw,\bb)} 
+ l_{s+1}' +2 l_{s+3}'.
\end{equation*}
Since $2l_{s+2}' - l_{s+1}' = 2(l_{s+1}' +l_{s+3}') - l_{s+1}' = 
l_{s+1}' +2 l_{s+3}'$, we obtain $X=Z$, as desired. 
\end{subcase}

\begin{subcase}[to be paired with Subcase~\ref{subcase62} below] \label{subcase42}
Assume that $(\bw,\bb)$ satisfies either
\begin{equation} \label{eq:c42a}
\begin{cases}
\text{$S(\bw) \cap \{s+1,s+2,s+3\} = \{s+1\}$ (i.e., $-w_{s}^{-1}\alpha$ is a simple root)}, \\[3mm]
\bb(s+1)=1, \quad w_{s}^{-1}\beta \in \Delta^{+}, \quad \alpha,\beta \in \Delta^{-}, \quad \text{and} \\[3mm]
w_{s+1} \Qe{|w_{s+1}^{-1}(\alpha+\beta)|} s_{\alpha+\beta}w_{s+1}=w_{s+2}, 
\end{cases}
\end{equation}
or
\begin{equation} \label{eq:c42b}
\begin{cases}
\text{$S(\bw) \cap \{s+1,s+2,s+3\} = \{s+3\}$ (i.e., $w_{s}^{-1}\alpha$ is a simple root)}, \\[3mm]
\bb(s+3)=1, \quad w_{s}^{-1}\beta \in \Delta^{-}, \quad \alpha,\beta \in \Delta^{+}, \quad \text{and} \\[3mm]
w_{s+1} \Qe{|w_{s+1}^{-1}(\alpha+\beta)|} s_{\alpha+\beta}w_{s+1}=w_{s+2}. 
\end{cases}
\end{equation}
It follows from Lemma~\ref{lem:bqbq} (applied to $w_{s}$, $w_{s}^{-1}\alpha$, and $w_{s}^{-1}\beta$) that 
\begin{equation*}
\underbrace{w_{s}}_{=v_{s}} \Be{|w_{s}^{-1}\alpha|} \underbrace{s_{\alpha}w_{s}}_{=:v_{s+1}} 
\Qe{|w_{s}^{-1}\beta|} \underbrace{s_{\alpha+\beta}s_{\alpha}w_{s}}_{=:v_{s+2}}
\Be{|w_{s}^{-1}\alpha|} \underbrace{s_{\beta}s_{\alpha+\beta}s_{\alpha}w_{s}}_{=w_{s+3}=v_{s+3}} \quad 
\text{in $\QBG(W)$}, 
\end{equation*}
and that
\begin{equation*}
\sgn(w_{s}^{-1}\alpha) = - \sgn(w_{s}^{-1}\beta)= - \sgn(w_{s}^{-1}(\alpha+\beta)) 
= \sgn(\alpha) = \sgn(\beta) = \sgn(\alpha+\beta).
\end{equation*}
We see that $\bv=(v_{0},v_{1},\dots,v_{m}) \in \QW_{\lambda,w}(\Gamma)$; notice that 
$S(\bv) \cap \{s+1,s+2,s+3\} = \emptyset$, and hence $\bc$ is defined only by \eqref{eq:Theta1b}.
In this case, $\YB(\bw,\bb):=(\bv,\bc) \in \ti{\QW}_{\lambda,w}(\Gamma)$ satisfies \eqref{eq:YB2}, 
which we can verify in exactly the same way as in Subcase~\ref{subcase41}. 
\end{subcase}

\begin{subcase}[to be paired with Subcase~\ref{subcase63} below] \label{subcase43}
Assume that $(\bw,\bb)$ satisfies either
\begin{equation} \label{eq:c43a}
\begin{cases}
\text{$S(\bw) \cap \{s+1,s+2,s+3\} = \{s+1\}$ (i.e., $-w_{s}^{-1}\alpha$ is a simple root)}, \\[3mm]
\bb(s+1)=1, \quad w_{s}^{-1}\beta \in \Delta^{+}, \quad \alpha \in \Delta^{+}, \quad \text{and} \\[3mm]
w_{s+1} \Qe{|w_{s+1}^{-1}(\alpha+\beta)|} s_{\alpha+\beta}w_{s+1}=w_{s+2}, 
\end{cases}
\end{equation}
or
\begin{equation} \label{eq:c43b}
\begin{cases}
\text{$S(\bw) \cap \{s+1,s+2,s+3\} = \{s+3\}$ (i.e., $w_{s}^{-1}\alpha$ is a simple root)}, \\[3mm]
\bb(s+3)=1, \quad w_{s}^{-1}\beta \in \Delta^{-}, \quad \alpha \in \Delta^{-}, \quad \text{and} \\[3mm]
w_{s+1} \Qe{|w_{s+1}^{-1}(\alpha+\beta)|} s_{\alpha+\beta}w_{s+1}=w_{s+2}. 
\end{cases}
\end{equation}
It follows from Lemma~\ref{lem:qbbq} that 
\begin{equation*}
\underbrace{w_{s}}_{=v_{s}} \Qe{|w_{s}^{-1}\beta|} \underbrace{s_{\beta}w_{s}}_{=:v_{s+1}} 
\Be{|w_{s}^{-1}\alpha|} \underbrace{s_{\alpha+\beta}s_{\beta}w_{s}}_{=:v_{s+2}}
\Be{|w_{s}^{-1}\beta|} \underbrace{s_{\alpha}s_{\alpha+\beta}s_{\beta}w_{s}}_{=w_{s+3}=v_{s+3}} \quad 
\text{in $\QBG(W)$},
\end{equation*}
and that 
\begin{equation*}
\sgn(w_{s}^{-1}\alpha) = - \sgn(w_{s}^{-1}\beta)= - \sgn(w_{s}^{-1}(\alpha+\beta)) 
= - \sgn(\alpha) = \sgn(\beta) = \sgn(\alpha+\beta).
\end{equation*}
We see that $\bv=(v_{0},v_{1},\dots,v_{m}) \in \QW_{\lambda,w}(\Xi)$; notice that 
$S(\bv) \cap \{s+1,s+2,s+3\} = \emptyset$, and hence $\bc$ is defined only by \eqref{eq:Theta1b}.
In this case, $\YB(\bw,\bb):=(\bv,\bc) \in \ti{\QW}_{\lambda,w}(\Xi)$ satisfies \eqref{eq:YB1}, 
which we can verify in exactly the same way as in Cases~\ref{case1} and \ref{case2}. 
\end{subcase}

\begin{subcase}[to be paired with Subcase~\ref{subcase64} below] \label{subcase44}
Assume that $(\bw,\bb)$ satisfies either
\begin{equation} \label{eq:c44a}
\begin{cases}
\text{$S(\bw) \cap \{s+1,s+2,s+3\} = \{s+1\}$ (i.e.,, $-w_{s}^{-1}\alpha$ is a simple root)}, \\[3mm]
\bb(s+1)=1, \quad w_{s}^{-1}\beta \in \Delta^{+}, \quad 
\alpha \in \Delta^{-}, \quad \beta \in \Delta^{+}, \quad \text{and} \\[3mm]
w_{s+1} \Qe{|w_{s+1}^{-1}(\alpha+\beta)|} s_{\alpha+\beta}w_{s+1}=w_{s+2}, 
\end{cases}
\end{equation}
or
\begin{equation} \label{eq:c44b}
\begin{cases}
\text{$S(\bw) \cap \{s+1,s+2,s+3\} = \{s+3\}$ (i.e., $w_{s}^{-1}\alpha$ is a simple root)}, \\[3mm]
\bb(s+3)=1, \quad w_{s}^{-1}\beta \in \Delta^{-}, \quad 
\alpha \in \Delta^{+}, \quad \beta \in \Delta^{-}, \quad \text{and} \\[3mm]
w_{s+1} \Qe{|w_{s+1}^{-1}(\alpha+\beta)|} s_{\alpha+\beta}w_{s+1}=w_{s+2}. 
\end{cases}
\end{equation}
It follows from Lemma~\ref{lem:bbqq} that 
\begin{equation*}
\underbrace{w_{s}}_{=v_{s}} \Be{|w_{s}^{-1}\beta|} \underbrace{s_{\beta}w_{s}}_{=:v_{s+1}} 
\Be{|w_{s}^{-1}\alpha|} \underbrace{s_{\alpha+\beta}s_{\beta}w_{s}}_{=:v_{s+2}}
\Qe{|w_{s}^{-1}\beta|} \underbrace{s_{\alpha}s_{\alpha+\beta}s_{\beta}w_{s}}_{=w_{s+3}=v_{s+3}} \quad 
\text{in $\QBG(W)$},
\end{equation*}
and that 
\begin{equation*}
\sgn(w_{s}^{-1}\alpha) = - \sgn(w_{s}^{-1}\beta)= - \sgn(w_{s}^{-1}(\alpha+\beta)) 
= \sgn(\alpha) = - \sgn(\beta) = \sgn(\alpha+\beta).
\end{equation*}
We see that $\bv=(v_{0},v_{1},\dots,v_{m}) \in \QW_{\lambda,w}(\Xi)$; notice that 
$S(\bv) \cap \{s+1,s+2,s+3\} = \emptyset$, and hence $\bc$ is defined only by \eqref{eq:Theta1b}.
In this case, $\YB(\bw,\bb):=(\bv,\bc) \in \ti{\QW}_{\lambda,w}(\Xi)$ satisfies \eqref{eq:YB1}, 
which we can verify in exactly the same way as in Cases~\ref{case1} and \ref{case2}. 
\end{subcase}

\begin{subcase}[to be paired with Subcase~\ref{subcase65} below] \label{subcase45}
Assume that $(\bw,\bb)$ satisfies
\begin{equation} \label{eq:c46a}
\begin{cases}
\text{$S(\bw) \cap \{s+1,s+2,s+3\} = \{s+1\}$ (i.e., $-w_{s}^{-1}\alpha$ is a simple root)}, \\[3mm]
\bb(s+1)=1, \quad \alpha \in \Delta^{-}, \quad \text{and} \quad 
w_{s+1} \Be{|w_{s+1}^{-1}(\alpha+\beta)|} s_{\alpha+\beta}w_{s+1}=w_{s+2}, 
\end{cases}
\end{equation}
or
\begin{equation} \label{eq:c46b}
\begin{cases}
\text{$S(\bw) \cap \{s+1,s+2,s+3\} = \{s+3\}$ (i.e.,, $w_{s}^{-1}\alpha$ is a simple root)}, \\[3mm]
\bb(s+3)=1, \quad \alpha \in \Delta^{+}, \quad \text{and} \quad 
w_{s+1} \Be{|w_{s+1}^{-1}(\alpha+\beta)|} s_{\alpha+\beta}w_{s+1}=w_{s+2}. 
\end{cases}
\end{equation}
It follows from Lemma~\ref{lem:bbqb} that 
\begin{equation*}
\underbrace{w_{s}}_{=v_{s}} \Be{|w_{s}^{-1}\alpha|} \underbrace{s_{\alpha}w_{s}}_{=:v_{s+1}} 
\Be{|w_{s}^{-1}\beta|} \underbrace{s_{\alpha+\beta}s_{\alpha}w_{s}}_{=:v_{s+2}}
\Qe{|w_{s}^{-1}\alpha|} \underbrace{s_{\beta}s_{\alpha+\beta}s_{\alpha}w_{s}}_{=w_{s+3}=v_{s+3}} \quad 
\text{in $\QBG(W)$}, 
\end{equation*}
and that 
\begin{equation*}
\sgn(\alpha) = \sgn(w_{s}^{-1}\alpha) = -\sgn(\beta) \quad \text{and} \quad 
\sgn(\alpha+\beta) = \sgn(w_{s}^{-1}(\alpha+\beta)) = \sgn(w_{s}^{-1}\beta). 
\end{equation*}
We see that $\bv=(v_{0},v_{1},\dots,v_{m}) \in \QW_{\lambda,w}(\Gamma)$; notice that 
$S(\bv) \cap \{s+1,s+2,s+3\} = \emptyset$, and hence $\bc$ is defined only by \eqref{eq:Theta1b}. 
In this case, $\YB(\bw,\bb):=(\bv,\bc) \in \ti{\QW}_{\lambda,w}(\Gamma)$ satisfies \eqref{eq:YB2}, 
which we can verify in exactly the same way as in Subcase~\ref{subcase41}. 
\end{subcase}

\begin{subcase}[to be paired with Subcase~\ref{subcase66} below] \label{subcase46}
Assume that $(\bw,\bb)$ satisfies
\begin{equation} \label{eq:c47a}
\begin{cases}
\text{$S(\bw) \cap \{s+1,s+2,s+3\} = \{s+1\}$ (i.e., $-w_{s}^{-1}\alpha$ is a simple root)}, \\[3mm]
\bb(s+1)=1, \quad \beta \in \Delta^{-}, \quad \text{and} \quad 
w_{s+1} \Be{|w_{s+1}^{-1}(\alpha+\beta)|} s_{\alpha+\beta}w_{s+1}=w_{s+2}, 
\end{cases}
\end{equation}
or
\begin{equation} \label{eq:c47b}
\begin{cases}
\text{$S(\bw) \cap \{s+1,s+2,s+3\} = \{s+3\}$ (i.e., $w_{s}^{-1}\alpha$ is a simple root)}, \\[3mm]
\bb(s+3)=1, \quad \beta \in \Delta^{+}, \quad \text{and} \quad 
w_{s+1} \Be{|w_{s+1}^{-1}(\alpha+\beta)|} s_{\alpha+\beta}w_{s+1}=w_{s+2}. 
\end{cases}
\end{equation}
It follows from Lemma~\ref{lem:qbbb} that 
\begin{equation*}
\underbrace{w_{s}}_{=v_{s}} \Qe{|w_{s}^{-1}\alpha|} \underbrace{s_{\alpha}w_{s}}_{=:v_{s+1}} 
\Be{|w_{s}^{-1}\beta|} \underbrace{s_{\alpha+\beta}s_{\alpha}w_{s}}_{=:v_{s+2}}
\Be{|w_{s}^{-1}\alpha|} \underbrace{s_{\beta}s_{\alpha+\beta}s_{\alpha}w_{s}}_{=w_{s+3}=v_{s+3}} \quad 
\text{in $\QBG(W)$},
\end{equation*}
and that 
\begin{equation*}
\sgn(\alpha) = - \sgn(w_{s}^{-1}\alpha) = -\sgn(\beta) \quad \text{and} \quad 
\sgn(\alpha+\beta) = \sgn(w_{s}^{-1}(\alpha+\beta)) = \sgn(w_{s}^{-1}\beta). 
\end{equation*}
We see that $\bv=(v_{0},v_{1},\dots,v_{m}) \in \QW_{\lambda,w}(\Gamma)$; notice that 
$S(\bv) \cap \{s+1,s+2,s+3\} = \emptyset$, and hence $\bc$ is defined only by \eqref{eq:Theta1b}.
In this case, $\YB(\bw,\bb):=(\bv,\bc) \in \ti{\QW}_{\lambda,w}(\Gamma)$ satisfies \eqref{eq:YB2}, 
which we can verify in exactly the same way as in Subcase~\ref{subcase41}. 
\end{subcase}

\begin{subcase}[to be paired with Subcase~\ref{subcase67} below] \label{subcase47}
Assume that $(\bw,\bb)$ satisfies
\begin{equation} \label{eq:c48a}
\begin{cases}
\text{$S(\bw) \cap \{s+1,s+2,s+3\} = \{s+1\}$ (i.e., $-w_{s}^{-1}\alpha$ is a simple root)}, \\[3mm]
\bb(s+1)=1, \quad \alpha,\beta \in \Delta^{+}, \quad \text{and} \quad 
w_{s+1} \Be{|w_{s+1}^{-1}(\alpha+\beta)|} s_{\alpha+\beta}w_{s+1}=w_{s+2}, 
\end{cases}
\end{equation}
or
\begin{equation} \label{eq:c48b}
\begin{cases}
\text{$S(\bw) \cap \{s+1,s+2,s+3\} = \{s+3\}$ (i.e., $w_{s}^{-1}\alpha$ is a simple root)}, \\[3mm]
\bb(s+3)=1, \quad \alpha,\beta \in \Delta^{-}, \quad \text{and} \quad
w_{s+1} \Be{|w_{s+1}^{-1}(\alpha+\beta)|} s_{\alpha+\beta}w_{s+1}=w_{s+2}. 
\end{cases}
\end{equation}
It follows from Lemma~\ref{lem:bqbb} that 
\begin{equation*}
\underbrace{w_{s}}_{=v_{s}} \Be{|w_{s}^{-1}\beta|} \underbrace{s_{\beta}w_{s}}_{=:v_{s+1}} 
\Qe{|w_{s}^{-1}\alpha|} \underbrace{s_{\alpha+\beta}s_{\beta}w_{s}}_{=:v_{s+2}}
\Be{|w_{s}^{-1}\beta|} \underbrace{s_{\alpha}s_{\alpha+\beta}s_{\beta}w_{s}}_{=w_{s+3}=v_{s+3}} \quad 
\text{in $\QBG(W)$},
\end{equation*}
and that 
\begin{equation*}
\sgn(w_{s}^{-1}\alpha) = - \sgn(w_{s}^{-1}\beta) = - \sgn(w_{s}^{-1}(\alpha+\beta)) =
- \sgn(\alpha) = - \sgn(\beta) = - \sgn(\alpha+\beta). 
\end{equation*}
We see that $\bv=(v_{0},v_{1},\dots,v_{m}) \in \QW_{\lambda,w}(\Xi)$; notice that 
$S(\bv) \cap \{s+1,s+2,s+3\} = \emptyset$, and hence $\bc$ is defined only by \eqref{eq:Theta1b}.
In this case, $\YB(\bw,\bb):=(\bv,\bc) \in \ti{\QW}_{\lambda,w}(\Xi)$ satisfies \eqref{eq:YB1}, 
which we can verify in exactly the same way as in Cases~\ref{case1} and \ref{case2}. 
\end{subcase}

\begin{subcase} \label{subcase48}
Assume that $(\bw,\bb)$ does not satisfy any of equations \eqref{eq:c41a}--\eqref{eq:c48b}; 
notice that $S(\bw) \cap \{s+1,s+2,s+3\} = \emptyset$ or $\bb(t) = 0$ 
for $t \in S(\bw) \cap \{s+1,s+2,s+3\}$. 
Then, we define $v_{s+1}$ and $v_{s+2}$ by the following directed path in $\QBG(W)$: 
\begin{equation*}
(w_{s}=) \quad 
v_{s} = v_{s+1} \edge{|w_{s+1}^{-1}(\alpha+\beta)|} s_{\alpha+\beta}v_{s+1}=v_{s+2}=v_{s+3} \quad (=w_{s+3}).
\end{equation*}
We see that $\bv=(v_{0},v_{1},\dots,v_{m}) \in \QW_{\lambda,w}(\Xi)$. 
If $S(\bv) \cap\{s+1,s+2,s+3\} \ne \emptyset$, then we set $\bc(t)=0$ 
for $t \in S(\bv) \cap\{s+1,s+2,s+3\}$. In this case, 
$\YB(\bw,\bb):=(\bv, \bc) \in \ti{\QW}_{\lambda,w}(\Xi)$ satisfies \eqref{eq:YB1}, 
which we can verify in exactly the same way as in Cases~\ref{case1} and \ref{case2}. 
\end{subcase}
\end{case}

\begin{case} \label{case5}
Assume that $\#\bigl\{ s+1 \le t \le s+3 \mid w_{t-1} \ne w_{t} \bigr\} = 2$; 
then the sequence 
\begin{equation} \label{eq:seq}
(w_{s},w_{s+1},w_{s+2},w_{s+3} \,;\, 
 w_{s}^{-1}\gamma_{s+1},w_{s+1}^{-1}\gamma_{s+2},w_{s+2}^{-1}\gamma_{s+3})
\end{equation}
is identical to one of the following: 
\begin{enu}
\item[(a)] 
$(w_{s},\,s_{\alpha}w_{s},\,s_{\alpha+\beta}s_{\alpha}w_{s},\,s_{\alpha+\beta}s_{\alpha}w_{s} \,;\, 
  w_{s}^{-1}\alpha,\,w_{s}^{-1}\beta,\,w_{s}^{-1}\alpha)$; 
note that $w_{s+3} = s_{\alpha}s_{\beta}w_{s}$. 

\item[(b)] 
$(w_{s},\,w_{s},\,s_{\alpha+\beta}w_{s},\,s_{\beta}s_{\alpha+\beta}w_{s} \,;\,
  w_{s}^{-1}\alpha,\,w_{s}^{-1}(\alpha+\beta),\,-w_{s}^{-1}\alpha)$;
note that $w_{s+3} = s_{\alpha}s_{\beta}w_{s}$. 

\item[(c)] 
$(w_{s},\,s_{\alpha}w_{s},\,s_{\alpha}w_{s},\,s_{\beta}s_{\alpha}w_{s} \,;\,
  w_{s}^{-1}\alpha,\,w_{s}^{-1}\beta,\,w_{s}^{-1}(\alpha+\beta))$;
note that $w_{s+3} = s_{\beta}s_{\alpha}w_{s}$.
\end{enu}
In Case 5, we make frequent uses of a fundamental fact about the existence and uniqueness of a label-increasing or label-decreasing directed path 
in the quantum Bruhat graph with respect to a fixed reflection order; see, for example, \cite[Theorem~7.3]{LNSSS1}. 

\begin{subcase} \label{subcase51}
Assume that $\sgn(w_{s}^{-1}\alpha) = \sgn(w_{s}^{-1}\beta)$. 
We fix a reflection order $\lhd$ on $\Delta^{+}$ such that 
$|w_{s}^{-1}\alpha| \lhd |w_{s}^{-1}(\alpha+\beta)| \lhd |w_{s}^{-1}\beta|$. 

\enskip

\paragraph{\bf (5.1a)}
If the sequence in \eqref{eq:seq} is of the form (a), then we have 
the label-increasing directed path 
\begin{equation} \label{eq:seq51z}
w_{s} \edge{|w_{s}^{-1}\alpha|} s_{\alpha}w_{s} 
      \edge{|w_{s}^{-1}\beta|} s_{\alpha+\beta}s_{\alpha}w_{s}=s_{\alpha}s_{\beta}w_{s}, 
\end{equation}
where $S(\bw) \cap \{s+1,s+2,s+3\}$ is either $\emptyset$ or $\{s+3\}$. 
It follows that there exists a unique label-decreasing directed path 
from $w_{s}$ to $w_{s+3}$, which is of the form:
\begin{equation} \label{eq:seq51a}
w_{s} \edge{|w_{s}^{-1}\beta|} s_{\beta}w_{s} 
      \edge{|w_{s}^{-1}(\alpha+\beta)|} s_{\alpha}s_{\beta}w_{s}=w_{s+3}
\end{equation}
or
\begin{equation} \label{eq:seq51b}
w_{s} \edge{|w_{s}^{-1}(\alpha+\beta)|} s_{\alpha+\beta}w_{s} 
      \edge{|w_{s}^{-1}\alpha|} s_{\beta}s_{\alpha+\beta}w_{s}=w_{s+3}. 
\end{equation}
If \eqref{eq:seq51a} holds, then we define
$(v_{s},v_{s+1},v_{s+2},v_{s+3}):=(w_{s},s_{\beta}w_{s},s_{\beta}w_{s},s_{\alpha}s_{\beta}w_{s})$.
We see that $\bv=(v_{0},v_{1},\dots,v_{m}) \in \QW_{\lambda,w}(\Xi)$, and that
\begin{align*}
& \text{$S(\bw) \cap \{s+1,s+2,s+3\} = \emptyset$ 
  (resp., $\{s+3\}$)} \\
& \iff \text{$S(\bv) \cap \{s+1,s+2,s+3\} = \emptyset$ 
  (resp., $\{s+2\}$)}.
\end{align*}
If $S(\bv) \cap \{s+1,s+2,s+3\} = \{s+2\}$, then we define $\bc(s+2):=\bb(s+3)$. 
In this case, $\YB(\bw,\bb):=(\bv,\bc) \in \ti{\QW}_{\lambda,w}(\Xi)$ satisfies \eqref{eq:YB1}.
Indeed, we can easily show the equalities in \eqref{eq:YB1}, 
except for $\deg(\bv,\bc) = \deg(\bw,\bb)$. In order to show that $\deg(\bv,\bc) = \deg(\bw,\bb)$, 
it suffices to show that $\deg'(\bv,\bc) = \deg'(\bw,\bb)$. 
For this equality, we claim that $X$ in \eqref{eq:X1} is equal to $Y$ in \eqref{eq:Y1}.
We set $\epsilon:=\sgn(w_{s}^{-1}\alpha) = \sgn(w_{s}^{-1}\beta)$. 
Recall that 
$\delta_{\mathsf{P}} = 1$ (resp., $=0$) if a statement $\mathsf{P}$ is true (resp., false). 
We have 
\begin{align*}
X_{1} & = 
\pair{ \wt_{s+1}(\bw) }{ \qwt_{s+1}^{\vee}(\bw,\bb) } + 
\pair{ \wt_{s+2}(\bw) }{ \qwt_{s+2}^{\vee}(\bw,\bb) } \\
& = 
 \pair{-l_{s+1}'w_{s}^{-1}\gamma_{s+1}}{%
 \qwt_{s}^{\vee}(\bw,\bb) + \delta_{s+1 \in T^{-}(\bw)} |w_{s+1}^{-1}\gamma_{s+1}|^{\vee} } \\
& \quad +
 \pair{-l_{s+2}'w_{s+1}^{-1}\gamma_{s+2}}{%
 \qwt_{s}^{\vee}(\bw,\bb) + \delta_{s+1 \in T^{-}(\bw)} |w_{s+1}^{-1}\gamma_{s+1}|^{\vee}
  + \delta_{s+2 \in T^{-}(\bw)} |w_{s+2}^{-1}\gamma_{s+2}|^{\vee}} \\
& = 
 \pair{-l_{s+1}'w_{s}^{-1}\alpha}{%
 \qwt_{s}^{\vee}(\bw,\bb) + \delta_{s+1 \in T^{-}(\bw)} |w_{s}^{-1}\alpha|^{\vee} } \\
& \quad +
 \pair{-l_{s+2}'w_{s}^{-1}\beta}{%
 \qwt_{s}^{\vee}(\bw,\bb) + \delta_{s+1 \in T^{-}(\bw)} |w_{s}^{-1}\alpha|^{\vee}
  + \delta_{s+2 \in T^{-}(\bw)} |w_{s}^{-1}\beta|^{\vee}} \\
& = 
 \pair{-l_{s+1}'w_{s}^{-1}\alpha-l_{s+2}'w_{s}^{-1}\beta}{%
 \qwt_{s}^{\vee}(\bw,\bb)} \\
& \quad - 2\delta_{s+1 \in T^{-}(\bw)}\epsilon l_{s+1}' 
 + \delta_{s+1 \in T^{-}(\bw)}\epsilon l_{s+2}' - 2\delta_{s+2 \in T^{-}(\bw)}\epsilon l_{s+2}',
\end{align*}
\begin{align*}
X_{2} & = \delta_{s+1 \in T^{-}(\bw)} \sgn(\alpha) l_{s+1}' + 
      \delta_{s+2 \in T^{-}(\bw)} \sgn(\alpha+\beta) l_{s+2}', \\
X_{3} & = \delta_{s+3 \in S(\bw)} \bb(s+3) l_{s+3}'; 
\end{align*}
recall that $X=X_{1}-X_{2}-X_{3}$. Also, we have 
\begin{align*}
Y_{1} & = 
\pair{ \wt_{s+1}(\bv) }{ \qwt_{s+1}^{\vee}(\bv,\bc) } + 
\pair{ \wt_{s+3}(\bv) }{ \qwt_{s+3}^{\vee}(\bv,\bc) } \\
& = 
 \pair{-k_{s+1}'v_{s}^{-1}\beta_{s+1}}{%
 \qwt_{s}^{\vee}(\bv,\bc) + \delta_{s+1 \in T^{-}(\bv)} |v_{s+1}^{-1}\beta_{s+1}|^{\vee} } \\
& \quad +
 \bigl\langle -k_{s+3}'v_{s+2}^{-1}\beta_{s+3},\,
 \qwt_{s}^{\vee}(\bv,\bc) + \\
& \qquad \delta_{s+1 \in T^{-}(\bv)} |v_{s+1}^{-1}\beta_{s+1}|^{\vee}
  + \delta_{s+3 \in T^{-}(\bv)} |v_{s+3}^{-1}\beta_{s+3}|^{\vee} 
  - \delta_{s+2 \in S(\bv)} \bc(s+2) v_{s+2}^{-1}\beta_{s+2}^{\vee} \bigr\rangle \\
& = 
 \pair{-k_{s+1}'w_{s}^{-1}\beta}{%
 \qwt_{s}^{\vee}(\bv,\bc) + \delta_{s+1 \in T^{-}(\bv)} |w_{s}^{-1}\beta|^{\vee} } \\
& \quad +
 \bigl\langle -k_{s+3}'w_{s}^{-1}(\alpha+\beta),\,
 \qwt_{s}^{\vee}(\bv,\bc) + \\
& \qquad \delta_{s+1 \in T^{-}(\bv)} |w_{s}^{-1}\beta|^{\vee}
  + \delta_{s+3 \in T^{-}(\bv)} |w_{s}^{-1}(\alpha+\beta)|^{\vee} 
  - \delta_{s+2 \in S(\bv)} \bc(s+2) w_{s}^{-1}\alpha^{\vee} \bigr\rangle \\
& = \pair{-k_{s+1}'w_{s}^{-1}\beta-k_{s+3}'w_{s}^{-1}(\alpha+\beta)}{\qwt_{s}^{\vee}(\bv,\bc)} \\
& \quad 
  - 2\delta_{s+1 \in T^{-}(\bv)}\epsilon k_{s+1}' 
  - \delta_{s+1 \in T^{-}(\bv)}\epsilon k_{s+3}'
  - 2\delta_{s+3 \in T^{-}(\bv)}\epsilon k_{s+3}'
  +  \delta_{s+2 \in S(\bv)} \bc(s+2) k_{s+3}',
\end{align*}
\begin{align*}
Y_{2} & = \delta_{s+1 \in T^{-}(\bv)} \sgn(\beta) k_{s+1}' + 
\delta_{s+3 \in T^{-}(\bv)} \sgn(\alpha) k_{s+3}', \\
Y_{3} & = \delta_{s+2 \in S(\bv)} \bc(s+2) k_{s+2}'; 
\end{align*}
recall that $Y=Y_{1}-Y_{2}-Y_{3}$. 
Here we note that $\qwt_{s}^{\vee}(\bw,\bb) = \qwt_{s}^{\vee}(\bv,\bc)$. 
By \eqref{eq:lk2}, we see that 
\begin{equation*}
\pair{-l_{s+1}'w_{s}^{-1}\alpha-l_{s+2}'w_{s}^{-1}\beta}{%
 \qwt_{s}^{\vee}(\bw,\bb)} =
\pair{-k_{s+1}'w_{s}^{-1}\beta-k_{s+3}'w_{s}^{-1}(\alpha+\beta)}{\qwt_{s}^{\vee}(\bv,\bc)}, 
\end{equation*}
\begin{equation*}
- \delta_{s+3 \in S(\bw)} \bb(s+3) l_{s+3}' = 
\delta_{s+2 \in S(\bv)} \bc(s+2) k_{s+3}' - \delta_{s+2 \in S(\bv)} \bc(s+2) k_{s+2}'. 
\end{equation*}
Hence, in order to show that $X=Y$, we need to show that
\begin{align}
& \delta_{s+1 \in T^{-}(\bw)} 
  (- 2\epsilon l_{s+1}'+ \epsilon l_{s+2}' - \sgn(\alpha) l_{s+1}')
  + \delta_{s+2 \in T^{-}(\bw)} 
  (- 2\epsilon l_{s+2}' - \sgn(\alpha+\beta) l_{s+2}') \nonumber \\
& = \delta_{s+1 \in T^{-}(\bv)}
  (- 2\epsilon k_{s+1}' - \epsilon k_{s+3}' - \sgn(\beta) k_{s+1}')
  + \delta_{s+3 \in T^{-}(\bv)}
  (- 2\epsilon k_{s+3}' - \sgn(\alpha) k_{s+3}'). \label{eq:51aa}
\end{align}
Since both the label-increasing directed path \eqref{eq:seq51z} and 
the label-decreasing directed path \eqref{eq:seq51a} are shortest directed paths 
from $w_{s}$ to $s_{\alpha}s_{\beta}w_{s}=s_{\beta}s_{\alpha+\beta}w_{s}$ in $\QBG(W)$, 
the sum of the labels of quantum edges in \eqref{eq:seq51z} is identical to 
that in \eqref{eq:seq51a}; see, for example, \cite[Proposition~8.1]{LNSSS1}. 
From this fact, together with the assumption that $\sgn(w_{s}^{-1}\alpha) = \sgn(w_{s}^{-1}\beta)$, 
we deduce that 
\begin{align*}
& T^{-}(\bw) \cap \{s+1,s+2,s+3\} = \emptyset, \ \{s+2\}, \text{ or } \{s+1,s+2\}, \\
& T^{-}(\bv) \cap \{s+1,s+2,s+3\} = \emptyset, \ \{s+1\}, \text{ or } \{s+3\},
\end{align*}
and that
\begin{align*}
& \text{$T^{-}(\bw) \cap \{s+1,s+2,s+3\} = \emptyset$ 
  (resp., $\{s+2\}$, $\{s+1,s+2\}$)} \\
& \iff \text{$T^{-}(\bv) \cap \{s+1,s+2,s+3\} = \emptyset$ 
  (resp., $\{s+1\}$, $\{s+3\}$)}.
\end{align*}
We show \eqref{eq:51aa} in the case that $\epsilon = 1$ and 
$T^{-}(\bw) \cap \{s+1,s+2,s+3\} = \{s+2\}$; the proofs in the other cases are similar or simpler. 
In this case, note that $\delta_{s+1 \in T^{-}(\bw)} = 0$, $\delta_{s+2 \in T^{-}(\bw)} = 1$, 
$\delta_{s+1 \in T^{-}(\bv)} = 1$, and $\delta_{s+3 \in T^{-}(\bv)} = 0$. 
Also, since $(s_{\alpha}w_{s})^{-1}(\alpha+\beta) = w_{s}^{-1}\beta \in \Delta^{+}$ and 
the edge $s_{\alpha}w_{s} \edge{|w_{s}^{-1}\beta|} s_{\alpha+\beta}s_{\alpha}w_{s}$ is a quantum edge, 
we deduce from Lemma~\ref{lem:ell} that $\alpha+\beta \in \Delta^{-}$, and hence 
$\sgn(\alpha+\beta)=-1$. Similarly, 
since $w_{s}^{-1}\beta \in \Delta^{+}$ and 
the edge $w_{s} \edge{|w_{s}^{-1}\beta|} s_{\beta}w_{s}$ is a quantum edge, 
we deduce from Lemma~\ref{lem:ell} that $\beta \in \Delta^{-}$, and hence $\sgn(\beta)=-1$. 
Thus, equation \eqref{eq:51aa} (which we need to show) follows from these equalities and \eqref{eq:lk2}. 

If \eqref{eq:seq51b} holds, then we define
$(v_{s},v_{s+1},v_{s+2},v_{s+3}):=(w_{s},w_{s},s_{\alpha+\beta}w_{s},s_{\beta}s_{\alpha+\beta}w_{s})$.
We see that $\bv=(v_{0},v_{1},\dots,v_{m}) \in \QW_{\lambda,w}(\Gamma)$, and that
\begin{align*}
& \text{$S(\bw) \cap \{s+1,s+2,s+3\} = \emptyset$ 
  (resp., $\{s+3\}$)} \\
& \iff \text{$S(\bv) \cap \{s+1,s+2,s+3\} = \emptyset$ 
  (resp., $\{s+1\}$)}.
\end{align*}
If $S(\bv) \cap \{s+1,s+2,s+3\} = \{s+1\}$, then we define $\bc(s+1):=\bb(s+3)$. 
In this case, $\YB(\bw,\bb):=(\bv,\bc) \in \ti{\QW}_{\lambda,w}(\Gamma)$ satisfies \eqref{eq:YB2}. 
Indeed, we can easily show the equalities in \eqref{eq:YB2}, 
except for $\deg(\bv,\bc) = \deg(\bw,\bb)$. In order to show that $\deg(\bv,\bc) = \deg(\bw,\bb)$, 
it suffices to show that $\deg'(\bv,\bc) = \deg'(\bw,\bb)$. 
For this equality, we claim that $X$ in \eqref{eq:X1} is equal to $Z$ in \eqref{eq:Z1}. 
As above, we have 
\begin{align*}
X_{1} & = 
 \pair{-l_{s+1}'w_{s}^{-1}\alpha-l_{s+2}'w_{s}^{-1}\beta}{%
 \qwt_{s}^{\vee}(\bw,\bb)} \\
& \quad - 2\delta_{s+1 \in T^{-}(\bw)}\epsilon l_{s+1}' 
 + \delta_{s+1 \in T^{-}(\bw)}\epsilon l_{s+2}' - 2\delta_{s+2 \in T^{-}(\bw)}\epsilon l_{s+2}', \\
X_{2} & = \delta_{s+1 \in T^{-}(\bw)} \sgn(\alpha) l_{s+1}' + 
      \delta_{s+2 \in T^{-}(\bw)} \sgn(\alpha+\beta) l_{s+2}', \\
X_{3} & = \delta_{s+3 \in S(\bw)} \bb(s+3) l_{s+3}'; 
\end{align*}
recall that $X=X_{1}-X_{2}-X_{3}$. By computations similar to the above, we have 
\begin{align*}
Z_{1}
& = \pair{ -l_{s+2}'w_{s}^{-1}(\alpha+\beta) +  l_{s+3}'w_{s}^{-1}\alpha }{ \qwt_{s}^{\vee}(\bv,\bc) } \\
& \qquad - 2\delta_{s+2 \in T^{-}(\bw)}\epsilon l_{s+2}' + \delta_{s+1 \in S(\bv)}\bc(s+1)l_{s+2}' \\
& \qquad  + \delta_{s+2 \in T^{-}(\bv)} \epsilon l_{s+3}' + 2 \delta_{s+3 \in T^{-}(\bv)} \epsilon l_{s+3}' 
  - 2 \delta_{s+1 \in S(\bv)} \bc(s+1) l_{s+3}', \\
Z_{2} & = \delta_{s+2 \in T^{-}(\bv)} \sgn(\alpha+\beta) l_{s+2}' + 
\delta_{s+3 \in T^{-}(\bv)} \sgn(\beta) l_{s+3}', \\
Z_{3} & = \delta_{s+1 \in S(\bv)} \bc(s+1) l_{s+1}'; 
\end{align*}
recall that $Z=Z_{1}-Z_{2}-Z_{3}$. Note that 
\begin{equation*}
\pair{-l_{s+1}'w_{s}^{-1}\alpha-l_{s+2}'w_{s}^{-1}\beta}{%
 \qwt_{s}^{\vee}(\bw,\bb)} =
\pair{ -l_{s+2}'w_{s}^{-1}(\alpha+\beta) +  l_{s+3}'w_{s}^{-1}\alpha }{ \qwt_{s}^{\vee}(\bv,\bc) },  
\end{equation*}
\begin{align*}
& - \delta_{s+3 \in S(\bw)} \bb(s+3) l_{s+3}' \\
& = \delta_{s+1 \in S(\bv)}\bc(s+1)l_{s+2}'- 2 \delta_{s+1 \in S(\bv)} \bc(s+1) l_{s+3}'-
    \delta_{s+1 \in S(\bv)} \bc(s+1) l_{s+1}'. 
\end{align*}
Hence, in order to show that $X=Z$, we need to show that
\begin{align}
& \delta_{s+1 \in T^{-}(\bw)} 
  (- 2\epsilon l_{s+1}'+ \epsilon l_{s+2}' - \sgn(\alpha) l_{s+1}')
  + \delta_{s+2 \in T^{-}(\bw)} 
  (- 2\epsilon l_{s+2}' - \sgn(\alpha+\beta) l_{s+2}') \nonumber \\
& = \delta_{s+2 \in T^{-}(\bw)} 
  (- 2\epsilon l_{s+2}' + \epsilon l_{s+3}' - \sgn(\alpha+\beta) l_{s+2}') + 
 \delta_{s+3 \in T^{-}(\bv)} 
  (2 \epsilon l_{s+3}' - \sgn(\beta) l_{s+3}'). \label{eq:51ab}
\end{align}
Since both the label-increasing directed path \eqref{eq:seq51z} and 
the label-decreasing directed path \eqref{eq:seq51b} are shortest directed paths 
from $w_{s}$ to $s_{\alpha}s_{\beta}w_{s}=s_{\beta}s_{\alpha+\beta}w_{s}$ in $\QBG(W)$, 
the sum of the labels of quantum edges in \eqref{eq:seq51z} is identical to 
that in \eqref{eq:seq51b}; see, for example, \cite[Proposition~8.1]{LNSSS1}. 
From this fact, together with the assumption that $\sgn(w_{s}^{-1}\alpha) = \sgn(w_{s}^{-1}\beta)$, 
we deduce that 
\begin{align*}
& T^{-}(\bw) \cap \{s+1,s+2,s+3\} = \emptyset, \ \{s+1\}, \text{ or } \{s+1,s+2\}, \\
& T^{-}(\bv) \cap \{s+1,s+2,s+3\} = \emptyset, \ \{s+3\}, \text{ or } \{s+2\},
\end{align*}
and 
\begin{align*}
& \text{$T^{-}(\bw) \cap \{s+1,s+2,s+3\} = \emptyset$ 
  (resp., $\{s+1\}$, $\{s+1,s+2\}$)} \\
& \iff \text{$T^{-}(\bv) \cap \{s+1,s+2,s+3\} = \emptyset$ 
  (resp., $\{s+3\}$, $\{s+2\}$)}.
\end{align*}
We show \eqref{eq:51ab} in the case that $\epsilon = 1$ and 
$T^{-}(\bw) \cap \{s+1,s+2,s+3\} = \{s+1\}$; the proofs in the other cases are similar or simpler. 
In this case, note that $\delta_{s+1 \in T^{-}(\bw)} = 1$, $\delta_{s+2 \in T^{-}(\bw)} = 0$, 
$\delta_{s+2 \in T^{-}(\bv)} = 0$, and $\delta_{s+3 \in T^{-}(\bv)} = 1$. 
Also, since $w_{s}^{-1}\alpha \in \Delta^{+}$ and 
the edge $w_{s} \edge{|w_{s}^{-1}\alpha|} s_{\alpha}w_{s}$ is a quantum edge, 
we deduce from Lemma~\ref{lem:ell} that $\alpha \in \Delta^{-}$, and hence 
$\sgn(\alpha)=-1$. Similarly, since $(s_{\alpha+\beta}w_{s})^{-1}(\beta) = - w_{s}^{-1}\alpha \in \Delta^{-}$ and 
the edge $s_{\alpha+\beta}w_{s} \edge{|w_{s}^{-1}\alpha|} s_{\beta}s_{\alpha+\beta}w_{s}$ is a quantum edge, 
we deduce from Lemma~\ref{lem:ell} that $\beta \in \Delta^{+}$, and hence $\sgn(\beta)=1$. 
Thus, equation \eqref{eq:51ab} (which we need to show) follows from these equalities and \eqref{eq:lk2}. 

\enskip

\paragraph{\bf (5.1b)}
If the sequence in \eqref{eq:seq} is of the form (b), then we have 
the label-decreasing directed path 
\begin{equation*}
w_{s} \edge{|w_{s}^{-1}(\alpha+\beta)|} s_{\alpha+\beta}w_{s} 
      \edge{|w_{s}^{-1}\alpha|} s_{\beta}s_{\alpha+\beta}w_{s}=w_{s+3},
\end{equation*}
where $S(\bw) \cap \{s+1,s+2,s+3\}$ is either $\emptyset$ or $\{s+1\}$. 
It follows that there exists a unique label-increasing directed path 
from $w_{s}$ to $w_{s+3}$, which is of the form:
\begin{equation} \label{eq:seq51c}
w_{s} \edge{|w_{s}^{-1}\alpha|} s_{\alpha}w_{s} 
      \edge{|w_{s}^{-1}\beta|} s_{\alpha+\beta}s_{\alpha}w_{s}=w_{s+3}. 
\end{equation}
If we define
$(v_{s},v_{s+1},v_{s+2},v_{s+3}):=
(w_{s},s_{\alpha}w_{s},s_{\alpha+\beta}s_{\alpha}w_{s},s_{\alpha+\beta}s_{\alpha}w_{s})$, 
then we see that $\bv=(v_{0},v_{1},\dots,v_{m}) \in \QW_{\lambda,w}(\Gamma)$, and that
\begin{align*}
& \text{$S(\bw) \cap \{s+1,s+2,s+3\} = \emptyset$ 
  (resp., $\{s+1\}$)} \\
& \iff \text{$S(\bv) \cap \{s+1,s+2,s+3\} = \emptyset$ 
  (resp., $\{s+3\}$)}.
\end{align*}
If $S(\bv) \cap \{s+1,s+2,s+3\} = \{s+3\}$, then we define $\bc(s+3):=\bb(s+1)$. 
In this case, $\YB(\bw,\bb):=(\bv,\bc) \in \ti{\QW}_{\lambda,w}(\Gamma)$ 
satisfies \eqref{eq:YB2}. 

\enskip

\paragraph{\bf (5.1c)}
If the sequence in \eqref{eq:seq} is of the form (c), then we have 
the label-increasing directed path 
\begin{equation*}
w_{s} \edge{|w_{s}^{-1}\alpha|} s_{\alpha}w_{s} 
      \edge{|w_{s}^{-1}(\alpha+\beta)|} s_{\beta}s_{\alpha}w_{s}=w_{s+3},
\end{equation*}
where $S(\bw) \cap \{s+1,s+2,s+3\}$ is either $\emptyset$ or $\{s+2\}$. 
It follows that there exists a unique label-decreasing directed path 
from $w_{s}$ to $w_{s+3}$, which is of the form:
\begin{equation} \label{eq:seq51d}
w_{s} \edge{|w_{s}^{-1}\beta|} s_{\beta}w_{s} 
      \edge{|w_{s}^{-1}\alpha|} s_{\alpha+\beta}s_{\beta}w_{s}=w_{s+3}. 
\end{equation}
If we define
$(v_{s},v_{s+1},v_{s+2},v_{s+3}):=
(w_{s},s_{\beta}w_{s},s_{\alpha+\beta}s_{\beta}w_{s},s_{\alpha+\beta}s_{\beta}w_{s})$, 
then we see that $\bv=(v_{0},v_{1},\dots,v_{m}) \in \QW_{\lambda,w}(\Xi)$, and that
\begin{align*}
& \text{$S(\bw) \cap \{s+1,s+2,s+3\} = \emptyset$ 
  (resp., $\{s+2\}$)} \\
& \iff \text{$S(\bv) \cap \{s+1,s+2,s+3\} = \emptyset$ 
  (resp., $\{s+3\}$)}.
\end{align*}
If $S(\bv) \cap \{s+1,s+2,s+3\} = \{s+3\}$, then we define $\bc(s+3):=\bb(s+2)$. 
In this case, $\YB(\bw,\bb):=(\bv,\bc) \in \ti{\QW}_{\lambda,w}(\Xi)$ 
satisfies \eqref{eq:YB1}. 
\end{subcase}

\begin{subcase} \label{subcase52}
Assume that $\sgn(w_{s}^{-1}\alpha) = - \sgn(w_{s}^{-1}\beta) = \sgn(w_{s}^{-1}(\alpha+\beta))$. 
We fix a reflection order $\lhd$ on $\Delta^{+}$ such that 
$|w_{s}^{-1}(\alpha+\beta)| \lhd |w_{s}^{-1}\alpha| \lhd |w_{s}^{-1}\beta|$. 

\enskip

\paragraph{\bf (5.2a)}
If the sequence in \eqref{eq:seq} is of the form (a), then we have 
the label-increasing directed path 
\begin{equation*}
w_{s} \edge{|w_{s}^{-1}\alpha|} s_{\alpha}w_{s} 
      \edge{|w_{s}^{-1}\beta|} s_{\alpha+\beta}s_{\alpha}w_{s}=s_{\alpha}s_{\beta}w_{s}, 
\end{equation*}
where $S(\bw) \cap \{s+1,s+2,s+3\}$ is either $\emptyset$ or $\{s+3\}$. 
It follows that there exists a unique label-decreasing directed path 
from $w_{s}$ to $w_{s+3}$, which is of the form:
\begin{equation} \label{eq:seq52a}
w_{s} \edge{|w_{s}^{-1}\beta|} s_{\beta}w_{s} 
      \edge{|w_{s}^{-1}(\alpha+\beta)|} s_{\alpha}s_{\beta}w_{s}=w_{s+3}. 
\end{equation}
Define
$(v_{s},v_{s+1},v_{s+2},v_{s+3}):=(w_{s},s_{\beta}w_{s},s_{\beta}w_{s},s_{\alpha}s_{\beta}w_{s})$. 
We see that $\bv=(v_{0},v_{1},\dots,v_{m}) \in \QW_{\lambda,w}(\Xi)$, and that 
\begin{align*}
& \text{$S(\bw) \cap \{s+1,s+2,s+3\} = \emptyset$ 
  (resp., $\{s+3\}$)} \\
& \iff \text{$S(\bv) \cap \{s+1,s+2,s+3\} = \emptyset$ 
  (resp., $\{s+2\}$)}.
\end{align*}
If $S(\bv) \cap \{s+1,s+2,s+3\} = \{s+2\}$, then we define $\bc(s+2):=\bb(s+3)$. 
In this case, $\YB(\bw,\bb):=(\bv,\bc) \in \ti{\QW}_{\lambda,w}(\Xi)$ 
satisfies \eqref{eq:YB1}. 

\enskip

\paragraph{\bf (5.2b)}
If the sequence in \eqref{eq:seq} is of the form (b), then we have 
the label-increasing directed path 
\begin{equation*}
w_{s} \edge{|w_{s}^{-1}(\alpha+\beta)|} s_{\alpha+\beta}w_{s} 
      \edge{|w_{s}^{-1}\alpha|} s_{\beta}s_{\alpha+\beta}w_{s}=w_{s+3},
\end{equation*}
where $S(\bw) \cap \{s+1,s+2,s+3\}$ is either $\emptyset$ or $\{s+1\}$. 
It follows that there exists a unique label-decreasing directed path 
from $w_{s}$ to $w_{s+3}$, which is of the form:
\begin{equation} \label{eq:seq52c}
w_{s} \edge{|w_{s}^{-1}\beta|} s_{\beta}w_{s} 
      \edge{|w_{s}^{-1}(\alpha+\beta)|} s_{\alpha}s_{\beta}w_{s}=w_{s+3}. 
\end{equation}
Define
$(v_{s},v_{s+1},v_{s+2},v_{s+3}):=
(w_{s},s_{\beta}w_{s},s_{\beta}w_{s},s_{\alpha}s_{\beta}w_{s})$. 
We see that $\bv=(v_{0},v_{1},\dots,v_{m}) \in \QW_{\lambda,w}(\Xi)$, and that
\begin{align*}
& \text{$S(\bw) \cap \{s+1,s+2,s+3\} = \emptyset$ 
  (resp., $\{s+1\}$)} \\
& \iff \text{$S(\bv) \cap \{s+1,s+2,s+3\} = \emptyset$ 
  (resp., $\{s+2\}$)}.
\end{align*}
If $S(\bv) \cap \{s+1,s+2,s+3\} = \{s+2\}$, then we define $\bc(s+2):=\bb(s+1)$. 
In this case, $\YB(\bw,\bb):=(\bv,\bc) \in \ti{\QW}_{\lambda,w}(\Xi)$ 
satisfies \eqref{eq:YB1}. 

\enskip

\paragraph{\bf (5.2c)}
If the sequence in \eqref{eq:seq} is of the form (c), then we have 
the label-decreasing directed path 
\begin{equation*}
w_{s} \edge{|w_{s}^{-1}\alpha|} s_{\alpha}w_{s} 
      \edge{|w_{s}^{-1}(\alpha+\beta)|} s_{\beta}s_{\alpha}w_{s}=w_{s+3},
\end{equation*}
where $S(\bw) \cap \{s+1,s+2,s+3\}$ is either $\emptyset$ or $\{s+2\}$. 
It follows that there exists a unique label-increasing directed path 
from $w_{s}$ to $w_{s+3}$, which is of the form:
\begin{equation} \label{eq:seq52d}
w_{s} \edge{|w_{s}^{-1}(\alpha+\beta)|} s_{\alpha+\beta}w_{s} 
      \edge{|w_{s}^{-1}\beta|} s_{\alpha}s_{\alpha+\beta}w_{s}=w_{s+3}. 
\end{equation}
Define
$(v_{s},v_{s+1},v_{s+2},v_{s+3}):=
(w_{s},w_{s},s_{\alpha+\beta}w_{s},s_{\alpha}s_{\alpha+\beta}w_{s})$. 
It is easily seen that 
$\bv=(v_{0},v_{1},\dots,v_{m}) \in \QW_{\lambda,w}(\Xi)$, and that 
\begin{align*}
& \text{$S(\bw) \cap \{s+1,s+2,s+3\} = \emptyset$ 
  (resp., $\{s+2\}$)} \\
& \iff \text{$S(\bv) \cap \{s+1,s+2,s+3\} = \emptyset$ 
  (resp., $\{s+1\}$)}.
\end{align*}
If $S(\bv) \cap \{s+1,s+2,s+3\} = \{s+1\}$, then we define $\bc(s+1):=\bb(s+2)$. 
In this case, $\YB(\bw,\bb):=(\bv,\bc) \in \ti{\QW}_{\lambda,w}(\Xi)$ 
satisfies \eqref{eq:YB1}. 
\end{subcase}

\begin{subcase} \label{subcase53}
Assume that $\sgn(w_{s}^{-1}\alpha) = - \sgn(w_{s}^{-1}\beta) = - \sgn(w_{s}^{-1}(\alpha+\beta))$. 
We fix a reflection order $\lhd$ on $\Delta^{+}$ such that 
$|w_{s}^{-1}(\alpha+\beta)| \lhd |w_{s}^{-1}\beta| \lhd |w_{s}^{-1}\alpha|$. 

\enskip

\paragraph{\bf (5.3a)}
If the sequence in \eqref{eq:seq} is of the form (a), then we have 
the label-decreasing directed path 
\begin{equation*}
w_{s} \edge{|w_{s}^{-1}\alpha|} s_{\alpha}w_{s} 
      \edge{|w_{s}^{-1}\beta|} s_{\alpha+\beta}s_{\alpha}w_{s}=s_{\alpha}s_{\beta}w_{s}, 
\end{equation*}
where $S(\bw) \cap \{s+1,s+2,s+3\}$ is either $\emptyset$ or $\{s+3\}$. 
It follows that there exists a unique label-increasing directed path 
from $w_{s}$ to $w_{s+3}$, which is of the form:
\begin{equation} \label{eq:seq53a}
w_{s} \edge{|w_{s}^{-1}(\alpha+\beta)|} s_{\alpha+\beta}w_{s} 
      \edge{|w_{s}^{-1}\alpha|} s_{\beta}s_{\alpha+\beta}w_{s}=w_{s+3}. 
\end{equation}
Define
$(v_{s},v_{s+1},v_{s+2},v_{s+3}):=(w_{s},w_{s},s_{\alpha+\beta}w_{s},s_{\beta}s_{\alpha+\beta}w_{s})$. 
It is easily seen that $\bv=(v_{0},v_{1},\dots,v_{m}) \in \QW_{\lambda,w}(\Gamma)$, and that
\begin{align*}
& \text{$S(\bw) \cap \{s+1,s+2,s+3\} = \emptyset$ 
  (resp., $\{s+3\}$)} \\
& \iff \text{$S(\bv) \cap \{s+1,s+2,s+3\} = \emptyset$ 
  (resp., $\{s+1\}$)}.
\end{align*}
If $S(\bv) \cap \{s+1,s+2,s+3\} = \{s+1\}$, then we define $\bc(s+1):=\bb(s+3)$. 
In this case, $\YB(\bw,\bb):=(\bv,\bc) \in \ti{\QW}_{\lambda,w}(\Gamma)$ 
satisfies \eqref{eq:YB2}. 

\enskip

\paragraph{\bf (5.3b)}
If the sequence in \eqref{eq:seq} is of the form (b), then we have 
the label-increasing directed path 
\begin{equation*}
w_{s} \edge{|w_{s}^{-1}(\alpha+\beta)|} s_{\alpha+\beta}w_{s} 
      \edge{|w_{s}^{-1}\alpha|} s_{\beta}s_{\alpha+\beta}w_{s}=w_{s+3}, 
\end{equation*}
where $S(\bw) \cap \{s+1,s+2,s+3\}$ is either $\emptyset$ or $\{s+1\}$. 
It follows that there exists a unique label-decreasing directed path 
from $w_{s}$ to $w_{s+3}$, which is of the form:
\begin{equation} \label{eq:seq53b}
w_{s} \edge{|w_{s}^{-1}\alpha|} s_{\alpha}w_{s} 
      \edge{|w_{s}^{-1}\beta|} s_{\alpha+\beta}s_{\alpha}w_{s}=s_{\alpha}s_{\beta}w_{s}
\end{equation}
or
\begin{equation} \label{eq:seq53c}
w_{s} \edge{|w_{s}^{-1}\beta|} s_{\beta}w_{s} 
      \edge{|w_{s}^{-1}(\alpha+\beta)|} s_{\alpha}s_{\beta}w_{s}=w_{s+3}. 
\end{equation}
If \eqref{eq:seq53b} holds, then we define
$(v_{s},v_{s+1},v_{s+2},v_{s+3}):=
 (w_{s},s_{\alpha}w_{s},s_{\alpha+\beta}s_{\alpha}w_{s},s_{\alpha+\beta}s_{\alpha}w_{s})$. 
We see that $\bv=(v_{0},v_{1},\dots,v_{m}) \in \QW_{\lambda,w}(\Gamma)$, and that
\begin{align*}
& \text{$S(\bw) \cap \{s+1,s+2,s+3\} = \emptyset$ 
  (resp., $\{s+1\}$)} \\
& \iff \text{$S(\bv) \cap \{s+1,s+2,s+3\} = \emptyset$ 
  (resp., $\{s+3\}$)}.
\end{align*}
If $S(\bv) \cap \{s+1,s+2,s+3\} = \{s+3\}$, then we define $\bc(s+3):=\bb(s+1)$. 
In this case, $\YB(\bw,\bb):=(\bv,\bc) \in \ti{\QW}_{\lambda,w}(\Gamma)$ 
satisfies \eqref{eq:YB2}. 

If \eqref{eq:seq53c} holds, then we define
$(v_{s},v_{s+1},v_{s+2},v_{s+3}):=(w_{s},s_{\beta}w_{s},s_{\beta}w_{s},s_{\alpha}s_{\beta}w_{s})$. 
We see that $\bv=(v_{0},v_{1},\dots,v_{m}) \in \QW_{\lambda,w}(\Xi)$, and that
\begin{align*}
& \text{$S(\bw) \cap \{s+1,s+2,s+3\} = \emptyset$ 
  (resp., $\{s+1\}$)} \\
& \iff \text{$S(\bv) \cap \{s+1,s+2,s+3\} = \emptyset$ 
  (resp., $\{s+2\}$)}.
\end{align*}
If $S(\bv) \cap \{s+1,s+2,s+3\} = \{s+2\}$, then we define $\bc(s+2):=\bb(s+1)$. 
In this case, $\YB(\bw,\bb):=(\bv,\bc) \in \ti{\QW}_{\lambda,w}(\Xi)$ 
satisfies \eqref{eq:YB1}. 

\enskip

\paragraph{\bf (5.3c)}
If the sequence in \eqref{eq:seq} is of the form (c), then we have 
the label-decreasing directed path 
\begin{equation*}
w_{s} \edge{|w_{s}^{-1}\alpha|} s_{\alpha}w_{s} 
      \edge{|w_{s}^{-1}(\alpha+\beta)|} s_{\beta}s_{\alpha}w_{s}=w_{s+3}, 
\end{equation*}
where $S(\bw) \cap \{s+1,s+2,s+3\}$ is either $\emptyset$ or $\{s+2\}$. 
It follows that there exists a unique label-increasing directed path 
from $w_{s}$ to $w_{s+3}$, which is either of the form:
\begin{equation} \label{eq:seq53d}
w_{s} \edge{|w_{s}^{-1}(\alpha+\beta)|} s_{\alpha+\beta}w_{s} 
      \edge{|w_{s}^{-1}\beta|} s_{\alpha}s_{\alpha+\beta}w_{s}=w_{s+3}
\end{equation}
or
\begin{equation} \label{eq:seq53e}
w_{s} \edge{|w_{s}^{-1}\beta|} s_{\beta}w_{s} 
      \edge{|w_{s}^{-1}\alpha|} s_{\alpha+\beta}s_{\beta}w_{s}=w_{s+3}.
\end{equation}
If \eqref{eq:seq53d} holds, then we define
$(v_{s},v_{s+1},v_{s+2},v_{s+3}):=
 (w_{s},w_{s},s_{\alpha+\beta}w_{s},s_{\alpha}s_{\alpha+\beta}w_{s})$. 
We see that $\bv=(v_{0},v_{1},\dots,v_{m}) \in \QW_{\lambda,w}(\Xi)$, and that
\begin{align*}
& \text{$S(\bw) \cap \{s+1,s+2,s+3\} = \emptyset$ 
  (resp., $\{s+2\}$)} \\
& \iff \text{$S(\bv) \cap \{s+1,s+2,s+3\} = \emptyset$ 
  (resp., $\{s+1\}$)}.
\end{align*}
If $S(\bv) \cap \{s+1,s+2,s+3\} = \{s+1\}$, then we define $\bc(s+1):=\bb(s+2)$. 
In this case, $\YB(\bw,\bb):=(\bv,\bc) \in \ti{\QW}_{\lambda,w}(\Xi)$ 
satisfies \eqref{eq:YB1}. 

If \eqref{eq:seq53e} holds, then we define
$(v_{s},v_{s+1},v_{s+2},v_{s+3}):=
 (w_{s},s_{\beta}w_{s},s_{\alpha+\beta}s_{\beta}w_{s},s_{\alpha+\beta}s_{\beta}w_{s})$. 
We see that $\bv=(v_{0},v_{1},\dots,v_{m}) \in \QW_{\lambda,w}(\Xi)$, and that 
\begin{align*}
& \text{$S(\bw) \cap \{s+1,s+2,s+3\} = \emptyset$ 
  (resp., $\{s+2\}$)} \\
& \iff \text{$S(\bv) \cap \{s+1,s+2,s+3\} = \emptyset$ 
  (resp., $\{s+3\}$)}.
\end{align*}
If $S(\bv) \cap \{s+1,s+2,s+3\} = \{s+3\}$, then we define $\bc(s+3):=\bb(s+2)$. 
In this case, $\YB(\bw,\bb):=(\bv,\bc) \in \ti{\QW}_{\lambda,w}(\Xi)$ 
satisfies \eqref{eq:YB1}. 
\end{subcase}
\end{case}

\begin{case} \label{case6}
Assume that $\#\bigl\{ s+1 \le t \le s+3 \mid w_{t-1} \ne w_{t} \bigr\} = 3$, 
i.e., $w_{s} \ne w_{s+1} \ne w_{s+2} \ne w_{s+3}$.  

\begin{subcase}[to be paired with Subcase~\ref{subcase41}] \label{subcase61}
Assume that $(\bw,\bb)$ satisfies
\begin{equation*}
w_{s} \Qe{|w_{s}^{-1}\alpha|} s_{\alpha}w_{s} \Qe{|w_{s}^{-1}\beta|} s_{\alpha+\beta}s_{\alpha}w_{s}
\Qe{|w_{s}^{-1}\alpha|} s_{\beta}s_{\alpha+\beta}s_{\alpha}w_{s}. 
\end{equation*}
It follows from Lemma~\ref{lem:qqq} that 
\begin{equation*}
\begin{cases}
w_{s} \Qe{|w_{s}^{-1}(\alpha+\beta)|} 
s_{\alpha+\beta}w_{s} = w_{s+3} \quad \text{in $\QBG(W)$}, \\[2mm]
\text{$|w_{s}^{-1}\alpha|$ is a simple root, and } \\[2mm]
\sgn(w_{s}^{-1}\alpha)=\sgn(w_{s}^{-1}\beta)=\sgn(w_{s}^{-1}(\alpha+\beta)).
\end{cases}
\end{equation*}
If we set $v_{s+1} := w_{s}$ and $v_{s+2} := s_{\alpha+\beta}w_{s}$, then 
we see that $\bv=(v_{0},v_{1},\dots,v_{m}) \in \QW_{\lambda,w}(\Gamma)$. 
Also, we have 
\begin{equation*}
S(\bv)=
 \begin{cases}
 \{s+1\} & \text{if $-w_{s}^{-1}\alpha$ is a simple root}, \\
 \{s+3\} & \text{if $w_{s}^{-1}\alpha$ is a simple root}. 
 \end{cases}
\end{equation*}
We set $\bc(t):=1$ for $t \in S(\bv) \cap \{s+1,s+2,s+3\}$. 
Then, $\YB(\bw,\bb):=(\bv,\bc) \in \ti{\QW}_{\lambda,w}(\Gamma)$ satisfies \eqref{eq:YB2}. 
\end{subcase}

\begin{subcase}[to be paired with Subcase~\ref{subcase42}] \label{subcase62}
Assume that $(\bw,\bb)$ satisfies
\begin{equation*}
w_{s} \Be{|w_{s}^{-1}\alpha|} s_{\alpha}w_{s} \Qe{|w_{s}^{-1}\beta|} s_{\alpha+\beta}s_{\alpha}w_{s}
\Be{|w_{s}^{-1}\alpha|} s_{\beta}s_{\alpha+\beta}s_{\alpha}w_{s} \quad \text{and} \quad
\ell(s_{\alpha+\beta}w_{s}) < \ell(w_{s}).
\end{equation*}
It follows from Lemma~\ref{lem:bqbq} that 
\begin{equation*}
\begin{cases}
w_{s} \Qe{|w_{s}^{-1}(\alpha+\beta)|} 
s_{\alpha+\beta}w_{s} = w_{s+3} \quad \text{in $\QBG(W)$}, \\[2mm]
\text{$|w_{s}^{-1}\alpha|$ is a simple root, and } \\[2mm]
\sgn(w_{s}^{-1}\alpha)=-\sgn(w_{s}^{-1}\beta)=-\sgn(w_{s}^{-1}(\alpha+\beta)).
\end{cases}
\end{equation*}
If we set $v_{s+1} := w_{s}$ and $v_{s+2} := s_{\alpha+\beta}w_{s}$, then 
we see that $\bv=(v_{0},v_{1},\dots,v_{m}) \in \QW_{\lambda,w}(\Gamma)$. 
Also, we have 
\begin{equation*}
S(\bv)=
 \begin{cases}
 \{s+1\} & \text{if $-w_{s}^{-1}\alpha$ is a simple root}, \\
 \{s+3\} & \text{if $w_{s}^{-1}\alpha$ is a simple root}. 
 \end{cases}
\end{equation*}
We set $\bc(t):=1$ for $t \in S(\bv) \cap \{s+1,s+2,s+3\}$. 
Then, $\YB(\bw,\bb):=(\bv,\bc) \in \ti{\QW}_{\lambda,w}(\Gamma)$ satisfies \eqref{eq:YB2}. 
\end{subcase}

\begin{subcase}[to be paired with Subcase~\ref{subcase43}] \label{subcase63}
Assume that $(\bw,\bb)$ satisfies
\begin{equation*}
w_{s} \Qe{|w_{s}^{-1}\alpha|} s_{\alpha}w_{s} \Be{|w_{s}^{-1}\beta|} s_{\alpha+\beta}s_{\alpha}w_{s}
\Be{|w_{s}^{-1}\alpha|} s_{\beta}s_{\alpha+\beta}s_{\alpha}w_{s} \quad \text{and} \quad
\ell(s_{\alpha+\beta}w_{s}) < \ell(w_{s}).
\end{equation*}
It follows from Lemma~\ref{lem:qbbq} that 
\begin{equation*}
\begin{cases}
w_{s} \Qe{|w_{s}^{-1}(\alpha+\beta)|} 
s_{\alpha+\beta}w_{s} = w_{s+3} \quad \text{in $\QBG(W)$}, \\[2mm]
\text{$|w_{s}^{-1}\beta|$ is a simple root, and } \\[2mm]
\sgn(w_{s}^{-1}\alpha)=-\sgn(w_{s}^{-1}\beta)=\sgn(w_{s}^{-1}(\alpha+\beta)).
\end{cases}
\end{equation*}
If we set $v_{s+1} := w_{s}$ and $v_{s+2} := s_{\alpha+\beta}w_{s}$, then 
we see that $\bv=(v_{0},v_{1},\dots,v_{m}) \in \QW_{\lambda,w}(\Xi)$. 
Also, we have 
\begin{equation*}
S(\bv)=
 \begin{cases}
 \{s+1\} & \text{if $-w_{s}^{-1}\beta$ is a simple root}, \\
 \{s+3\} & \text{if $w_{s}^{-1}\beta$ is a simple root}. 
 \end{cases}
\end{equation*}
We set $\bc(t):=1$ for $t \in S(\bv) \cap \{s+1,s+2,s+3\}$. 
Then, $\YB(\bw,\bb):=(\bv,\bc) \in \ti{\QW}_{\lambda,w}(\Xi)$ satisfies \eqref{eq:YB1}. 
\end{subcase}

\begin{subcase}[to be paired with Subcase~\ref{subcase44}] \label{subcase64}
Assume that $(\bw,\bb)$ satisfies
\begin{equation*}
w_{s} \Be{|w_{s}^{-1}\alpha|} s_{\alpha}w_{s} \Be{|w_{s}^{-1}\beta|} s_{\alpha+\beta}s_{\alpha}w_{s}
\Qe{|w_{s}^{-1}\alpha|} s_{\beta}s_{\alpha+\beta}s_{\alpha}w_{s} \quad \text{and} \quad
\ell(s_{\alpha+\beta}w_{s}) < \ell(w_{s}).
\end{equation*}
It follows from Lemma~\ref{lem:bbqq} that 
\begin{equation*}
\begin{cases}
w_{s} \Qe{|w_{s}^{-1}(\alpha+\beta)|} 
s_{\alpha+\beta}w_{s} = w_{s+3} \quad \text{in $\QBG(W)$}, \\[2mm]
\text{$|w_{s}^{-1}\beta|$ is a simple root, and } \\[2mm]
\sgn(w_{s}^{-1}\alpha)=-\sgn(w_{s}^{-1}\beta)=\sgn(w_{s}^{-1}(\alpha+\beta)).
\end{cases}
\end{equation*}
If we set $v_{s+1} := w_{s}$ and $v_{s+2} := s_{\alpha+\beta}w_{s}$, then 
we see that $\bv=(v_{0},v_{1},\dots,v_{m}) \in \QW_{\lambda,w}(\Xi)$. 
Also, we have 
\begin{equation*}
S(\bv)=
 \begin{cases}
 \{s+1\} & \text{if $-w_{s}^{-1}\beta$ is a simple root}, \\
 \{s+3\} & \text{if $w_{s}^{-1}\beta$ is a simple root}. 
 \end{cases}
\end{equation*}
We set $\bc(t):=1$ for $t \in S(\bv) \cap \{s+1,s+2,s+3\}$. 
Then, $\YB(\bw,\bb):=(\bv,\bc) \in \ti{\QW}_{\lambda,w}(\Xi)$ satisfies \eqref{eq:YB1}. 
\end{subcase}

\begin{subcase}[to be paired with Subcase~\ref{subcase45}] \label{subcase65}
Assume that $(\bw,\bb)$ satisfies
\begin{equation*}
w_{s} \Be{|w_{s}^{-1}\alpha|} s_{\alpha}w_{s} \Be{|w_{s}^{-1}\beta|} s_{\alpha+\beta}s_{\alpha}w_{s}
\Qe{|w_{s}^{-1}\alpha|} s_{\beta}s_{\alpha+\beta}s_{\alpha}w_{s} \quad \text{and} \quad
\ell(s_{\alpha+\beta}w_{s}) > \ell(w_{s}).
\end{equation*}
It follows from Lemma~\ref{lem:bbqb} that 
\begin{equation*}
\begin{cases}
w_{s} \Be{|w_{s}^{-1}(\alpha+\beta)|} 
s_{\alpha+\beta}w_{s} = w_{s+3} \quad \text{in $\QBG(W)$}, \\[2mm]
\text{$|w_{s}^{-1}\alpha|$ is a simple root, and } \\[2mm]
\sgn(w_{s}^{-1}\beta)=\sgn(w_{s}^{-1}(\alpha+\beta)).
\end{cases}
\end{equation*}
If we set $v_{s+1} := w_{s}$ and $v_{s+2} := s_{\alpha+\beta}w_{s}$, then 
we see that $\bv=(v_{0},v_{1},\dots,v_{m}) \in \QW_{\lambda,w}(\Gamma)$. 
Also, we have 
\begin{equation*}
S(\bv)=
 \begin{cases}
 \{s+1\} & \text{if $-w_{s}^{-1}\beta$ is a simple root}, \\
 \{s+3\} & \text{if $w_{s}^{-1}\beta$ is a simple root}. 
 \end{cases}
\end{equation*}
We set $\bc(t):=1$ for $t \in S(\bv) \cap \{s+1,s+2,s+3\}$. 
Then, $\YB(\bw,\bb):=(\bv,\bc) \in \ti{\QW}_{\lambda,w}(\Gamma)$ satisfies \eqref{eq:YB2}. 
\end{subcase}

\begin{subcase}[to be paired with Subcase~\ref{subcase46}] \label{subcase66}
Assume that $(\bw,\bb)$ satisfies
\begin{equation*}
w_{s} \Qe{|w_{s}^{-1}\alpha|} s_{\alpha}w_{s} \Be{|w_{s}^{-1}\beta|} s_{\alpha+\beta}s_{\alpha}w_{s}
\Be{|w_{s}^{-1}\alpha|} s_{\beta}s_{\alpha+\beta}s_{\alpha}w_{s} \quad \text{and} \quad
\ell(s_{\alpha+\beta}w_{s}) > \ell(w_{s}).
\end{equation*}
It follows from Lemma~\ref{lem:qbbb} that 
\begin{equation*}
\begin{cases}
w_{s} \Be{|w_{s}^{-1}(\alpha+\beta)|} 
s_{\alpha+\beta}w_{s} = w_{s+3} \quad \text{in $\QBG(W)$}, \\[2mm]
\text{$|w_{s}^{-1}\alpha|$ is a simple root, and } \\[2mm]
\sgn(w_{s}^{-1}\beta)=\sgn(w_{s}^{-1}(\alpha+\beta)).
\end{cases}
\end{equation*}
If we set $v_{s+1} := w_{s}$ and $v_{s+2} := s_{\alpha+\beta}w_{s}$, then 
we see that $\bv=(v_{0},v_{1},\dots,v_{m}) \in \QW_{\lambda,w}(\Gamma)$. 
Also, we have 
\begin{equation*}
S(\bv)=
 \begin{cases}
 \{s+1\} & \text{if $-w_{s}^{-1}\beta$ is a simple root}, \\
 \{s+3\} & \text{if $w_{s}^{-1}\beta$ is a simple root}. 
 \end{cases}
\end{equation*}
We set $\bc(t):=1$ for $t \in S(\bv) \cap \{s+1,s+2,s+3\}$. 
Then, $\YB(\bw,\bb):=(\bv,\bc) \in \ti{\QW}_{\lambda,w}(\Gamma)$ satisfies \eqref{eq:YB2}. 
\end{subcase}

\begin{subcase}[to be paired with Subcase~\ref{subcase47}] \label{subcase67}
Assume that $(\bw,\bb)$ satisfies
\begin{equation*}
w_{s} \Be{|w_{s}^{-1}\alpha|} s_{\alpha}w_{s} \Qe{|w_{s}^{-1}\beta|} s_{\alpha+\beta}s_{\alpha}w_{s}
\Be{|w_{s}^{-1}\alpha|} s_{\beta}s_{\alpha+\beta}s_{\alpha}w_{s} \quad \text{and} \quad
\ell(s_{\alpha+\beta}w_{s}) > \ell(w_{s}).
\end{equation*}
It follows from Lemma~\ref{lem:bqbb} that 
\begin{equation*}
\begin{cases}
w_{s} \Be{|w_{s}^{-1}(\alpha+\beta)|} 
s_{\alpha+\beta}w_{s} = w_{s+3} \quad \text{in $\QBG(W)$}, \\[2mm]
\text{$|w_{s}^{-1}\beta|$ is a simple root, and } \\[2mm]
\sgn(w_{s}^{-1}\beta)=\sgn(w_{s}^{-1}\beta)=\sgn(w_{s}^{-1}(\alpha+\beta)).
\end{cases}
\end{equation*}
If we set $v_{s+1} := w_{s}$ and $v_{s+2} := s_{\alpha+\beta}w_{s}$, then 
we see that $\bv=(v_{0},v_{1},\dots,v_{m}) \in \QW_{\lambda,w}(\Xi)$. 
Also, we have 
\begin{equation*}
S(\bv)=
 \begin{cases}
 \{s+1\} & \text{if $-w_{s}^{-1}\beta$ is a simple root}, \\
 \{s+3\} & \text{if $w_{s}^{-1}\beta$ is a simple root}. 
 \end{cases}
\end{equation*}
We set $\bc(t):=1$ for $t \in S(\bv) \cap \{s+1,s+2,s+3\}$. 
Then, $\YB(\bw,\bb):=(\bv,\bc) \in \ti{\QW}_{\lambda,w}(\Xi)$ satisfies \eqref{eq:YB1}. 
\end{subcase}

\begin{subcase}[to be paired with Subcase~\ref{subcase69} below] \label{subcase68}
Assume that $(\bw,\bb)$ satisfies
\begin{equation*}
w_{s} \Be{|w_{s}^{-1}\alpha|} s_{\alpha}w_{s} \Qe{|w_{s}^{-1}\beta|} s_{\alpha+\beta}s_{\alpha}w_{s}
\Qe{|w_{s}^{-1}\alpha|} s_{\beta}s_{\alpha+\beta}s_{\alpha}w_{s}.
\end{equation*}
It follows from Lemma~\ref{lem:bqq} that 
\begin{equation*}
w_{s} \Qe{|w_{s}^{-1}\beta|} s_{\beta}w_{s} \Qe{|w_{s}^{-1}\alpha|} s_{\alpha+\beta}s_{\alpha}w_{s}
\Be{|w_{s}^{-1}\beta|} s_{\alpha}s_{\alpha+\beta}s_{\beta}w_{s}, 
\end{equation*}
and that
\begin{equation*}
\sgn(w_{s}^{-1}\alpha)=\sgn(w_{s}^{-1}\beta)=\sgn(w_{s}^{-1}(\alpha+\beta))=\sgn(\alpha)=-\sgn(\beta)=-\sgn(\alpha+\beta).
\end{equation*}
If we set $v_{s+1} := s_{\beta}w_{s}$ and $v_{s+2} := s_{\alpha+\beta}s_{\alpha}w_{s}$, then 
we see that $\bv=(v_{0},v_{1},\dots,v_{m}) \in \QW_{\lambda,w}(\Xi)$; notice that 
$S(\bv) \cap \{s+1,s+2,s+3\} = \emptyset$, and hence $\bc$ is defined only by \eqref{eq:Theta1b}.
In this case, $\YB(\bw,\bb):=(\bv,\bc) \in \ti{\QW}_{\lambda,w}(\Xi)$ satisfies \eqref{eq:YB1}. 
\end{subcase}

\begin{subcase}[to be paired with Subcase~\ref{subcase68}] \label{subcase69}
Assume that $(\bw,\bb)$ satisfies
\begin{equation*}
w_{s} \Qe{|w_{s}^{-1}\alpha|} s_{\alpha}w_{s} \Qe{|w_{s}^{-1}\beta|} s_{\alpha+\beta}s_{\alpha}w_{s}
\Be{|w_{s}^{-1}\alpha|} s_{\beta}s_{\alpha+\beta}s_{\alpha}w_{s}.
\end{equation*}
It follows from Lemma~\ref{lem:bqq} that 
\begin{equation*}
w_{s} \Be{|w_{s}^{-1}\beta|} s_{\beta}w_{s} \Qe{|w_{s}^{-1}\alpha|} s_{\alpha+\beta}s_{\alpha}w_{s}
\Qe{|w_{s}^{-1}\beta|} s_{\alpha}s_{\alpha+\beta}s_{\beta}w_{s},
\end{equation*}
and that
\begin{equation*}
\sgn(w_{s}^{-1}\alpha)=\sgn(w_{s}^{-1}\beta)=\sgn(w_{s}^{-1}(\alpha+\beta))=-\sgn(\alpha)=\sgn(\beta)=-\sgn(\alpha+\beta).
\end{equation*}
If we set $v_{s+1} := s_{\beta}w_{s}$ and $v_{s+2} := s_{\alpha+\beta}s_{\alpha}w_{s}$, then 
we see that $\bv=(v_{0},v_{1},\dots,v_{m}) \in \QW_{\lambda,w}(\Xi)$; notice that 
$S(\bv) \cap \{s+1,s+2,s+3\} = \emptyset$, and hence $\bc$ is defined only by \eqref{eq:Theta1b}.
In this case, $\YB(\bw,\bb):=(\bv,\bc) \in \ti{\QW}_{\lambda,w}(\Xi)$ satisfies \eqref{eq:YB1}. 
\end{subcase}

\begin{subcase} \label{subcase610}
Assume that $(\bw,\bb)$ satisfies
\begin{equation*}
w_{s} \Be{|w_{s}^{-1}\alpha|} s_{\alpha}w_{s} \Be{|w_{s}^{-1}\beta|} s_{\alpha+\beta}s_{\alpha}w_{s}
\Be{|w_{s}^{-1}\alpha|} s_{\beta}s_{\alpha+\beta}s_{\alpha}w_{s}.
\end{equation*}
It follows from Lemma~\ref{lem:bbb} that 
\begin{equation*}
w_{s} \Be{|w_{s}^{-1}\beta|} s_{\beta}w_{s} \Be{|w_{s}^{-1}\alpha|} s_{\alpha+\beta}s_{\alpha}w_{s}
\Be{|w_{s}^{-1}\beta|} s_{\alpha}s_{\alpha+\beta}s_{\beta}w_{s},
\end{equation*}
and that 
\begin{equation*}
\sgn(w_{s}^{-1}\alpha)=\sgn(w_{s}^{-1}\beta)=\sgn(w_{s}^{-1}(\alpha+\beta))=\sgn(\alpha)=\sgn(\beta)=\sgn(\alpha+\beta).
\end{equation*}
If we set $v_{s+1} := s_{\beta}w_{s}$ and $v_{s+2} := s_{\alpha+\beta}s_{\alpha}w_{s}$, then 
we see that $\bv=(v_{0},v_{1},\dots,v_{m}) \in \QW_{\lambda,w}(\Xi)$; notice that 
$S(\bv) \cap \{s+1,s+2,s+3\} = \emptyset$, and hence $\bc$ is defined only by \eqref{eq:Theta1b}.
In this case, $\YB(\bw,\bb):=(\bv,\bc) \in \ti{\QW}_{\lambda,w}(\Xi)$ satisfies \eqref{eq:YB1}. 
\end{subcase}
\end{case}

Thus we have defined $\YB(\bw,\bb)$ for $(\bw,\bb) \in \ti{\QW}_{\lambda,w}(\Gamma)$. 
Also, we define $\YB(\bv,\bc) \in \ti{\QW}_{\lambda,w}(\Gamma) \sqcup \ti{\QW}_{\lambda,w}(\Xi)$ 
for $(\bv,\bc) \in \ti{\QW}_{\lambda,w}(\Xi)$ by interchanging $\alpha$ and $\beta$ in 
Cases~\ref{case1}--\ref{case6}, and then define $\ti{\QW}^{(0)}_{\lambda,w}(\Xi)$ as in \eqref{eq:QW0} 
(with $\Gamma$ replaced by $\Xi$). We can verify that these subsets and 
the map $(\bw,\bb) \mapsto \YB(\bw,\bb)$ satisfy conditions (1)--(3). 
This completes the proof of Theorem~\ref{thm:YB} in the case that 
$\pair{\alpha}{\beta^{\vee}}=\pair{\beta}{\alpha^{\vee}}=-1$. 

\subsection{In type $A_{1} \times A_{1}$.}

We give a proof of Theorem~\ref{thm:YB} 
in the case that $\pair{\alpha}{\beta^{\vee}}=\pair{\beta}{\alpha^{\vee}}=0$. 
Assume that $\Gamma \in \AP(\lambda)$ is of the form
\begin{equation*}
\Gamma: 
  A_{\circ}=A_{0} 
  \xrightarrow{\gamma_{1}} A_{1} 
  \xrightarrow{\gamma_{2}} \cdots 
  \xrightarrow{\gamma_{m}} A_{m}=A_{\lambda}, 
\end{equation*}
with $\gamma_{s+1}=\alpha$, $\gamma_{s+2}=\beta$, i.e., 
\begin{equation*}
\cdots \edge{\gamma_{s}} A_{s} \edge{\alpha} A_{s+1} \edge{\beta} A_{s+2} 
  \edge{\gamma_{s+3}} A_{s+3}  \cdots  \qquad \text{(in $\Gamma$)}. 
\end{equation*}
Then, $\Xi=(\beta_{1},\dots,\beta_{m})$, where 
$\beta_{k}:=\gamma_{k}$ for $1 \le k \le m$ with $k \ne s+1,s+2$, and 
$\beta_{s+1}=\beta$, $\beta_{s+2}=\alpha$; note that 
$\Xi$ is an alcove path from $A_{\circ}$ to $A_{\lambda}$ of the form:
\begin{equation*}
\Xi: 
  A_{\circ}=A_{0} 
  \xrightarrow{\gamma_{1}} \cdots 
  \edge{\gamma_{s}} A_{s} \edge{\beta} B_{s+1} \edge{\alpha} A_{s+2} 
  \edge{\gamma_{s+3}} \cdots 
\xrightarrow{\gamma_{m}} A_{m}=A_{\lambda}
\end{equation*}
for some alcove $B_{s+1}$. 

Now, for $\bw=(w_{0},w_{1},\dots,w_{m}) \in \QW_{\lambda,w}(\Gamma)$, 
we set 
\begin{equation*}
v_{k}:=w_{k} \qquad \text{for $0 \le k \le m$ with $k \ne s+1$}, 
\end{equation*}
and define $v_{s+1}$ as follows:
\begin{enumerate}
\item[(i)] if $w_{s}=w_{s+1}=w_{s+2}$, then 
we define $v_{s+1}$ by $v_{s}=v_{s+1} = v_{s+2}$; 

\item[(ii)] if $w_{s} \edge{|w_{s}^{-1}\alpha|} s_{\alpha}w_{s}=w_{s+1}=w_{s+2}$, then 
we define $v_{s+1}$ by $v_{s}=v_{s+1} \edge{|v_{s+1}^{-1}\alpha|} s_{\alpha}v_{s+1} = v_{s+2}$; 

\item[(iii)] if $w_{s} = w_{s+1} \edge{|w_{s+1}^{-1}\beta|} s_{\beta}w_{s+1}=w_{s+2}$, then 
we define $v_{s+1}$ by $v_{s} \edge{|v_{s}^{-1}\beta|} s_{\beta}v_{s} = v_{s+1} = v_{s+2}$; 

\item[(iv)] if $w_{s} \edge{|w_{s}^{-1}\alpha|} s_{\alpha}w_{s}=w_{s+1} 
\edge{|w_{s+1}^{-1}\beta|} s_{\beta}w_{s+1}=w_{s+2}$, then 
we define $v_{s+1}$ by $v_{s} \edge{|v_{s}^{-1}\beta|} s_{\beta}v_{s} = v_{s+1} 
\edge{|v_{s+1}^{-1}\alpha|} s_{\alpha}v_{s+1}=v_{s+2}$. 
\end{enumerate}
Then it follows that 
$\bv:=(v_{0},v_{1},\dots,v_{m}) \in \QW_{\lambda,w}(\Xi)$. 
Also, for $\bb:S(\bw) \rightarrow \{0,1\}$, we define $\bc:S(\bv) \rightarrow \{0,1\}$ as follows. 
Noting that $S(\bw) \setminus \{s+1,s+2\} = S(\bv) \setminus \{s+1,s+2\}$, we set 
\begin{equation}
\bc|_{S(\bv) \setminus \{s+1,s+2\}} := 
\bb|_{S(\bw) \setminus \{s+1,s+2\}}. 
\end{equation}
Notice that $s+1 \in S(\bw) \cap \{s+1,s+2\}$ if and only if 
$s+2 \in S(\bv) \cap \{s+1,s+2\}$; in this case, 
we set $\bc(s+2):=\bb(s+1)$. Similarly, 
notice that $s+2 \in S(\bw) \cap \{s+1,s+2\}$ if and only if 
$s+1 \in S(\bv) \cap \{s+1,s+2\}$; in this case, 
we set $\bc(s+1):=\bb(s+2)$. Then we see that 
$\YB(\bw,\bb):=(\bv,\bc) \in \ti{\QW}_{\lambda,w}(\Xi)$.

As in the case that $\pair{\alpha}{\beta^{\vee}}=\pair{\beta}{\alpha^{\vee}}=-1$, 
we can verify that the map $\YB$ is a bijection 
from $\ti{\QW}_{\lambda,w}^{(0)}(\Gamma):=\ti{\QW}_{\lambda,w}(\Gamma)$ to 
$\ti{\QW}_{\lambda,w}^{(0)}(\Xi):=\ti{\QW}_{\lambda,w}(\Xi)$ satisfying \eqref{eq:YB1}. 
This completes the proof of Theorem~\ref{thm:YB} in the case that 
$\pair{\alpha}{\beta^{\vee}}=\pair{\beta}{\alpha^{\vee}}=0$. 

Thus we have established Theorem~\ref{thm:YB}. 

\appendix

\section{Technical lemmas.}
\label{sec:lem}

In this section, we assume that $\Fg$ is simply-laced, 
and that $\alpha,\beta \in \Delta$ are such that 
$\pair{\alpha}{\beta^{\vee}}=\pair{\beta}{\alpha^{\vee}}=-1$.
%
%
%
\begin{lem} \label{lem:qqq}
Let $w \in W$. Then, 
\begin{equation*}
w \Qe{|\alpha|} ws_{\alpha} \Qe{|\beta|} 
ws_{\alpha}s_{\beta} \Qe{|\alpha|} ws_{\alpha}s_{\beta}s_{\alpha}
\end{equation*}
if and only if the following conditions {\rm (1)--(3)} hold\,{\rm :}
\begin{enu}
\item $|\alpha|$ is a simple root\,{\rm;} 
\item $\sgn(\alpha)=\sgn(\beta)$\,{\rm;} 
\item we have the quantum edge $w \Qe{|\alpha+\beta|} ws_{\alpha+\beta}$ in $\QBG(W)$. 
\end{enu}
In this case, 
\begin{equation*}
\sgn(w\alpha)=\sgn(w\beta)=\sgn(w(\alpha+\beta))=-\sgn(\alpha)=-\sgn(\beta)=-\sgn(\alpha+\beta).
\end{equation*}
\end{lem}

\begin{proof}
We may assume that $\alpha \in \Delta^{+}$, since 
%
%
\begin{equation} \label{eq:iff}
\begin{cases}
\pair{\alpha}{\beta^{\vee}}=\pair{\beta}{\alpha^{\vee}}=-1 \iff 
\pair{-\alpha}{-\beta^{\vee}}=\pair{-\beta}{-\alpha^{\vee}}=-1; \\[3mm]
w \Qe{|\alpha|} ws_{\alpha} \Qe{|\beta|} ws_{\alpha}s_{\beta} \Qe{|\alpha|} ws_{\alpha}s_{\beta}s_{\alpha}  \\[2mm]
\hspace*{30mm} \iff
w \Qe{|-\alpha|} ws_{-\alpha} \Qe{|-\beta|} ws_{-\alpha}s_{-\beta} \Qe{|-\alpha|} ws_{-\alpha}s_{-\beta}s_{-\alpha}; \\[3mm]
\text{condition (1) (resp., (2), (3)) holds for $\alpha$ and $\beta$} \\[2mm]
\hspace*{30mm} \iff
\text{condition (1) (resp., (2), (3)) holds for $-\alpha$ and $-\beta$}. 
\end{cases}
\end{equation}
We set $v:=ws_{\alpha}s_{\beta}s_{\alpha}=ws_{\alpha+\beta}$. 

We first prove the ``only if'' part. We have 
\begin{align}
\ell(v) & = 
\ell(w) - 2\pair{\rho}{\alpha^{\vee}} + 1 - 2\pair{\rho}{|\beta|^{\vee}} + 1 - 2\pair{\rho}{\alpha^{\vee}} + 1 \nonumber \\
& = \ell(w) - 2 \pair{\rho}{2\alpha^{\vee}+|\beta|^{\vee}} +3. \label{eq:qqq1}
\end{align}
Also, we have 
\begin{equation} \label{eq:qqq2}
\ell(v) = \ell(ws_{\alpha+\beta}) \ge \ell(w)-\ell(s_{\alpha+\beta})=\ell(w) - 2 \pair{\rho}{|\alpha+\beta|^{\vee}} + 1. 
\end{equation}
Combining \eqref{eq:qqq1} and \eqref{eq:qqq2}, we obtain
\begin{equation} \label{eq:qqq3}
 \pair{\rho}{|\alpha+\beta|^{\vee}} - \pair{\rho}{2\alpha^{\vee}+|\beta|^{\vee}} +1 \ge 0. 
\end{equation}
Suppose, for a contradiction, that 
$\beta$ is negative. If $\alpha+\beta$ is positive, then we see by \eqref{eq:qqq3} that 
$\pair{\rho}{- \alpha^{\vee}+2\beta^{\vee}} +1 \ge 0$, 
which contradicts the inequalities $\pair{\rho}{\alpha^{\vee}} \ge 1$ and 
$\pair{\rho}{\beta^{\vee}} \le -1$. If $\alpha+\beta$ is negative, then 
we see by \eqref{eq:qqq3} that $\pair{\rho}{-3\alpha^{\vee}} +1 \ge 0$, 
which also contradicts the inequality $\pair{\rho}{\alpha^{\vee}} \ge 1$. 
Hence $\beta$ is positive, which shows (2). We see by \eqref{eq:qqq3} that 
$- \pair{\rho}{\alpha^{\vee}} + 1 \ge 0$, which shows (1). 
In addition, since equality holds in the inequality $- \pair{\rho}{\alpha^{\vee}} + 1 \ge 0$, 
we deduce that the inequality in \eqref{eq:qqq2} is, in fact, equality, 
which implies (3). This proves the ``only if'' part. 

Next we prove the ``if'' part. 
Recall that $\alpha \in \Delta^{+}$, and hence $\beta \in \Delta^{+}$ by (2). 
It follows from (3) that
\begin{equation} \label{eq:qqq1a}
\ell(v) = \ell(ws_{\alpha+\beta}) = \ell(w) - 2 \pair{\rho}{\alpha^{\vee}+\beta^{\vee}} + 1. 
\end{equation}
Also, we see that 
\begin{align}
\ell(v) & = \ell(ws_{\alpha}s_{\beta}s_{\alpha}) \ge \ell(ws_{\alpha}s_{\beta}) - 2\pair{\rho}{\alpha^{\vee}} + 1 \nonumber \\
& \ge \ell(ws_{\alpha}) - 2\pair{\rho}{\beta^{\vee}} + 1 - 2\pair{\rho}{\alpha^{\vee}} + 1 \nonumber \\
& \ge \ell(w) - 2\pair{\rho}{\alpha^{\vee}} + 1 - 2\pair{\rho}{\beta^{\vee}} + 1 - 2\pair{\rho}{\alpha^{\vee}} + 1 \nonumber \\
& = \ell(w) - 2 \pair{\rho}{2\alpha^{\vee}+\beta^{\vee}} +3. \label{eq:qqq2a}
\end{align}
Combining \eqref{eq:qqq1a} and \eqref{eq:qqq2a}, we obtain
\begin{equation} \label{eq:qqq3a}
 \pair{\rho}{\alpha^{\vee}+\beta^{\vee}} - \pair{\rho}{2\alpha^{\vee}+\beta^{\vee}} +1 \le 0;
\end{equation}
the left-hand side of \eqref{eq:qqq3a} 
is equal to $- \pair{\rho}{\alpha^{\vee}} +1$, which is equal to $0$ by (1). 
Hence the inequality in \eqref{eq:qqq3a} is equality. Therefore, we see that 
all the inequalities in \eqref{eq:qqq2a} are, in fact, equalities.
This proves the ``if'' part. 

Since $w \Qe{|\alpha|} ws_{\alpha}$ and $\alpha \in \Delta^{+}$, 
it follows from Lemma~\ref{lem:ell} that $w\alpha \in \Delta^{-}$. Similarly, since
$ws_{\alpha}s_{\beta} \Qe{|\alpha|} ws_{\alpha}s_{\beta}s_{\alpha}$ 
and $\alpha \in \Delta^{+}$, it follows from Lemma~\ref{lem:ell} 
that $w\beta = ws_{\alpha}s_{\beta}(\alpha) \in \Delta^{-}$. 
Hence we obtain $\sgn(w\alpha)=\sgn(w\beta)=-\sgn(\alpha)=-\sgn(\beta)$. 
This completes the proof of the lemma. 
\end{proof}

By arguments similar to those for Lemma~\ref{lem:qqq}, 
we can prove the following lemmas. 
%
%
\begin{lem} \label{lem:bqbq}
Let $w \in W$. Then, 
\begin{equation*}
w \Be{|\alpha|} ws_{\alpha} \Qe{|\beta|} 
ws_{\alpha}s_{\beta} \Be{|\alpha|} ws_{\alpha}s_{\beta}s_{\alpha}
 \quad \text{\rm and} \quad \ell(ws_{\alpha+\beta}) < \ell(w)
\end{equation*}
if and only if the following conditions {\rm (1)--(3)} hold\,{\rm:}
\begin{enu}
\item $|\alpha|$ is a simple root\,{\rm;} 
\item $\sgn(\alpha)=\sgn(w\alpha)=-\sgn(\beta)=\sgn(w\beta)$\,{\rm;}
\item we have the quantum edge $w \Qe{|\alpha+\beta|} ws_{\alpha+\beta}$ in $\QBG(W)$. 
\end{enu}
In this case, 
\begin{equation*}
\sgn(\alpha)=-\sgn(\beta)=-\sgn(\alpha+\beta)=\sgn(w\alpha)=\sgn(w\beta)=\sgn(w(\alpha+\beta)).
\end{equation*}
\end{lem}
%
%
\begin{lem} \label{lem:bqbb}
Let $w \in W$. Then, 
\begin{equation*}
w \Be{|\beta|} ws_{\beta} \Qe{|\alpha|} 
ws_{\beta}s_{\alpha} \Be{|\beta|} ws_{\beta}s_{\alpha}s_{\beta}
\quad \text{\rm and} \quad \ell(ws_{\alpha+\beta}) > \ell(w)
\end{equation*}
if and only if the following conditions {\rm (1)--(3)} hold\,{\rm:}
\begin{enu}
\item $|\alpha|$ is a simple root\,{\rm;}
\item $\sgn(\alpha)=-\sgn(w\alpha)=-\sgn(w\beta)$\,{\rm;}
\item we have the Bruhat edge $w \Be{|\alpha+\beta|} ws_{\alpha+\beta}$ in $\QBG(W)$. 
\end{enu}
In this case, 
\begin{equation*}
\sgn(\alpha) = - \sgn(\beta) = - \sgn(\alpha+\beta) 
= - \sgn(w\alpha) = - \sgn(w\beta) = -\sgn(w(\alpha+\beta)).
\end{equation*}
\end{lem}
%
%
\begin{lem} \label{lem:qbbq}
Let $w \in W$. Then, 
\begin{equation*}
w \Qe{|\beta|} ws_{\beta} \Be{|\alpha|} 
ws_{\beta}s_{\alpha} \Be{|\beta|} ws_{\beta}s_{\alpha}s_{\beta}
 \quad \text{\rm and} \quad \ell(ws_{\alpha+\beta}) < \ell(w)
\end{equation*}
if and only if the following conditions {\rm (1)--(3)} hold\,{\rm:}
\begin{enu}
\item $|\alpha|$ is a simple root\,{\rm;}
\item $\sgn(\alpha)=-\sgn(\beta)=-\sgn(w\alpha)$\,{\rm;} 
\item we have the quantum edge $w \Qe{|\alpha+\beta|} ws_{\alpha+\beta}$ in $\QBG(W)$. 
\end{enu}
In this case, 
\begin{equation*}
\sgn(\alpha)=-\sgn(\beta)=-\sgn(\alpha+\beta)=-\sgn(w\alpha)=\sgn(w\beta)=\sgn(w(\alpha+\beta)).
\end{equation*}
\end{lem}
%
%
\begin{lem} \label{lem:bbqq}
Let $w \in W$. Then, 
\begin{equation*}
w \Be{|\beta|} ws_{\beta} \Be{|\alpha|} 
ws_{\beta}s_{\alpha} \Qe{|\beta|} ws_{\beta}s_{\alpha}s_{\beta}
 \quad \text{\rm and} \quad \ell(ws_{\alpha+\beta}) < \ell(w)
\end{equation*}
if and only if the following conditions {\rm (1)--(3)} hold\,{\rm:}
\begin{enu}
\item $|\alpha|$ is a simple root\,{\rm;}
\item $\sgn(\alpha)=-\sgn(\beta)=\sgn(w\alpha)$\,{\rm;}
\item we have the quantum edge $w \Qe{|\alpha+\beta|} ws_{\alpha+\beta}$ in $\QBG(W)$. 
\end{enu}
In this case, 
\begin{equation*}
\sgn(\alpha)=-\sgn(\beta)=-\sgn(\alpha+\beta)=\sgn(w\alpha)=-\sgn(w\beta)=\sgn(w(\alpha+\beta)).
\end{equation*}
\end{lem}
%
%
\begin{lem} \label{lem:bbqb}
Let $w \in W$. Then, 
\begin{equation*}
w \Be{|\alpha|} ws_{\alpha} \Be{|\beta|} 
ws_{\alpha}s_{\beta} \Qe{|\alpha|} ws_{\alpha}s_{\beta}s_{\alpha}
\quad \text{\rm and} \quad \ell(ws_{\alpha+\beta}) > \ell(w)
\end{equation*}
if and only if the following conditions {\rm (1)--(3)} hold\,{\rm:}
\begin{enu}
\item $|\alpha|$ is a simple root\,{\rm;}
\item $\sgn(\alpha)=\sgn(w\alpha)$\,{\rm;}
\item we have the Bruhat edge $w \Be{|\alpha+\beta|} ws_{\alpha+\beta}$ in $\QBG(W)$. 
\end{enu}
In this case,
\begin{equation*}
\sgn(w\alpha) = \sgn(\alpha)= - \sgn(w\beta) \quad \text{\rm and} \quad 
\sgn(w(\alpha+\beta))=\sgn(\alpha+\beta)=\sgn(\beta).
\end{equation*}
\end{lem}
%
%
\begin{lem} \label{lem:qbbb}
Let $w \in W$. Then, 
\begin{equation*}
w \Qe{|\alpha|} ws_{\alpha} \Be{|\beta|} 
ws_{\alpha}s_{\beta} \Be{|\alpha|} ws_{\alpha}s_{\beta}s_{\alpha}
\quad \text{\rm and} \quad \ell(ws_{\alpha+\beta}) > \ell(w)
\end{equation*}
if and only if the following conditions {\rm (1)--(3)} hold\,{\rm:} 
\begin{enu}
\item $|\alpha|$ is a simple root\,{\rm;}

\item $\sgn(\alpha)=\sgn(w\beta)$\,{\rm;}

\item we have the Bruhat edge $w \Be{|\alpha+\beta|} ws_{\alpha+\beta}$ in $\QBG(W)$. 
\end{enu}
In this case, 
\begin{equation*}
\sgn(w\alpha) = - \sgn(\alpha)= - \sgn(w\beta) \quad \text{\rm and} \quad 
\sgn(w(\alpha+\beta))=\sgn(\alpha+\beta)=\sgn(\beta).
\end{equation*}
\end{lem}
%
%
\begin{lem} \label{lem:qbq}
For any $w \in W$, we do not have the following directed path in $\QBG(W)${\rm :} 
\begin{equation} \label{eq:qbq}
w \Qe{|\alpha|} ws_{\alpha} \Be{|\beta|} ws_{\alpha}s_{\beta} \Qe{|\alpha|} ws_{\alpha}s_{\beta}s_{\alpha}. 
\end{equation}
\end{lem}
%
%
\begin{lem} \label{lem:bqq}
Let $w \in W$. Then, 
\begin{equation*}
w \Be{|\alpha|} ws_{\alpha} \Qe{|\beta|} 
ws_{\alpha}s_{\beta} \Qe{|\alpha|} ws_{\alpha}s_{\beta}s_{\alpha}
\end{equation*}
if and only if 
\begin{equation*}
w \Qe{|\beta|} ws_{\beta} \Qe{|\alpha|} 
ws_{\beta}s_{\alpha} \Be{|\beta|} ws_{\beta}s_{\alpha}s_{\beta}. 
\end{equation*}
In this case, 
\begin{equation*}
\sgn(\alpha)=\sgn(\beta)=\sgn(\alpha+\beta)=\sgn(w\alpha)=-\sgn(w\beta)=-\sgn(w(\alpha+\beta)).
\end{equation*}
\end{lem}
%
%
\begin{lem} \label{lem:bbb}
Let $w \in W$. Then, 
\begin{equation*}
w \Be{|\alpha|} ws_{\alpha} \Be{|\beta|} 
ws_{\alpha}s_{\beta} \Be{|\alpha|} ws_{\alpha}s_{\beta}s_{\alpha}
\end{equation*}
if and only if 
\begin{equation*}
w \Be{|\beta|} ws_{\beta} \Be{|\alpha|} 
ws_{\beta}s_{\alpha} \Be{|\beta|} ws_{\beta}s_{\alpha}s_{\beta}. 
\end{equation*}
In this case, 
\begin{equation*}
\sgn(\alpha)=\sgn(\beta)=\sgn(\alpha+\beta)=\sgn(w\alpha)=\sgn(w\beta)=\sgn(w(\alpha+\beta)).
\end{equation*}
\end{lem}

\section{An example.}
\label{sec:example}

In this appendix, we assume that $\Fg$ is of type $A_{2}$, 
i.e., $\Fg=\Fsl_{3}(\BC)$. Applying Theorem~\ref{thm:main} to the case 
that $\lambda=\vpi_{1}+\vpi_{2}$, $\Gamma=(\theta,\alpha_{2},\theta,\alpha_{1}) \in \RAP(\vpi_{1}+\vpi_{2})$, 
and $w = \lng$, we obtain in $\BK \subset \KQGr$,
\begin{align}
& \be^{\vpi_{1}+\vpi_{2}} \cdot \OQG{w_{\circ}} 
  = \OQGL{w_{\circ}}{-\vpi_{1}-\vpi_{2}} \nonumber \\
& + q^{2}\Bigl(
  \OQGL{ t_{\theta^{\vee}} }{ \vpi_{1}+\vpi_{2} }  - 
  \OQGL{ s_{2}t_{\theta^{\vee}} }{ \vpi_{1}+\vpi_{2} } \nonumber \\
& \qquad - \OQGL{ s_{1}t_{\theta^{\vee}} }{ \vpi_{1}+\vpi_{2} } 
  + \OQGL{ s_{2}s_{1}t_{\theta^{\vee}} }{ \vpi_{1}+\vpi_{2} } \nonumber \\
& \qquad + \OQGL{ s_{1}s_{2}t_{\theta^{\vee}} }{ \vpi_{1}+\vpi_{2} } 
  - \OQGL{ w_{\circ}t_{\theta^{\vee}} }{ \vpi_{1}+\vpi_{2} } \Bigr) \nonumber \\
& + q \OQGL{ s_{2}s_{1}t_{\alpha_{2}^{\vee}} }{ \vpi_{1}-2\vpi_{2} }
  - q \OQGL{ w_{\circ}t_{\alpha_{2}^{\vee}} }{ \vpi_{1}-2\vpi_{2} } \nonumber \\
& + q \OQGL{ s_{1}s_{2}t_{\alpha_{1}^{\vee}} }{ -2\vpi_{1}+\vpi_{2} }
  - q \OQGL{ w_{\circ}t_{\alpha_{1}^{\vee}} }{ -2\vpi_{1}+\vpi_{2} } \nonumber \\
& + q \OQG{ t_{\theta^{\vee}} }
  - q \OQG{ s_{2}s_{1}t_{\theta^{\vee}} }
  - q \OQG{ s_{1}s_{2}t_{\theta^{\vee}} }
  + q \OQG{ w_{\circ}t_{\theta^{\vee}} }. \label{eq22}
\end{align}
Indeed, observe that 
\begin{equation*}
l_{1}=l_{2}=1,\ l_{3}=2,\ l_{4}=1, \qquad 
l_{1}'=1,\ l_{2}'=l_{3}'=l_{4}'=0;
\end{equation*}
for the definitions of $l_{t}$ and $l_{t}'$, see Section~\ref{subsec:alcove}. 
Also, we see that 
\begin{equation*}
\begin{split}
& \QW_{\vpi_{1}+\vpi_{2},\lng} = \\
& \bigl\{ \bw_{1}=(\lng,\lng,\lng,\lng,\lng),\ \bw_{2}=(\lng,e,e,e,e),\ 
\bw_{3}=(\lng,e,s_{2},s_{2},s_{2}), \\
& 
\bw_{4}=(\lng,e,s_{2},s_{2}s_{1},s_{2}s_{1}),\ 
\bw_{5}=(\lng,e,s_{2},s_{2}s_{1},\lng),\ 
\bw_{6}=(\lng,e,s_{2},s_{2},s_{1}s_{2}), \\
& \bw_{7}=(\lng,e,e,e,s_{1}),\ 
  \bw_{8}=(\lng,\lng,s_{1}s_{2},s_{1}s_{2},s_{1}s_{2}),\ 
  \bw_{9}=(\lng,\lng,s_{1}s_{2},s_{1},s_{1}), \\ 
& \bw_{10}=(\lng,\lng,s_{1}s_{2},s_{1},e),\ 
  \bw_{11}=(\lng,\lng,\lng,e,e),\ 
  \bw_{12}=(\lng,\lng,\lng,e,s_{1}), \\
& \bw_{13}=(\lng,\lng,\lng,\lng,s_{2}s_{1}) \bigr\}, 
\end{split}
\end{equation*}
and that
\begin{equation*}
S(\bw_{j})=
 \begin{cases}
 \{2,4\} & \text{if $j=1$}, \\
 \{3\} & \text{if $j=8$}, \\
 \{4\} & \text{if $j=9$}, \\
 \{2\} & \text{if $j=11,12,13$}, \\
 \emptyset & \text{otherwise}; 
 \end{cases}
\end{equation*}
observe that $\# \ti{\QW}_{\vpi_{1}+\vpi_{2},\lng}=21$, which is greater than $15$, 
the number of terms on the right-hand side of \eqref{eq22}. 
Here, recall the definition of $\bG(\bw,\bb)$ 
for $(\bw,\bb) \in \ti{\QW}_{\lambda,w}$ from Theorem~\ref{thm:main}. 
Let us first compute $\bG(\bw,\bb)$ for $\bw=\bw_{10},\,\bw_{11}$.  
Note that $S(\bw_{10})=\emptyset$ and $S(\bw_{11})=\{2\}$. We have 
\begin{equation*}
(-1)^{(\bw_{10},\emptyset)} = 1, \qquad 
(-1)^{(\bw_{11}, 2 \mapsto 0)}= 1, \qquad
(-1)^{(\bw_{11}, 2 \mapsto 1)}=-1,
\end{equation*}
\begin{equation*}
\qwt(\bw_{10},\emptyset) = 2\alpha_{1}+\alpha_{2}, \quad 
\qwt(\bw_{11}, 2 \mapsto 0)= \theta, \quad
\qwt(\bw_{11}, 2 \mapsto 1)= 2\alpha_{1}+\alpha_{2},
\end{equation*}
\begin{equation*}
\wt(\bw) = \lng^{-1}(\vpi_{1}+\vpi_{2}) \quad 
  \text{for $\bw = \bw_{10},\,\bw_{11}$}, 
\end{equation*}
\begin{equation*}
\deg(\bw_{10},\emptyset) = 3, \qquad 
\deg(\bw_{11}, 2 \mapsto 0)= 1, \qquad
\deg(\bw_{11}, 2 \mapsto 1)= 3. 
\end{equation*}
Therefore, 
\begin{align*}
& \bG(\bw_{10},\emptyset) = q^{3} \OQGL{ t_{2\alpha_1^{\vee}+\alpha_2^{\vee}} }{-\vpi_{1}+2\vpi_{2}}, \\
& \bG(\bw_{11}, 2 \mapsto 0) = q \OQG{ t_{\theta^{\vee}} }, \qquad 
  \bG(\bw_{11}, 2 \mapsto 1) = - q^{3} \OQGL{ t_{2\alpha_1^{\vee}+\alpha_2^{\vee}} }{-\vpi_{1}+2\vpi_{2}}; 
\end{align*}
notice that $\bG(\bw_{10},\emptyset) + \bG(\bw_{11}, 2 \mapsto 1) = 0$. 

Next, let us compute $\bG(\bw,\bb)$ for $\bw = \bw_{9},\,\bw_{12}$. 
Note that $S(\bw_{9})=\{4\}$ and $S(\bw_{12})=\{2\}$. We have 
\begin{equation*}
(-1)^{(\bw_{9}, 4 \mapsto 0)}= 1, \qquad
(-1)^{(\bw_{9}, 4 \mapsto 1)}=-1,
\end{equation*}
\begin{equation*}
(-1)^{(\bw_{12}, 2 \mapsto 0)}= -1, \qquad
(-1)^{(\bw_{12}, 2 \mapsto 1)}= 1,
\end{equation*}
\begin{equation*}
\qwt(\bw_{9}, 4 \mapsto 0)= \theta, \qquad
\qwt(\bw_{9}, 4 \mapsto 1)= 2\alpha_{1}+\alpha_{2},
\end{equation*}
\begin{equation*}
\qwt(\bw_{12}, 2 \mapsto 0)= \theta, \qquad
\qwt(\bw_{12}, 2 \mapsto 1)= 2\alpha_{1}+\alpha_{2},
\end{equation*}
\begin{equation*}
\wt(\bw) = \lng^{-1}(\vpi_{1}+\vpi_{2}) \quad 
  \text{for $\bw=\bw_{9},\,\bw_{12}$}, 
\end{equation*}
\begin{equation*}
\deg(\bw_{9}, 4 \mapsto 0)= 1, \qquad
\deg(\bw_{9}, 4 \mapsto 1)= 3, 
\end{equation*}
\begin{equation*}
\deg(\bw_{12}, 2 \mapsto 0)= 1, \qquad
\deg(\bw_{12}, 2 \mapsto 1)= 3. 
\end{equation*}
Therefore, 
\begin{align*}
& \bG(\bw_{9}, 4 \mapsto 0) = q \OQG{ t_{\theta^{\vee}} }, \qquad 
  \bG(\bw_{9}, 4 \mapsto 1) = - q^{3} \OQGL{ t_{2\alpha_1^{\vee}+\alpha_2^{\vee}} }{-\vpi_{1}+2\vpi_{2}}, \\
& \bG(\bw_{12}, 2 \mapsto 0) = - q \OQG{ t_{\theta^{\vee}} }, \qquad 
  \bG(\bw_{12}, 2 \mapsto 1) = q^{3} \OQGL{ t_{2\alpha_1^{\vee}+\alpha_2^{\vee}} }{-\vpi_{1}+2\vpi_{2}}; 
\end{align*}
notice that 
$\bG(\bw_{9}, 4 \mapsto 0) + \bG(\bw_{12}, 2 \mapsto 0) = 0$ and 
$\bG(\bw_{9}, 4 \mapsto 1) + \bG(\bw_{12}, 2 \mapsto 1) = 0$.

By similar computations, we deduce that 
\begin{align*}
& \bG(\bw_{1}, 2 \mapsto 0, 4 \mapsto 0) = \OQGL{ \lng }{-\vpi_{1}-\vpi_{2}}, \\
& \bG(\bw_{1}, 2 \mapsto 1, 4 \mapsto 0) = -q \OQGL{ \lng t_{\alpha_{1}^{\vee}} }{-2\vpi_{1}+\vpi_{2}}, \\
& \bG(\bw_{1}, 2 \mapsto 0, 4 \mapsto 1) = -q \OQGL{ \lng t_{\alpha_{2}^{\vee}} }{\vpi_{1}-2\vpi_{2}}, \\
& \bG(\bw_{1}, 2 \mapsto 1, 4 \mapsto 1) = q \OQG{ \lng t_{\theta^{\vee}} },
\end{align*}
\begin{align*}
& \bG(\bw_{2},\emptyset) = q^{2} \OQGL{ t_{\theta^{\vee}} }{\vpi_{1}+\vpi_{2}}, &
& \bG(\bw_{3},\emptyset) = - q^{2} \OQGL{ s_{2} t_{\theta^{\vee}} }{\vpi_{1}+\vpi_{2}}, \\
& \bG(\bw_{4},\emptyset) = q^{2} \OQGL{ s_{2}s_{1} t_{\theta^{\vee}} }{\vpi_{1}+\vpi_{2}}, &
& \bG(\bw_{5},\emptyset) = - q^{2} \OQGL{ \lng t_{\theta^{\vee}} }{\vpi_{1}+\vpi_{2}}, \\
& \bG(\bw_{6},\emptyset) = q^{2} \OQGL{ s_{1}s_{2} t_{\theta^{\vee}} }{\vpi_{1}+\vpi_{2}}, &
& \bG(\bw_{7},\emptyset) = q^{2} \OQGL{ s_{1} t_{\theta^{\vee}} }{\vpi_{1}+\vpi_{2}},
\end{align*}
\begin{align*}
& \bG(\bw_{8}, 3 \mapsto 0) = q \OQGL{ s_{1}s_{2}t_{\alpha_{1}^{\vee}} }{-2\vpi_{1}+\vpi_{2}}, & 
& \bG(\bw_{8}, 3 \mapsto 1) = - q \OQG{ s_{1}s_{2}t_{\theta^{\vee}} }, & \\
& \bG(\bw_{13}, 2 \mapsto 0) = q \OQGL{ s_{2}s_{1}t_{\alpha_{2}^{\vee}} }{\vpi_{1}-2\vpi_{2}}, & 
& \bG(\bw_{13}, 2 \mapsto 1) = - q \OQG{ s_{2}s_{1}t_{\theta^{\vee}} }.
\end{align*}
Thus we obtain \eqref{eq22}, as desired. 

%

\end{document}